\documentclass[twoside,11pt,DIV18]{scrartcl}
\usepackage{latexsym,amsmath,amssymb,graphics,graphicx,amscd,amsfonts}
\usepackage[arrow, matrix, curve]{xy}
\usepackage[authoryear,round]{natbib}
\usepackage[thmmarks,amsmath]{ntheorem}
\usepackage{eucal}

\usepackage{hyperref} 

\usepackage{dsfont}
\newcommand{\bp}{{\bar{P}}}

\DeclareMathOperator{\SO}{SO(3)}
\newcommand{\C}{{\mathcal{C}}}

\newcommand{\J}{{\mathcal{J}}}

\newcommand{\mx}{\mathfrak{X}}
\newcommand{\lie}[1]{\mathfrak{#1}}
\newcommand{\dr}{\mathbf{d}}

\newcommand{\ip}[1]{{\mathbf{i}}_{#1}}
\newcommand{\an}[1]{\arrowvert_{#1}}

\newcommand{\poi}[1]{\{#1\}}
\newcommand{\pois}{\poi{\cdot\,,\cdot}}
\newcommand{\D}{{\mathcal{D}}}
\newcommand{\V}{{\mathcal{V}}}
\newcommand{\K}{{\mathcal{K}}}
\newcommand{\T}{{\mathcal{T}}}

\newcommand{\inc}{\iota_{(\rho)}}
\newcommand{\incg}{\iota_{\sigma}}
\newcommand{\inv}{^{-1}}
\newcommand{\R}{\mathbb{R}}
\newcommand{\Z}{\mathbb{Z}}

\newcommand{\sfe}{\mathbb{S}^1}

\DeclareMathOperator{\dom}{Dom}
\DeclareMathOperator{\erz}{span}

\DeclareMathOperator{\Id}{Id}

\DeclareFontFamily{U}{matha}{\hyphenchar\font45}
\DeclareFontShape{U}{matha}{m}{n}{
      <5> <6> <7> <8> <9> <10> gen * matha
      <10.95> matha10 <12> <14.4> <17.28> <20.74> <24.88> matha12
      }{}
\DeclareSymbolFont{matha}{U}{matha}{m}{n}
\DeclareFontSubstitution{U}{matha}{m}{n}
\DeclareMathSymbol{\operp}         {2}{matha}{"6B}

\newtheorem{theorem}{Theorem}[section]
\newtheorem{lemma}[theorem]{Lemma}
\newtheorem{proposition}[theorem]{Proposition}
\newtheorem{corollary}[theorem]{Corollary}
\theorembodyfont{\normalfont}
\newtheorem{definition}[theorem]{Definition}
\theoremsymbol{\ensuremath{\lozenge}}
\newtheorem{example}[theorem]{Example}
\theoremsymbol{\ensuremath{\triangle}}
\newtheorem{remark}[theorem]{Remark}
\theoremstyle{nonumberplain}
\theoremsymbol{\ensuremath{\square}}
\theoremheaderfont{\scshape}
\theoremseparator{:}
\theorembodyfont{\normalfont}
\newtheorem{proof}{Proof}
\begin{document}
\title{ Dirac optimal reduction}
\author{M. Jotz and T. S. Ratiu
}
\date{}
\maketitle

\begin{abstract}
\centerline{\textbf{Abstract}}
The purpose of this paper is to generalize the (Poisson) Optimal Reduction Theorem in 
\cite{OrRa04} to general proper Lie group actions on Dirac manifolds,
 formulated both in terms of point and orbit reduction. 
A comparison  to general  standard singular Dirac reduction is given emphasizing the 
desingularization role played by optimal reduction. 
\end{abstract}

\noindent \textbf{AMS Classification:} Primary subjects: 70H45, 70G65, 

\hspace{3.5cm} Secondary subjects: 70G45, 53D17, 53D99

\noindent \textbf{Keywords:} Dirac structures, singular reduction,
proper action, optimal momentum map.

\tableofcontents

\section{Introduction}

Dirac structures, introduced in \cite{CoWe88} and systematically investigated for the first time in \cite{Courant90a}, have a wide range of applications in geometry and theoretical physics. They include 2-forms, Poisson structures, foliations and  also provide a
convenient geometric setting for the theory of nonholonomic systems and
circuit theory. The study of sub-objects and quotient objects in the Dirac category, central in the theory of reduction, is of particular interest.

Recent years have seen a significant development of Dirac structures both 
from the geometric point of view as well as in applications. In the presence 
of symmetry, one can perform reduction to eliminate variables; see 
\cite{Courant90a}, \cite{Blankenstein00}, \cite{BlvdS01}, \cite{BuCaGu07}, 
\cite{JoRa08}, \cite{JoRaZa11}, \cite{MaYo07}, \cite{MaYo09}, for the regular 
case and \cite{BlRa04}, \cite{JoRaSn11}, \cite{JoRa11}
 for the singular situation. All these 
reduction procedures are in the spirit of Poisson reduction (\cite{MaRa86}, 
\cite{Sniatycki03}, \cite{FeOrRa09}, \cite{JoRa09}).

In symplectic reduction, a central role is played by the momentum map. 
While its existence in the symplectic category is not guaranteed, in 
concrete applications it is rarely the case that a symplectic Lie group action 
fails to admit a momentum map. The situation is, however, drastically 
different in the Poisson category, where the existence of the momentum 
map imposes unreasonable constraints on the symmetries. Because of 
that, \cite{OrRa04} introduced the optimal momentum map, a conservation 
law of a Poisson symmetry, that is always defined and has values in a 
topological space. A reduction method based on the optimal 
momentum map was  proved.

In the present paper we generalize the optimal momentum map and the 
optimal reduction procedure
in \cite{OrRa04} to closed Dirac manifolds\footnote{Historically, ``closed" and ``integrable" Dirac manifolds were synonyms. However, due to the relation of Dirac geometry to groupoids, an ``integrable" Dirac structure is one that, viewed as an algebroid, integrates to a groupoid. This is the reason why in this paper we use exclusively the term ``closed".}.
As we shall see, with necessary assumptions and appropriately 
extended definitions, this important desingularization 
method works also for  Dirac manifolds. The power of optimal 
reduction can be immediately  seen, already
for free actions (\ref{section_free_case}), by noting that 
it has as trivial corollary the stratification in 
presymplectic leaves of a closed Dirac manifold.

Singular optimal reduction is carried out in two steps.
 First, one 
considers  appropriately chosen distributions jointly
defined by the symmetry group and the Dirac structure.
In the free case, this reduces to one distribution that is automatically integrable in the sense
of Stefan and Sussmann if the hypotheses for regular/standard Dirac reduction are satisfied.
 In the nonfree case, this is to no longer true, in general.
However, if these generalized distributions are integrable, their leaves define the level sets of corresponding 
natural optimal momentum maps. 

Second, one passes to the quotient and 
constructs on it the reduced Dirac structure. It is not possible to extend this 
result in a naive manner to non-closed Dirac structures because the first 
consequence of non-closedness is the non-integrability of 
the projection of the Dirac structure on its tangent part, and hence, in general,
the distribution 
used in the previously described reduction process is also nonintegrable. 

\paragraph{Outline of the paper.} In Section \ref{dis} we briefly review the 
relevant definitions and properties of generalized distributions; special attention is given to   
integrability conditions for tangent distributions. Section \ref{sec:Dirac_structures} recalls the 
general theory of Dirac structures and Dirac actions of Lie groups and Lie 
algebras. Closed Dirac structures are emphasized and it is shown that 
flows of Hamiltonian vector fields defined by admissible functions leave the Dirac structure invariant. 
As a corollary, it is proved that the intersection of the Dirac structure with the tangent bundle is an 
integrable generalized distribution. The necessary material from the  
theory of proper Lie group actions is reviewed in Section \ref{sec:proper}. 
An averaging procedure that is an important technical tool in the 
proof of many statements is discussed and 
the concept of descending sections of the Pontryagin bundle is recalled. 

In Section \ref{section_free_case}, we summarize the results obtained later on in the particular case of a free 
Lie group action. No proofs are provided since they will be 
done for the general case in subsequent sections. Nevertheless,
we believe that this section is helpful, since it illustrates 
the much more complicated general theory in a simple case; 
the distribution appearing here is the image of an appropriately chosen Lie algebroid by its anchor map.

Section \ref{sec:optimal_distributions} is devoted to the study of two 
special distributions that are crucial in the reduction procedure; they are 
called optimal distributions and are associated to orbit and isotropy types, 
respectively. 
Under the necessary conditions for standard singular reduction, these two optimal distributions are 
algebraically involutive and a hypothesis is given so that the one associated to isotropy types is
integrable.

In Section \ref{sec:optimal_momentum} we introduce and study two optimal 
momentum maps, objects always available for  canonical Lie 
group actions on closed Dirac manifolds if the optimal distributions are spanned by their descending sections and are integrable. If the Dirac structure comes from a Poisson manifold, these conditions are always satisfied and the two optimal momentum maps coincide.  Optimal reduction is presented in Section \ref{sec:optimal_reduction}. Two point optimal reduction theorems associated to the two optimal momentum maps are proved. In addition, an optimal orbit reduction theorem is presented and it is shown that the three reduction procedures, i.e., 
the two optimal point reduction and the optimal orbit reduction, are  equivalent. 

In Section \ref{sec:comparison} we show that standard and optimal Dirac reduction are equivalent under appropriate hypotheses. Section 
\ref{sec:examples} is devoted to the study of several examples illustrating the theory. An appendix summarizes the technical results used to compute
the symmetry invariant generators for the set of invariant vector fields used 
in the examples.

\paragraph{Notations and conventions.}
We will write $C^\infty(M)$ for the sheaf of local functions on $M$. That is,
an element $f \in C^\infty(M)$ is a smooth function $f:U\to \R$, with $U$ an
open subset of $M$.
In the same manner, if $E$ is a vector bundle over $M$, or a generalized
distribution on $M$, we will denote by
$\Gamma(E)$ the set of local sections of $E$. In particular, the sets of local
vector fields and one-forms on $M$ will be denoted by $\mx(M)$ and
$\Omega^1(M)$, respectively. We will write $\dom(\sigma)$ for the open  
domain of definition  of $\sigma\in \Gamma(E)$.

\medskip

The Lie group $G$ is always assumed to be connected;
$\mathfrak{g}$ denotes its Lie algebra. All $G$-actions on $M$
are smooth and are denoted by $\Phi: G\times M\to M$, 
$(g, m) \mapsto gm = g \cdot m=\Phi_g(m)$, for all $g \in G$ and 
$m \in M$. If $\xi \in \mathfrak{g}$, then $\xi_M \in 
\mathfrak{X} (M) $ defined by $\xi_M(m) := 
\left.\frac{d}{dt}\right|_{t=0} \exp (t \xi) \cdot m $ is 
called the \emph{infinitesimal generator} or 
\emph{fundamental vector field} defined by $\xi$.

A section $X$ of $TM$ (respectively $\alpha$ of $T^*M$) is 
called \emph{$G$-invariant} if $\Phi_g^*X=X$ (respectively
$\Phi_g^*\alpha=\alpha$) for all $g\in G$. Here, the vector field $\Phi_g^\ast X$  
is defined by $\Phi_g^\ast X=T\Phi_{g^{-1}}\circ X\circ\Phi_g$, that is,
$(\Phi_g^\ast X)(m)=T_{gm}\Phi_{g^{-1}}X(gm)$ for all $m\in M$.

\medskip

Recall that a subset $ N \subset  M $ is an \emph{initial} submanifold 
of $ M $ if  $ N $ carries a manifold structure such that the inclusion 
$ \iota :N \hookrightarrow M$ is a smooth immersion and satisfies
the following condition: for any smooth manifold $ P $  an arbitrary
map $ g: P \rightarrow  N$ is smooth if and only if $\iota \circ  g: P 
\rightarrow  M$ is smooth; in this case, $\iota$ is said to be a \textit{regular immersion}. The notion of initial submanifold lies strictly
between those of injectively immersed and embedded submanifolds.

\section{Generalized distributions, symmetries, and standard Dirac reduction}

\subsection{Generalized distributions}\label{dis}
The \emph{Pontryagin} bundle $\mathsf P_M$ of a smooth manifold 
$M$ is the direct sum $\mathsf{P}_M=TM\oplus T^*M$.
A  \emph{generalized distribution} $\Delta$ on $M$ is a subset
$\Delta$ of $\mathsf{P}_M$  such that for each $m\in M$, the set 
$\Delta(m) :  = \Delta \cap \mathsf{P}_M(m)$ is a
vector subspace of $\mathsf{P}_M(m)= T_mM\times T_m^*M$. The 
\emph{rank}  of $\Delta$ at $ m  \in  M $ is $\dim \Delta (m)$. 
A point $m\in M$ is a \emph{regular} point of the distribution 
$\Delta$ if there exists a neighborhood $U$ of $m$ such that 
the rank of $\Delta$ is constant on $U$. Otherwise, $m$ is 
a \emph{singular} point of the distribution. 

A local \emph{differentiable  section} of $\Delta$ is a smooth 
section $\sigma\in \Gamma(\mathsf P_M)$ defined on some open 
subset $U \subset  M$ such that $\sigma(u) \in  \Delta (u)$ 
for each $u \in  U$; the open domain of definition of $\sigma$
is denoted by $\operatorname{Dom}(\sigma)$. Let $\Gamma(\Delta)$ 
be the space of differentiable local sections of $\Delta$. 
A generalized
distribution is said to be \emph{differentiable} or 
\emph{smooth} if for every point $ m \in  M $ and every 
vector $ v \in  \Delta (m) $, there is a differentiable 
section $ \sigma \in  \Gamma(\Delta)$ defined on an open
neighborhood $ U $ of $ m $ such that $ \sigma(m) = v $. 
A subset $S\subseteq\Gamma(\mathsf P_M)$ is said to 
\textit{span} the smooth generalized distribution $\Delta$ 
if $\Gamma(\Delta)=\erz_{C^\infty(M)}(S)$; $S$
spans \emph{pointwise}
$\Delta$ if for all $m\in M$, the values of 
the elements of $S$ at $m$ span $\Delta(m)$.

\medskip 

A smooth generalized distribution contained in $TM$  is
called a \emph{smooth tangent distribution}; a smooth 
generalized distribution contained in $T^*M$ is
called a \emph{smooth cotangent distribution}.

\subsubsection{Smooth orthogonals and annihilators}
The Pontryagin bundle $\mathsf{P}_M=TM \oplus T^* M$ of  a 
smooth manifold $M$ is endowed with a non-degenerate symmetric 
fiberwise bilinear form of signature $(\dim M, \dim M)$ given by 
\begin{equation}\label{pairing}
\left\langle (u_m, \alpha_m), ( v_m, \beta_m ) \right\rangle : 
= \left\langle\beta_m , u_ m \right\rangle + \left\langle
\alpha_m, v _m \right\rangle
\end{equation}
for all $u_m, v_m\in T_mM$ and $\alpha_m,\beta_m \in T^\ast_mM$. 

If $\Delta \subset \mathsf P_M$ is a smooth generalized 
distribution, its \emph{smooth} \emph{orthogonal}  is the smooth
generalized distribution $\Delta^\perp \subseteq \mathsf P_M$  
defined by 
\begin{align*}
\Delta^ \perp (m): =
 \left\{\tau(m) \left|  
\begin{array}{c}\tau \in
\Gamma(\mathsf P_M) \text{ with } m\in \dom(\tau) \text{is such 
that  for all }\\
\sigma \in\Gamma(\Delta) \text{  with } m\in\dom(\sigma),\\
\text{ we have }  \left\langle \sigma,\tau
\right\rangle = 0 \text{ on } \dom(\tau)\cap\dom(\sigma)\end
{array}\right.\right\}.
\end{align*}
Note that the smooth orthogonal of a 
smooth generalized distribution is smooth, by construction. The 
inclusion $\Delta \subset \Delta^{ \perp\perp }$ is, in general, 
strict. If the distribution $\Delta$ is a vector subbundle of 
$\mathsf P_M$, then its smooth orthogonal is
also a vector subbundle of $\mathsf P_M$. 

\medskip

Let $\T\subseteq TM$ be a tangent
distribution. The \emph{smooth annihilator} of
$\T$
is the smooth codistribution $\T^\circ \subseteq T ^\ast M$ defined by
\[\T^\circ(m):=\left\{\alpha_m\left|\begin{array}{c}
 \alpha\in\Omega^1(M), m\in\dom(\alpha) \text{
  and } 
\alpha(X)=0\\ \text{ on }\dom(\alpha)\cap\dom(X)\text{ for all }
X\in\Gamma(\T)
\end{array}
\right.\right\}
\]
for all $m\in M$.  Analogously, we define the smooth annihilator
$\C^\circ \subseteq TM$ of a
codistribution $\C \subseteq T ^\ast M $.

\medskip
  
The tangent distribution $\mathcal{V}$ spanned by the 
fundamental vector fields of the action of a Lie group $G$ 
on a manifold $M$ is defined at every point $m \in M $ by
$\mathcal{V}(m):=\{\xi_M(m)\mid \xi\in\mathfrak{g}\}$.
If the action is not free, the rank of the fibers of $\V$
can vary on $M$. The smooth annihilator  $\mathcal{V}^\circ$ 
of $\mathcal{V}$ has the expression
\[
\mathcal{V}^\circ(m)=\{\alpha(m)\mid \alpha\in\Omega^1(M), \,
m\in\dom(\alpha),\text{ such that } \alpha(\xi_M)=0 
\text{ for all }\xi \in \mathfrak{g}\}.
\]
We will also use the smooth generalized distribution 
${\mathcal K}:=\mathcal{V}\oplus
\{0\}\subseteq \mathsf P_M$ 
and its smooth orthogonal ${\mathcal K}^\perp=
TM\oplus\mathcal{V}^\circ$.

\subsubsection{Generalized foliations and integrability of 
tangent distributions}\label{int_sing_dis}

To give content to the notion of integrability of a smooth 
tangent distribution and elaborate on it, we need to quickly 
review the concept and main properties of generalized 
foliations (see \cite{Stefan74a, Stefan74b, Stefan80}, 
\cite{Sussmann73} for the original articles and 
\cite{LiMa87}, \cite{Vaisman94}, \cite{Pflaum01}, or 
\cite{OrRa04},  for a quick review of this theory). 
\medskip

A \emph{generalized foliation} on $ M $ is a partition 
$\mathfrak{F} : = \{ \mathcal{L} _\alpha \}_{ \alpha \in  A}$ 
of $ M $  into disjoint connected sets, called \emph{leaves},  
such that each point $ m \in  M $ has a
\emph{generalized foliated chart} 
$ (U, \varphi : U \rightarrow  V \subseteq
\mathds{R} ^{\dim M}) $, $ m \in  U $. This means that  
there is some natural number $ p_\alpha \leq \dim M $, 
called the \emph{dimension} of the leaf $\mathcal{L} _\alpha$, 
and a subset $S_\alpha \subset \mathds{R}^{\dim M - p_\alpha}$ 
such that $ \varphi(U \cap \mathcal{L} _\alpha ) = 
\{(x^1, \ldots ,x^{\dim M} ) \in  V \mid  
(x^{p_\alpha+1}, \ldots  ,x^{\dim M} ) \in  S _\alpha \}$. 
Note that each $ (x^{p_\alpha+1}_ \circ  , \ldots  ,x^{\dim M}_ \circ  ) \in
S _\alpha$ 
determines a connected component $ ( U \cap \mathcal{L}_\alpha )_ \circ $ of
 $U \cap \mathcal{L}_\alpha$, 
that is, $ \varphi (( U \cap \mathcal{L}_\alpha )_ \circ ) 
= \{(x^1, \ldots , x^{p_\alpha}, x^{p_\alpha+1}_ \circ  , \ldots  ,x^{\dim M}_
\circ  ) \in V \} $. The key difference with the concept of foliation is that the 
number $ p_\alpha $ can change from leaf to leaf.  
The generalized foliated charts induce on each leaf a smooth manifold
structure that 
makes them into initial submanifolds of $ M $.

A leaf $ \mathcal{L} _\alpha $ is called \emph{regular} if it has an open
neighborhood 
that intersects only leaves whose dimension equals $ \dim \mathcal{L} _\alpha
$. 
If such a neighborhood does not exist, then $ \mathcal{L} _\alpha $ 
is called a \emph{singular} leaf. A point is called \emph{regular}
(\emph{singular}) 
if it is contained in a regular (singular) leaf. The set of vectors tangent to
the 
leaves of  $ \mathfrak{F} $ is defined by
\[
T(M,  \mathfrak{F}): 
= \bigcup _{ \alpha \in  A} 
\bigcup_{m \in  \mathcal{L} _\alpha} T _m \mathcal{L} _\alpha \subset  TM.
\]

\medskip
Let us turn now to the relationship between distributions and generalized
foliations.
In all that follows, $ \T $ is a smooth tangent distribution. An
\emph{integral manifold} of $ \T $ is an injectively immersed connected
manifold $ \iota _L : L \hookrightarrow M $,  where $ \iota _L $  is the
inclusion, satisfying the condition $ T _m \iota_L (T _m  L )  \subset  \T (m)
$  for every $ m \in  L $. The integral manifold $ L $ is of \emph{maximal
  dimension} at $ m \in  L $ if $ T _m \iota_L (T _m  L )  =  \T (m) $. The
distribution $ \T $ is \emph{completely integrable} if for every $m \in  M $
there is an integral manifold $ L $ of $ \T $, $ m \in  L $, everywhere of
maximal dimension. The distribution $\T$ is \emph{involutive} if it is
invariant under the (local) flows associated to differentiable sections of
$\T$. The distribution $\T$ is \emph{algebraically involutive} if for any two
smooth vector fields defined on an open set of $M $ which take values in $\T$,
their bracket also takes values in $\T$. Clearly, involutive distributions are
algebraically involutive and the converse is true if the distribution is a vector
subbundle. 

Recall that the Frobenius theorem states that a vector subbundle of $ TM $ is algebraically 
involutive if and only if it is the tangent bundle of a
foliation on $M $.

The same is true for distributions under the involutivity asumption: \textit{A smooth distribution is
  involutive 
if and only if it coincides with the set of vectors tangent to a generalized
foliation, 
that is, it is completely integrable.}  This is known as the Stefan-Sussmann Theorem.

We will formulate the Stefan-Sussmann theorem in the setting of 
a smooth tangent distribution spanned by a family of vector 
fields. Note that each smooth tangent distribution is spanned 
by the family of its smooth sections.

\medskip

Let  $\mathcal F$  be an everywhere defined family of local vector fields on $M$.
By {\it everywhere defined} we mean that for every $m \in M$ there
exists $X \in \mathcal F$ such that $m \in {\rm Dom}(X) $.  
Associate to the flows of the vector fields in $F$ the set of local
diffeomorphisms $\mathcal{A}_{\mathcal F}:=\{\phi _t\mid \,\phi _t \text{ flow of } X \in
\mathcal F\} $ of $M$ and the pseudogroup of transformations generated by it,
\[
A _{\mathcal F}:=(\mathds I, M)\bigcup\{\phi _{t _1}^1 \circ \cdots\circ \phi _{t
_n}^n\mid n \in \mathds
N\text{ and  } \phi _{t _n}^n \in \mathcal{A}_{\mathcal F}\text{ or } (\phi
_{t _n}^n) ^{-1} \in \mathcal{A}_{\mathcal F}\}.
\]
Analogously, we also define, for any $z  \in M $, the following vector
subspaces of $T_zM $:
\begin{align*}
\mathcal{D}_{\mathcal F} (z)&:=\erz_\R\left\{ \left. \frac{d}{dt}\right\arrowvert_{t=t_0}
\!\!\!\phi _t (y)
\,\Bigr|\, \phi _t\text{ flow of }X
\in \mathcal F,
\, \phi _{t_0} (y)=z\right\}\\
&=\erz_\R\{X(z)\in T_zM| \,X\in \mathcal F \text{ and }
z \in {\rm Dom}(X)\},\\
D _{\mathcal F} (z) &:=\erz_\R\{T _y \phi_T \left(
\mathcal{D}_{\mathcal F}(y)\right)\mid \phi_T\in A _{\mathcal F},
\phi_T(y)=z\}.
\end{align*}
Note that, by construction, $\mathcal{D}_{\mathcal F}$ is a smooth tangent distribution;   
$\mathcal{D}_{\mathcal F}$ is said to be the smooth tangent distribution \emph{spanned} by $\mathcal F$. 

The $A _{\mathcal F} $-orbits, also called the \emph{accessible sets}
of the family ${\mathcal F}$,  form a generalized foliation  whose leaves
have as tangent spaces the values of $D _{\mathcal F}$ (see, for example, \cite{OrRa04}).  
An important question is determining when
the smooth tangent distribution $\mathcal{D}_{\mathcal F} $ spanned by ${\mathcal F}$ is
integrable. 

\begin{theorem}
\label{frobenius with vector fields}
{\rm(\cite{Stefan74a} and~\cite{Sussmann73})}.
Let $\mathcal{D}_{\mathcal F}$ be a differentiable generalized distribution on the smooth
manifold $M$ spanned pointwise by an everywhere defined
family of vector fields ${\mathcal F}$. The following properties are
equivalent:
\begin{enumerate}
\item The distribution $\mathcal{D}_{\mathcal F}$ is invariant  under the
pseudogroup of transformations generated by ${\mathcal F}$, that is, for each
$\phi_T\in
A_{{\mathcal F}}$ and for each
$z\in M$ in the domain of $\phi_T$,
\[
T_z\phi_T(\mathcal{D}_{\mathcal F} (z))=\mathcal{D}_{\mathcal F}(\phi_T(z)).
\]
\item  $\mathcal{D}_{\mathcal F}= D _{\mathcal F} $.
\item For any $X \in {\mathcal F} $ with flow $\phi $ and any $x\in {\rm
Dom}(X)$, there
exist:
\begin{enumerate}
\item A finite set $\{X _1, \ldots, X _p\}\!\subset\! {\mathcal F}$ such
that $\mathcal{D}_{\mathcal F}
(x)\!=\! \erz_\R\{X _1(x), \ldots, X _p(x)\} $.
\item A constant $\epsilon>0 $ and Lebesgue integrable functions
$\lambda_{ij}:(-
\epsilon,
\epsilon) \rightarrow \mathds R$ {\rm (}$1 \leq i,j \leq p${\rm )} such that for
every $t \in (-
\epsilon, \epsilon)$ and $j \in \{1, \ldots, p\} $:
\[
[X, X _j](\phi _t(x))=\sum _{i=1}^{p}\lambda_{ij}(t) X _i(\phi _t(x))
\]
and $\mathcal{D}_{\mathcal F}
(\phi _t(x))= \erz_\R\{X _1(\phi _t(x)), \ldots, X _p(\phi _t(x))\} $.
\end{enumerate}
\item The distribution $\mathcal{D}_{\mathcal F}$ is integrable and its
maximal integral
manifolds are  the $A_{{\mathcal F}}$-orbits.
\end{enumerate}
\end{theorem}

As already mentioned, given an involutive (and hence a completely
integrable) distribution $ \T $, each point $ m \in  M $ belongs to
exactly one connected integral manifold $ \mathcal{L} _m $ that is maximal
relative to inclusion. It turns out that $ \mathcal{L} _m $ is an initial
submanifold and that it is also the \emph{accessible} set of $ m $, that is, $
\mathcal{L} _m $ equals the subset of points in $ M $ that can be reached by
applying to $ m $ a finite number of composition of flows of elements of
$\Gamma (\T)$.  The collection of all maximal integral submanifolds of $
\T $ forms a generalized foliation $ \mathfrak{F} _\T $ such that $
\T = T(M, \mathfrak{F} _\T) $.  Conversely, given a generalized
foliation $ \mathfrak{F} $ on $ M $, the subset $ T(M, \mathfrak{F}) \subset
TM $ is a smooth completely integrable (and hence involutive) distribution
whose collection of maximal integral submanifolds coincides with $
\mathfrak{F} $. These two statements expand the Stefan-Sussmann Theorem cited
above.

\subsection{Generalities on Dirac structures}
\label{sec:Dirac_structures}

\subsubsection{Dirac structures} \label{admsble}
Recall that  the Pontryagin bundle $\mathsf{P}_M=TM \oplus T^* M$ of  a 
smooth manifold $M$ is endowed with a non-degenerate symmetric 
fiberwise bilinear form of signature $(\dim M, \dim M)$ given by 
\eqref{pairing}. A \emph{Dirac structure}  (\cite{CoWe88}, \cite{Courant90a}) 
on $M $ is a Lagrangian subbundle
$\mathsf{D} \subset\mathsf P_M $. That is, 
$ \mathsf{D}$ coincides with its orthogonal relative to \eqref{pairing} and so its
fibers are necessarily 
$\dim M $-dimensional.

The space $\Gamma(TM \oplus T ^\ast M) $ of smooth local sections of the Pontryagin
bundle is also endowed with a 
$\mathds{R}$-bilinear skew-symmetric bracket (which does not 
satisfy the Jacobi identity) given by
\begin{align}\label{wrong_bracket}
[(X, \alpha), (Y, \beta) ] : &
= \left( [X, Y],  \boldsymbol{\pounds}_{X} \beta - \boldsymbol{\pounds}_{Y} \alpha + \frac{1}{2}
  \mathbf{d}\left(\alpha(Y) 
- \beta(X) \right) \right) \nonumber \\
&= \left([X, Y],  \boldsymbol{\pounds}_{X} \beta - \mathbf{i}_Y \mathbf{d}\alpha 
- \frac{1}{2} \mathbf{d} \left\langle (X, \alpha), (Y, \beta) \right\rangle
\right)
\end{align}
(see \cite{Courant90a}). The Dirac structure is 
\emph{closed} (or \emph{integrable}) if $[ \Gamma(\mathsf{D}), \Gamma(\mathsf{D}) ] 
\subset \Gamma(\mathsf{D}) $. Since \linebreak
$\left\langle (X, \alpha), (Y, \beta) \right\rangle = 0 $ if $(X, \alpha), (Y,
\beta) \in \Gamma(\mathsf{D})$, 
integrability of the Dirac structure is often expressed in the literature
relative to a non-skew-symmetric
 bracket that differs from 
\eqref{wrong_bracket} by eliminating in the second line the third term of
the second component. This truncated expression which satisfies the Jacobi
identity but is no longer skew-symmetric is called the \emph{Courant-Dorfman bracket}:
\begin{equation}\label{Courant_bracket}
[(X, \alpha), (Y, \beta) ] : = \left( [X, Y],  \boldsymbol{\pounds}_{X} \beta - \ip{Y} \dr\alpha \right).
\end{equation}

A Dirac structure $\mathsf{D}$ on  a manifold $M$ defines two smooth tangent distributions 
$\mathsf{G_0}, \mathsf{G_1} \subset TM $ and two smooth cotangent distributions 
$\mathsf{P_0}, \mathsf{P_1} \subset T^*M$; their fibers 
at $m \in M$ are:
\begin{align*}
\mathsf{G_0}(m)&:= \{X(m) \in T_mM \mid X \in \mathfrak{X}(M), (X, 0) \in
\Gamma(\mathsf{D}) \} \\
\mathsf{G_1}(m)&:= \{X(m) \in T_mM \mid X \in \mathfrak{X}(M), \, \text{there exists}\;
\alpha \in \Omega^1(M), \text{such that}\; 
(X, \alpha) \in \Gamma(\mathsf{D})\}
\end{align*}
and
\begin{align*}
\mathsf{P_0}(m)&:= \{\alpha(m) \in T^*_mM \mid \alpha \in \Omega^1(M), 
(0, \alpha) \in \Gamma(\mathsf{D}) \} \\
\mathsf{P_1}(m)&:= \{\alpha(m) \in T^*_mM \mid  \alpha \in \Omega^1(M),\, \text{there
exists }\;
X \in\mathfrak{X}(M), \text{such that}\; (X, \alpha) \in \Gamma(\mathsf{D}) \}.
\end{align*}
The smoothness of $\mathsf{G_0}, \mathsf{G_1}, \mathsf{P_0}, \mathsf{P_1}$ is
obvious since, by definition, 
they are generated by smooth local sections. In general, these are not vector
subbundles of $TM$ and 
$T ^\ast M $, respectively. It is also clear that $\mathsf{G_0} \subset
\mathsf{G_1}$ and 
$\mathsf{P_0} \subset \mathsf{P_1}$. 

We have the equalities 
\[\mathsf{P_0}=\mathsf{G_1}^\circ, \quad \mathsf{G_0}=\mathsf{P_1}^\circ
\]
and the inclusions
\[\mathsf{P_1}\subseteq\mathsf{G_0}^\circ, 
\quad \mathsf{G_1}\subseteq\mathsf{P_0}^\circ.
\]
If $\mathsf{P_1}$ (respectively $\mathsf{G_1}$) has constant rank on $M$,
the first inclusion above (respectively the second) is an  equality.

\medskip

A function $f\in C^\infty(M)$ is called $\mathsf{D}$-\emph{admissible}, or
simply \emph{admissible} if there is no possibility of confusion, if 
there exists a vector field $X\in\mx(M)$ such that  $(X,\dr f)\in 
\Gamma(\mathsf{D})$. The section $(X,\dr f)$ is then called ($\mathsf{D}$-)
\emph{admissible} or ($\mathsf{D}$-)\emph{Hamiltonian}
and  $X=:X_f$ is a ($\mathsf{D}$-)\emph{Hamiltonian} vector field for $f$.

Note that the vector field $X_f$ is not unique if $\mathsf{G_0}\neq\{0\}$;
if $X_f$ is a Hamiltonian vector  field  for $f$, then, for any section $Z$ of $
\mathsf{G_0}$, the sum  $X_f+Z$ is  also a Hamiltonian vector field for $f$. 
Indeed, since $(X_f,\dr f), (Z,0)\in\Gamma(\mathsf{D})$, the sum $(X_f+Z,\dr 
f)=(X_f,\dr f)+(Z,0)$ is also a section of $\mathsf{D}$. The smooth tangent 
distribution $\mathsf{G_0}$ is spanned by the  Hamiltonian vector fields of 
the constant functions, which are consequently all admissible.

Define a   bracket $\pois_{\mathsf D}$  on the  set $C^\infty(M)^{\mathsf D}$ 
of admissible  functions by $\{f,g\}_{\mathsf D}:=X_f(g)=-X_g(f)$
for all $f,g\in C^\infty(M)^{\mathsf D}$. This bracket
does not depend on the choices made for $X_g$ and $X_f$, and if 
the Dirac manifold $(M,\mathsf D)$ is integrable,
$\pois_{\mathsf D}:C^\infty(M)^{\mathsf D}\times C^\infty(M)^{\mathsf D}\to
C^\infty(M)^{\mathsf D}$ is a Poisson bracket on $C^\infty(M)^{\mathsf D}$.

\subsubsection{Properties of  integrable Dirac structures}
Assume that $(M, \mathsf D) $ is an integrable  Dirac manifold. Then, 
relative to the Courant bracket \eqref{Courant_bracket} 
and the anchor $\pi_{TM}:\mathsf D\to TM$ given by the projection on the first 
factor,  $\mathsf D $ becomes a Lie algebroid over $ M $.
The smooth distribution $ \mathsf{G_1}=\pi _{TM}( \mathsf D ) \subset  TM $
is then completely integrable in the sense of Stefan and Sussmann and each
leaf $N$ of $\mathsf{G_1}$ inherits a presymplectic form 
$\omega_N$ given by
\begin{equation}\label{induceddiracN}
\omega_N(\tilde{X},\tilde{Y})(p)=\alpha(Y)(p)=-\beta(X)(p)
\end{equation}
for all $p\in N$ and $\tilde{X},\tilde{Y}\in\mathfrak{X}(N)$, 
where $X,Y\in \Gamma(\mathsf{G}_1)$ are arbitrary sections $\iota_N$-related
to  $\tilde{X},\tilde{Y}$, respectively,  and $\iota_N:N\hookrightarrow M$ is 
the inclusion; $\iota_N$-relatedness is denoted by $\tilde{X} \sim_{ \iota_N}
X$, $\tilde{Y} \sim_{ \iota_N} Y$.
The one-forms $\alpha, \beta \in \Omega^1(M)$ are such that 
$(X,\alpha), (Y,\beta)\in \Gamma(\mathsf D)$. 
Formula \eqref{induceddiracN} is independent of all the choices involved. 
Note that there is an induced Dirac
structure on $N$ given by the graph of the bundle map $\flat: TN \rightarrow
T ^\ast N$  associated to $ \omega_N$. The proofs of these facts can be 
found in \cite{Courant90a}. Since $\mathsf{G_1}$ has constant rank on $N$, 
the codistribution $\mathsf{P_0}$ has also constant rank on $N$ and $
\mathsf D\cap(TN\oplus T^*M\an{N})$  is a smooth vector bundle over $N$. 
Then the induced Dirac structure $\mathsf{D}_N$ on $N$ can be described 
in the following way:
\begin{align}\label{inducedDirac}
\Gamma(\mathsf{D}_N):=\left\{(\tilde X,\tilde \alpha)\mid \exists
  (X,\alpha)\in\Gamma(\mathsf D)\quad \text{ such that }\quad
\tilde X\sim_{\iota_N}X \quad\text{ and }\quad \tilde \alpha=\iota_N^*\alpha
\right\}
\end{align}
(see \cite{Courant90a} or \cite{JoRa08}).
\medskip

Let $(M,\mathsf{D})$ be a closed Dirac manifold and $f\in C^
\infty(M)^{\mathsf D}$  a smooth admissible function. Let $\phi$ be the flow
of a corresponding Hamiltonian vector field $X_f$ and $(N,\omega_N)$ 
a presymplectic leaf of $(M,\mathsf{D})$ intersecting the domain of 
definition of $(X_f,\dr f)$. Since $X_f$ is a section of 
$\Gamma(\mathsf{G}_1)$, there exists a vector field $\tilde X_f\in\mx(N)$
that is $\iota_N$-related to $X_f$. The flow $\tilde \phi$ of $\tilde X_f$ 
satisfies  $\iota_N\circ\tilde \phi_t=\phi_t\circ\iota_N$
for all $t$ where $\phi_t$ is defined.
We have $\ip{\tilde X_f}\omega_N=\iota_N^*\dr f=\dr (\iota_N^*f)$ by definition
of $\omega_N$. This yields 
\[\boldsymbol{\pounds}_{\tilde X_f}\omega_N=\ip{\tilde X_f}\dr \omega_N+\dr (\ip{\tilde
    X_f}\omega_N)
=0+\dr^2(\iota_N^*f)=0,\]
and hence
\[\left.\frac{d}{dt}\right\an{t=t_0}\tilde\phi_t^*\omega_N
=\tilde\phi_{t_0}^*(\boldsymbol{\pounds}_{\tilde X_f}\omega_N)=0.
\]
Thus $\tilde\phi_t^*\omega_N=\tilde\phi_0^*\omega_N=\omega_N$ for all 
$t$ for which $\tilde\phi_t$ is defined. 

We want to show that the flow 
$\phi_t$ preserves the Dirac structure, that is, $(\phi_t^*X,\phi_t^*\alpha)\in
\Gamma(\mathsf{D})$ for all $(X,\alpha)\in\Gamma(\mathsf{D})$. To see
 this, we choose a point $m\in \dom(X,\alpha)\cap\dom(X_f,\dr f)\subseteq 
 M$ and a section $(Y,\beta)$ of $\mathsf{D}$ defined on a neighborhood 
 of $m$. Let $(N,\omega_N)$ be the presymplectic leaf of $(M,\mathsf{D})$ 
 through the point $m$. Then we can write $m=\iota_N(n)$ for some point 
 $n$ in the initial immersed submanifold  $N$ of $M$. Using the definition 
 of $\omega_N$,  we compute on the common domain of definition of 
 $(X_f,\dr f)$, $(X,\alpha)$, and $(Y,\beta)$:
\begin{align*}
\langle(\phi_t^*X,\phi_t^*\alpha),(Y,\beta)\rangle(m)
&=(\phi_t^*\alpha)_{\iota_N(n)}\bigl(Y(\iota_N(n))\bigr)
+\beta_{\iota_N(n)}\bigl((\phi_t^*X)(\iota_N(n))\bigr)\\
&=(\iota_N^*\alpha)_{\tilde\phi_t(n)}\bigr((\tilde\phi_{-t}^*\tilde Y)
(\tilde\phi_t(n))\bigl)+
(\iota_N^*\beta)_{n}\bigl((\tilde\phi_t^*\tilde X)(n)\bigr)\\
&=(\ip{\tilde X}\omega_N)_{\tilde\phi_t(n)}
\bigr((\tilde\phi_{-t}^*\tilde Y)(\tilde\phi_t(n))\bigl)+
(\ip{\tilde Y}\omega_N)_{n}\bigl((\tilde\phi_t^*\tilde X)(n)\bigr)\\
&={\omega_N}(\tilde\phi_t(n))\bigl(\tilde
X(\tilde\phi_t(n)),(\tilde\phi_{-t}^*\tilde Y)(\tilde\phi_t(n)) \bigr)
+{\omega_N}(n)\bigl(\tilde Y(n),(\tilde\phi_t^*\tilde X)(n)\bigr)\\
&=(\phi_t^*\omega_N)(n)\bigl((\tilde\phi_t^*\tilde X)(n),\tilde Y(n) \bigr)
+{\omega_N}(n)\bigl(\tilde Y(n),(\tilde\phi_t^*\tilde X)(n)\bigr)\\
&=\omega_N(n)\bigl((\tilde\phi_t^*\tilde X)(n),\tilde Y(n) \bigr)
+{\omega_N}(n)\bigl(\tilde Y(n),(\tilde\phi_t^*\tilde X)(n)\bigr)=0.
\end{align*}
This shows that $(\phi_t^*X,\phi_t^*\alpha)\in 
\Gamma(\mathsf{D}^\perp)=\Gamma(\mathsf{D})$.
We have proved the following theorem.
\begin{theorem}\label{invDirac}
Let $(M,\mathsf{D})$ be a closed Dirac manifold,  $f\in C^\infty(M)$ an 
admissible function of $(M,\mathsf{D})$, and $X_f$ a Hamiltonian vector 
field for $f$. Let $\phi$ be the flow of $X_f$.
 The Dirac structure is invariant under
$\phi_t$ for all $t$ for which $\phi_t$ is defined,
 that is, $(\phi_t^*X,\phi_t^*\alpha)\in \Gamma(\mathsf{D})$
for all $(X,\alpha)\in\Gamma(\mathsf{D})$.
\end{theorem}

Thus, if $g\in C^\infty(M)^{\mathsf D}$ is an admissible function and 
$\phi$ is the flow of the vector field $X_f$, then 
 the function $\phi^*_tg$ is also admissible. Furthermore, if $X_H$ is a
 solution of the implicit Hamiltonian system $(X_H,\dr H)\in
 \Gamma(\mathsf{D})$  for a Hamiltonian $H\in C^\infty(M)^{\mathsf D}$, 
then the Dirac structure $\mathsf{D}$ is conserved
along the solution curves of the system.

Note that if $X$ is an arbitrary section of $\mathsf{G_1}$, it is \textit{not} 
possible to show in the same manner that the Dirac structure is conserved along the flow lines of $X$. 
Recall also that the space of sections of $\mathsf{G_1}$ is not necessarily
generated by $\{X_f \mid f \text{ admissible}\}$. Therefore, the flows of sections of 
$\mathsf{G_1}$ do not conserve the Dirac structure of $M $, in general. As we shall see later on, this is 
a major technical problem. Certain  conditions on the admissible functions will have to be imposed.

\begin{corollary}
Let $(M,\mathsf{D})$ be a closed Dirac manifold. Then the distribution
$\mathsf{G_0}$ is completely integrable in the sense of Stefan and Sussmann and
each of its leaves inherits the  trivial presymplectic structure.  
\end{corollary}
\begin{proof}
We have to show that the flow of each element of 
$\Gamma(\mathsf{G_0})$ leaves $\mathsf{G_0}$ invariant. If
$X\in\Gamma(\mathsf{G_0})$, then $(X,0)\in\Gamma(\mathsf{D})$ and thus  
$X$ is a Hamiltonian vector field for any constant function on $M$. Thus, 
any constant function $k$ on $M$ is admissible and, by the preceding 
theorem, we get $(\phi_t^*Y,\phi_t^*\beta)\in\Gamma(\mathsf{D})$ for all 
$(Y,\beta)\in\Gamma(\mathsf{D})$, where $\phi$ is the flow of the vector 
field $X$. But then $(\phi_t^*Z,\phi_t^*0)=(\phi_t^*Z,0)\in
\Gamma(\mathsf{D})$ for all $(Z,0)\in\Gamma(\mathsf{D})$ and $t $ for 
which $\phi _t$ is defined, and hence $\phi_t^*Z\in\Gamma(\mathsf{G_0})$ 
for each $Z \in \Gamma(\mathsf{G_0})$. This shows that $\mathsf{G_0} $ is 
completely integrable by the first point of Theorem \ref{frobenius with vector fields}. 

Let $N$ be a leaf of $\mathsf{G_0}$. Then the Dirac structure defines a
$2$-form on $N$ by
\[
\omega_N(n)(\tilde X(n),\tilde Y(n))=\alpha_n(Y(n))=-\beta_n(X(n)),
\]
where  $\tilde{X}, \tilde{Y} \in \mathfrak{X}(N)$; $X, Y\in \mx(M)$ 
and $\alpha, \beta\in\Omega^1(M)$ are chosen such that
$\tilde{X}\sim_{\iota_N}X$,  
$\tilde{Y}\sim_{\iota_N}Y$ and $(X,\alpha),
(Y,\beta)\in\Gamma(\mathsf{D})$. But, since $X$ and $Y$ can be chosen in
$\Gamma(\mathsf{G_0})$,  this yields automatically $\omega_N(n)(\tilde
X(n),\tilde Y(n))=0$ for all $n\in N$, $\tilde X,\tilde Y\in\mx(N)$. 
\end{proof}
Note that the leaf of $\mathsf{G_0}$ through any point $m\in M$ is initially
immersed in the leaf of $\mathsf{G_1}$ through $m$.

\subsection{Proper actions and orbit type manifolds}
\label{sec:proper}
\subsubsection{The stratification by orbit types}\label{tubeth} 
In this section we consider a smooth and proper action 
\begin{equation}
\begin{array}{cccl}
\Phi &:G\times M&\rightarrow& M\\
&(g,m)&\mapsto& \Phi (g,m)=\Phi
_{g}(m)= gm=g\cdot m  \label{action}
\end{array}
\end{equation}
of a Lie group $G$ on a manifold $M$. Let $\pi :M\rightarrow \bar{M}: = 
M/G$ be the  orbit map.

For each closed Lie subgroup $H$ of $G$ we define the \emph{isotropy type} set
\begin{equation*}
M_{H}=\{m\in M\mid G_{m}=H\} 
\end{equation*}
where $G_{m}=\{g\in G\mid gm=m\}$ is the isotropy subgroup of $m\in M$. 
Since the action is proper, all isotropy groups are compact. The sets 
$M_{H}$, where $H $ ranges over the closed Lie subgroups of $G$ for 
which $M_{H}$ is non-empty, form a partition of $M$ and therefore they are 
the equivalence  classes of an equivalence relation in $M$.
Define the normalizer of $H$ in $G$ by
\[
N(H):=\{g\in G\mid gHg^{-1}=H\};
\]
$N(H)$ is a closed Lie subgroup of $G$. Since $H$ is a normal subgroup 
of $N(H)$ the quotient $N(H)/H$ is a Lie group. If $m\in M_{H}$, we have 
$G_{m} =H$ and $G_{gm}=gHg^{-1}$ for all $g\in G$. Consequently, $gm \in M_{H}$ if and only if $g\in N(H)$. The action of $G$ on $M$ restricts to an
action of $N(H)$ on $M_{H}$, which induces a free and proper action of 
$N(H)/H$ on $M_{H}$.

Define the \emph{orbit type} set
\begin{equation}
M_{(H)}:=\{m\in M\mid G_{m}\text{ is conjugated to }H\mathbf{\}.}  \label{P(H)}
\end{equation}
Then, 
\[
M_{(H)}=\{gm\mid g\in G,m\in M_{H}\}=\pi ^{-1}(\pi (M_{H})). 
\]
The connected components of  $M_{H}$ and $M_{(H)}$ are embedded submanifolds of $M$; 
therefore $M_H$ is called an \emph{isotropy type manifold} and $M_{(H)}$ an 
\emph{orbit type manifold}.
Moreover, 
\[
\pi \left(M_{(H)} \right)=\{gm\mid m\in M_{H}\}/G=M_{H}/N(H)=M_{H}/(N(H)/H). 
\]
Since the action of $N(H)/H$ on $M_{H}$ is free and proper, it follows that 
$M_{H}/(N(H)/H)$ is a quotient manifold of $M_{H}$. Hence the subset
$\pi(M_{(H)}) \subseteq \bar M=M/G$ is a manifold.

The partitions of $M$ by the connected components of the orbit type manifolds
 is a decomposition of the differential space $M$. The
corresponding stratification of $M$ is called the \textit{orbit type
stratification} of $M$ (\cite{DuKo00}, \cite{Pflaum01}).
The orbit space $\bar M=M/G$ with its quotient topology
has also the structure of a stratified space with strata the projections of the
connected components of the orbit type manifolds.

\subsubsection{Tube theorem and $G$-invariant average}
If the action of the Lie group $G$ on $M$ is proper, we can find for each
point $m\in M$ a $G$-invariant neighborhood of $m$ such that the action can be
described easily on this neighborhood. The proof of the following theorem can
be found, for example,  in \cite{OrRa04}.
\begin{theorem}[Tube Theorem]
\label{tube_thm}
Let $M$ be a manifold and $G$ a Lie group
acting properly on $M$. For a given point $m\in M$ denote $H := G_m$. 
Then there exists a 
$G $-invariant open neighborhood $U$ of the orbit $G\cdot m$, called 
\emph{tube at $m$}, and a $G $-equivariant diffeomorphism
$G \times_H B \stackrel{\sim}\longrightarrow  U$. The set $B$ is an open 
$H$-invariant neighborhood of $0$
in an $H $-representation space $H$-equivariantly isomorphic to 
$T_mM/T_m(G\cdot m)$. 
The $H $-representation on $T_mM/T_m(G\cdot m)$ is given by
$h\cdot (v + T_m(G \cdot m)) := T_m\Phi_h(v) + T_m(G
\cdot m)$, $h \in H $, $v \in T_mM $. The smooth manifold  
$ G \times_H B$ is the quotient of the smooth 
free and proper (twisted)
action $\Psi$ of $H$ on $G\times B$ given by $\Psi(h,(g,b)):=(g
h^{-1},h\cdot b)$, $g \in G $, $h \in H $, $b \in B$. 
The $G $-action on $G \times _H B $ is given by $k\cdot [g, b]: = 
[kg, b]_H $, where $k, g \in G $, $b \in B $, and $[g, b]_H \in G \times _H B$
 is the equivalence class  {\rm (}i.e., $H $-orbit{\rm )} of $(g,b)$.
\end{theorem}

Let $m\in M$ and $H:=G_m$. If the action of $G$ on $M$ is proper, the 
isotropy subgroup $H$ of $m$ is a compact Lie subgroup of $G$. Hence, 
there exists a Haar measure $dh$ on $H$, that is, a $G$-invariant measure 
on $H$ satisfying $\int_Hdh=1$ (see, for example, \cite{DuKo00}). Left 
$G$-invariance of $dh$ is equivalent to right $G$-invariance of $dh$, that is, 
$R_g^*dh=dh=L_g^*dh$ for all $g\in H$, where $L_g:H\to H$ (respectively
$R_g:H\to H$) denotes left (respectively right) translation by $g$ on $H$.  

Let $X\in\mx(M)$ be defined on the tube $U$ at $m\in M$ of the proper 
action of the Lie group $G$ on $M$. Using the Tube Theorem, we write  
the points of $U$ as equivalence classes $[g,b]_H$ with $g\in G$ and 
$b\in B$. Define the vector field $X_G$ by 
\begin{align*}
X_G([g,b]_H):=\left(\Phi_{g^{-1}}^*\left(\int_H\Phi_h^*Xdh\right)\right)([g,b]_H),
\end{align*}
that is, for each point $m'=[g,b]_H\in U$ we have
\begin{align*}
X_G([g,b]_H)=T_{[e,b]_H}\Phi_g
\left(\int_H\left(T_{[h,b]_H}\Phi_{h^{-1}}X([h,b]_H)\right)dh\right).
\end{align*}
This defines a smooth $G$-invariant vector field $X_G$ called 
the \emph{$G$-invariant average\/} of the vector field
$X$ (see \cite{JoRaSn11}). Note that $X_G$ is, 
in general, not equal to $X$ (at any point); it can
even vanish. Indeed,  $G$-invariant vector fields
are tangent to the isotropy type manifolds (see \cite{OrRa04}).
Hence, if we choose a $G$-invariant 
Riemannian 
metric on $M$
and a section $X$ of the ($G$-invariant) orthogonal $TP^\perp\subseteq TM\an{P}$ 
of $TP$
relative to this metric, where $P$ is a stratum of $M$, its $G$-invariant
average is both a section of $TP^\perp$ and tangent to
$P$, so it is  the zero section (see also
\cite{CuSn01}, Lemma 2.4). 

Similarly, define for $\alpha\in\Omega^1(M)$ the $G$-invariant
average $\alpha_G\in \Omega^1(M)^G$ of $\alpha$ by
\begin{align*}
\alpha_G([g,b]_H):=\left(\Phi_{g^{-1}}^*
\left(\int_H\Phi_h^*\alpha dh\right)\right)([g,b]_H),
\end{align*}
that is, for each point $m'=[g,b]_H\in U$ we have
\begin{align}
\label{alpha_G}
{\alpha_G}([g,b]_H)&=\left(\int_H\Phi_h^*\alpha dh\right)_{[e,b]_H}\circ
T_{[g,b]_H}\Phi_{g^{-1}} =\left(\int_H\left(\alpha([h,b]_H)\circ T_{[e,b]_H}\Phi_h\right) dh\right)\circ
T_{[g,b]_H}\Phi_{g^{-1}}.
\end{align}
The one-form $\alpha_G$ is well-defined, 
smooth, and
$G$-invariant (see \cite{JoRaSn11}). 

If $(X,\alpha)$ is a section of a $G$-invariant generalized
distribution $\Delta$, then $(X_G,\alpha_G)$ is a $G$-invariant section of $\Delta$. 

Note that, in the same manner, we can define the $G$-invariant average 
$f_G$ of a smooth  function $f$ defined on the tube $U$ for the action of $G
$ at $m$. The function $f_G$ is defined by
\[
f_G([g,b]_H):=\int_{h\in H}f([h,b]_H)dh.
\]

\subsubsection{Descending sections of $\mathsf P_M$}\label{subsec_descending_sections}
Let $\V_G^\circ$ be the cotangent
distribution on $M$ spanned pointwise 
by the $G$-equivariant sections of $\V^\circ$, 
that is, $\mathcal{V}_G^\circ(m): = \left\{\alpha_m\mid \alpha \in 
\Gamma( \mathcal{V}^{\circ})^G \right\} = \left\{\mathbf{d}f(m) \mid f 
\in C ^{\infty}(M)^G \right\}$ (see Lemma 5.8 in \cite{JoRaSn11}).
 If  $\alpha\in\Gamma(\V^\circ)^G$, it
pushes-forward to the ``one-form'' $\bar\alpha:=\pi_*\alpha$ such that,
for every $\bar Y\in \mx(\bar M)$ and every vector field  $Y\in\mx(M)$ 
satisfying $Y\sim_\pi \bar Y$, we have 
\[\pi^*(\bar\alpha(\bar Y))=\alpha(Y).\]

Each vector field $X$ satisfying
$[X,\Gamma(\V)]\subseteq\Gamma(\V)$ can be written as
$X=X^G+X^\V$, with $X^G\in\mx(M)^G$ and 
$X^\V\in\Gamma(\V)$,
and $X $ pushes-forward to a ``vector field'' $\bar X$
on $\bar M$. Since we will not need these objects in the rest of the paper,
we  will not give more details about 
what we call the ``vector fields'' and ``one-forms''
on the stratified space $\bar M=M/G$ and refer to \cite{JoRaSn11} for more information.

A local section $(X,\alpha)$ of $TM\oplus
\V^\circ=\K^\perp$ satisfying $[X,\Gamma(\V)]\subseteq
\Gamma(\V)$ and $\alpha\in\Gamma(\V^\circ)^G$ is called a
\emph{descending section} of $\mathsf{P}_M$.

\subsubsection{Review of standard Dirac reduction}

\paragraph{Symmetries of Dirac manifolds.}
Let $(M, \mathsf D)$ be a smooth Dirac manifold.
Let $G$ be a Lie group and
$\Phi: G\times M \rightarrow M$ a smooth left action. Then $G$ is called a
\emph{symmetry Lie group of} $\mathsf D$ if for every $g\in G$ the 
condition
$(X,\alpha) \in \Gamma(\mathsf D)$ implies that  $\left( \Phi_g^\ast X, 
\Phi_g^\ast  \alpha \right) \in \Gamma(\mathsf D)$. We say then that the Lie 
group $G$ acts \textit{canonically} or \textit{by Dirac actions} on $M$. 

Let $\mathfrak{g}$ be a Lie algebra and $\xi \in \mathfrak{g} \mapsto \xi_M 
\in \mathfrak{X}(M)$ be a smooth left Lie algebra action, that is, the map $(x, 
\xi) \in M \times \mathfrak{g} \mapsto \xi_M(x) \in TM $ is smooth and $\xi \in
\mathfrak{g} \mapsto \xi_M  \in \mathfrak{X}(M)$ is a Lie algebra
anti-homomorphism.  
The Lie algebra $\mathfrak{g}$ is said to be a 
\emph{symmetry Lie algebra of} $\mathsf D$ if for every $\xi \in 
\mathfrak{g}$ the condition $(X,\alpha) \in \Gamma(\mathsf D)$ implies that  
$\left(\boldsymbol{\pounds}_{\xi_M}X,\boldsymbol{\pounds}_{\xi_M}\alpha \right) \in \Gamma(\mathsf D)$.  Of 
course, if $\mathfrak{g}$ is the Lie algebra of the Lie group
$G $ and $\xi_M$ is the infinitesimal generator for all $\xi\in \mathfrak{g}$, 
then if $G$ is a symmetry Lie group  of $\mathsf D$ it follows that 
$\mathfrak{g}$ is a symmetry Lie algebra of $\mathsf D$.

\paragraph{Standard Dirac reduction} 
We present a short review of the Dirac reduction methods. For more 
details, see \cite{JoRaSn11} and \cite{JoRaZa11}.
Let $(M,\mathsf D)$ be a smooth Dirac manifold acted upon in a smooth 
proper and Dirac manner by a Lie group $G$ such that the intersection 
$\mathsf D\cap(\T\oplus\V_G^\circ)$ is spanned pointwise by its descending sections.

Consider the subset $\D^G$ of $\Gamma(\mathsf D)$ defined by
\[ 
\D^G:=\{(X,\alpha)\in\Gamma(\mathsf D)\mid \alpha\in\Gamma(\V^\circ)^G\text{ and }
[X,\Gamma(\V)]\subseteq\Gamma(\V)\},
\]
that is, the set of the descending sections of $\mathsf D$.

Each vector field $X$ satisfying
$[X,\Gamma(\V)]\subseteq\Gamma(\V)$ pushes forward to a vector field 
$\bar X$ on $\bar M$. For each stratum $\bar{P}$ of $\bar{M}$, 
the restriction of $\bar{X}$ to points
of ${\bar P}$ is a vector field $X_{\bar{P}}$ on $\bar{P}$. 
On the other hand, if
$(X,\alpha)\in\D^G$, then we have $\alpha\in\Gamma(\V^\circ)^G$ and it
pushes forward to the one-form $\bar{\alpha}:=\pi_*\alpha$ such that,
for every $\bar{Y}\in \mx(\bar{M})$ and every vector field  $Y\in\mx(M)$ satisfying
$Y\sim_\pi \bar Y$, we have 
\[\pi^*(\bar\alpha(\bar Y))=\alpha(Y).\]
Moreover, for each stratum ${\bar P}$ of $\bar M$, the restriction of 
$\bar\alpha$ to points of ${\bar P}$ defines a one-form $\alpha_{\bar P}$ 
on ${\bar P}$. Let 
\[
\bar\D:=\{(\bar X,\bar\alpha)\mid (X,\alpha)\in\D^G\}
\]
and for each stratum $\bar{P}$ of $\bar{M}$, set 
\[
\D_{\bar P}:=\{(X_{\bar P},\alpha_{\bar P})\mid (\bar X,\bar\alpha)\in\bar\D\}.\]
Define the smooth generalized
distribution $\mathsf D_{\bar P}$ on ${\bar P}$ by
\begin{equation}\label{def_D_barP}
\mathsf D_{\bar P}(s):=\{(X_{\bar P}(s),\alpha_{\bar P}(s))
\in T_s{\bar P}\times T^*_s{\bar P}\mid (X_{\bar P}, \alpha_{\bar P})\in\D_{\bar P}\}
\end{equation}
for all $s\in\bar P$.
Note that $\Gamma(\mathsf D_{\bar{P}}) = \mathcal{D}_{\bar{P}}$. 
We have the following three theorems.

\begin{theorem}\label{singred2}
Let $(M,\mathsf D)$ be a Dirac manifold with a proper Dirac action of a connected Lie
group $G$ on it. Assume that the intersection $\mathsf D\cap(\T\oplus\V_G^\circ)$ is 
spanned poitwise by its descending sections.
Then each element $(\bar X,\bar \alpha)\in\mx(\bar M)\times \Omega^1(\bar M)$
orthogonal to all the sections in $\bar \D$ is already an element of $\bar \D$.
\end{theorem}

\begin{theorem}\label{singred}
Let $(M,\mathsf D)$ be a Dirac manifold with a proper Dirac action of a connected Lie group 
$G$ on it.
Let $\bar P$ be a stratum of the quotient space $\bar M$.
If $\mathsf D\cap(\T\oplus\V_G^\circ)$ is spanned poitwise by its descending sections, then
$\mathsf D_{\bar P}$ defined in \eqref{def_D_barP} is a Dirac structure on ${\bar P}$.
If $(M,\mathsf D)$ is integrable, then $(\bar P,  \mathsf D_{\bar P})$ is integrable.
\end{theorem}

In the regular case, this simplifies to the following statement.
\begin{theorem}\label{conj-orbits-red} 
Let $G$ be a connected Lie group acting in a proper way on the manifold 
$M$ such that all isotropy subgroups are conjugated. Assume that $\mathsf D\cap\K^\perp$
has constant rank on $M$, where $\K^\perp:=TM\oplus\V^\circ$.
Then the  Dirac structure $\mathsf D$ on $M$ induces a Dirac structure 
$\bar{\mathsf D}$ on the quotient $\bar M=M/G$ given by
\begin{equation}\label{red_dir_conj_iso}
\bar{\mathsf D}(\bar m):=\left\{\left(\bar X(\bar m),\bar\alpha(\bar m)\right)\in
  T_{\bar m}\bar M\times T^*_{\bar m}\bar M \,\bigg|
\begin{array}{c}\exists X\in
  \mx(M)\text{ such that } X\sim_{\pi}\bar X\\ 
\text{ and }(X,\pi^*\bar\alpha)\in\Gamma(\mathsf D)\end{array}\right\}
\end{equation} for all $\bar m\in \bar M$. If $\mathsf D$ is integrable, 
then $\bar{\mathsf D}$ is also integrable.
\end{theorem}

\section{The free case}
\label{section_free_case}

In this section we present, without proofs, the theory of optimal reduction for free Lie group actions. We do 
this because the main ideas are easier to follow in this situation and because this case follows 
closely the non-free Poisson case. 
The proofs will be given later for the general case of a proper 
action; this is technically considerably more involved due to
the fact that the characteristic distribution $\mathsf{P_1}$ 
of an arbitrary Dirac structure is not equal to $T^*M$ (as was the
case for a Poisson structure). 
In the general case, \emph{two} types of natural optimal distributions arise simultaneously
as generalizations
of the Poisson optimal distribution. This is related to the fact that
one can consider \emph{orbit type} and \emph{isotropy type} manifolds
when carrying out singular reduction.
In addition, as we shall see, 
these two natural distributions are not integrable, in general.
In the free case, where they are both equal, one has to assume 
that the intersection of the Dirac structure with 
the vector bundle $\mathcal K^\perp$ associated to the action has constant rank
(as for the standard reduction). In the general case, this doesn't make sense because $\mathcal K^\perp$
is not necessarily smooth; we will give additional hypotheses guaranteeing the integrability of the optimal distributions.

\medskip

Let $(M,\mathsf D)$ be a \textit{closed} Dirac manifold, $G$ a symmetry 
Lie group of $\mathsf D$ acting freely and properly on
$M$. Assume in the following that $\mathsf D\cap \mathcal{K}^\bot$ is
a vector bundle, where $ \mathcal{K} = \mathcal{V} \oplus \{0\} \subset  TM \oplus T ^\ast M $ and 
$\mathcal{K}^\bot=TM\oplus \mathcal{V}^\circ$. To define the optimal momentum map (as in \cite{OrRa04}, \S5.5.7) 
we need to introduce an additional smooth distribution.
Define 
\[
\mathcal{D}_G(m):=\{X(m)\mid \text{there is } \alpha \in \Gamma(\mathcal{V}^\circ)
\subseteq\Omega^1(M)\text{ such that }(X,\alpha) \in
\Gamma(\mathsf D)\}\subseteq\mathsf{G_1}(m)
\]
for all $m\in M$. Then $\mathcal{D}_G = \cup_{m \in  M}\mathcal{D}_G(m)$ is a smooth distribution on $M$.  

If the manifold $M$ is Poisson and the Dirac structure is the graph of the
Poisson map $\sharp:T^*M \rightarrow TM$, then 
$\mathcal{D}_G(p)=\{X_f(p)\mid \text{there is } f\in C^\infty(M)^G \text{ such that } X_f= \sharp(\mathbf{d}f)\}$, which recovers the definition in \cite{OrRa04}.

Returning to the general case of Dirac manifolds, note that
\begin{equation}
\label{D_G_characterization}
\mathcal{D}_G = \pi_{TM}(\mathsf{D}\cap
(TM\oplus\mathcal{V}^\circ))=\pi_{TM}(\mathsf{D}\cap 
\mathcal{K}^\bot),
\end{equation}
where $\pi_{TM}: TM \oplus T ^\ast M \rightarrow TM$ is the projection on the first factor 
and that we always have 
\[
\mathsf{G_0}\subseteq\mathcal{D}_G\subseteq\mathsf{G_1}.
\]
We have the following lemmas.
\begin{lemma}\label{ginv}
Let  $(X,\alpha), (Y,\beta) \in \Gamma(\mathsf D\cap\mathcal{K}^\perp)$, i.e., 
 $X,Y\in \Gamma(\mathcal{D}_G)$. Then the $1$-form ${\boldsymbol{\pounds}}_{X}\beta-{\mathbf{i}}_{Y} \mathbf{d} \alpha$
is a local section of $\mathcal{V}^\circ$. 
\end{lemma}

For the proof see Lemma \ref{lemma42}.

\begin{corollary}
If $\mathsf D$ is integrable, the space of local sections of the intersection of vector bundles 
$\mathsf D\cap \mathcal{K}^\bot$ is closed under the Courant bracket. Hence, under the assumption that 
$\mathsf D\cap \mathcal{K}^\bot$ has constant dimensional fibers, this vector bundle inherits a Lie algebroid
structure relative to the  Courant-Dorfman bracket on  $\Gamma (\mathsf D\cap \mathcal{K}^\bot)$ and the
anchor map $\pi_{TM}:\mathsf D\cap \mathcal{K}^\bot\to TM$. Thus, the distribution
$\mathcal{D}_G=\pi_{TM}(\mathsf D\cap \mathcal{K}^\bot)$ is integrable in the sense of Stefan-Sussmann.
\end{corollary}
This corollary is an immediate consequence of Lemma \ref{ginv}
and the closedness of $\mathsf D$. Note that in the general case of proper nonfree actions,
we will have to \emph{assume} that the corresponding
distributions are integrable, or give additionnal hypotheses under which this is true. 

Thus, if $ \mathsf D \cap \mathcal{K} ^\perp$ is a vector bundle,
it is a Lie algebroid and  $M$ admits hence a generalized foliation by the 
leaves of the generalized
distribution $\mathcal{D}_G$ (see \cite{Courant90a}). The optimal
 momentum map is now defined like in \cite{OrRa04}.

\begin{definition} Assume that $ \mathsf D \cap \mathcal{K} ^\perp$ is a vector subbundle of 
$ TM \oplus T ^\ast M $. The projection 
\begin{equation}
\mathcal{J}: M\to M/\mathcal{D}_G
          \end{equation}
on the leaf space of $\mathcal{D}_G$ is called the \textbf{(Dirac) optimal
  momentum map}.
\end{definition}

In order to formulate  the reduction theorem for the
optimal momentum map, we need an induced action of $G$ on the leaf space of
$\mathcal{D}_G$. This doesn't follow, as usual, from the $G$-equivariance
of the vector fields spanning $\mathcal{D}_G$ because, in this case, they  are  not necessarily $G$-equivariant.

\begin{proposition}\label{inducedaction}
If $m$ and $m'$ are in the same leaf of $\mathcal{D}_G$ then $\Phi_g(m)$ and $\Phi_g(m')$ are in the same leaf of $\mathcal{D}_G$ for all $g \in
G$. Hence there is a well defined action $\bar{\Phi}:G\times M/\mathcal{D}_G \to M/\mathcal{D}_G$ given by
\begin{equation*}
\bar{\Phi}_g(\mathcal{J}(m)):=\mathcal{J}(\Phi _g(m)) 
\end{equation*}
\end{proposition}

For the proof see Propositions  \ref{inducedaction} and \ref{inducedaction2}.

Denote by  $G_\rho$ the isotropy subgroup of $\rho\in M/\mathcal{D}_G$
for this induced action. If $g\in G_\rho$ and $m\in \mathcal{J}^{-1}(\rho)$, then
\[\mathcal{J}(\Phi_g(m))=\bar{\Phi}_g(\mathcal{J}(m))
=\bar{\Phi}_g(\rho)=\rho=\mathcal{J}(m)\] 
and we conclude, as usual, that $G_\rho$ leaves $\mathcal{J}^{-1}(\rho)$ invariant. Thus we get an induced
action of $G_\rho$ on $\mathcal{J}^{-1}(\rho)$, which is free if the original $G$-action on $M$ is free. 

Also, $\mathcal{J}^{-1}(\rho)$ is an initial submanifold of $M$ since it is a leaf of the generalized foliation defined by the integrable distribution $\mathcal{D}_G$. By Proposition 3.4.4 in \cite{OrRa04}, there is a unique smooth structure on $G_ \rho$ with respect to which this subgroup is an initial Lie subgroup of $G $ with Lie algebra
\[
\mathfrak{g}_ \rho = \{ \xi\in \mathfrak{g} \mid \xi_M(m) \in T_m \mathcal{J}^{-1}( \rho) , \text{ for all } m \in \mathcal{J}^{-1}( \rho) \}.
\] 
In general, $G _\rho$ is not closed in $G $.

\begin{definition}\label{noether_property}
Let $(M,\mathsf D)$ be a Dirac manifold with integrable Dirac structure $\mathsf D$ and $G$ a
Lie group acting canonically on it. Let $P$ be a set and $\mathbf{J}:M\to P$ a
map. We say that $\mathbf{J}$ has the \textbf{Noether property} for the
$G$-action on $(M,\mathsf D)$ if the flow $F_t$ of any implicit Hamiltonian vector
field associated to any $G$-invariant admissible function $h\in C^\infty(M)$
preserves the fibers of $\mathbf{J}$, that is, 
\begin{equation*}
\mathbf{J}\circ F_t=\mathbf{J}\arrowvert_{{\rm Dom}(F_t)}
\end{equation*}
where ${\rm Dom}(F_t)$ is the domain of definition of $F_t$.
\end{definition}

Like in the Poisson case (see Theorem 5.5.15 in \cite{OrRa04}), 
one gets the following universality property. 
Note that if  $\mathsf D\cap \mathcal K^\perp$ is spanned by 
sections with exact cotangent projections, i.e.,
by the family
$\{(X_f,\dr f)\in\Gamma(\mathsf D)\mid \dr f\in\Gamma(\V^\circ)\}$, 
then, by $G$-invariant averaging, it is spanned 
by 
$\{(X_f,\dr f)\in\Gamma(\mathsf D)\mid  f\in C^\infty(M)^G\}$.

\begin{theorem}
Let $G$ be a symmetry Lie group of the Dirac manifold $(M,\mathsf D)$ and 
$\mathbf{J}:M\to P$ a function with the Noether property.
Assume that $\mathsf D\cap \mathcal K^\perp$ is spanned by 
sections with exact cotangent projections.
Then there exists a unique map $\phi:M/\mathcal{D}_G\to P$
such that the following diagram commutes:
\[\begin{xy}
\xymatrix@!R{
      M \ar[rr]^{\mathbf{J}} \ar[rd]_{\mathcal{J}}  &     &  P  \\
                             &  M/\mathcal{D}_G \ar[ur]_\phi &
  }\end{xy}\]If $\mathbf{J}$ is  $G$-equivariant with respect to some $G$-action on $P$,
then  $\phi_G$  is
 also $G$-equivariant.
If $\mathbf{J}$ is smooth
 and $M/\mathcal D_G$ is a smooth manifold, then $\phi_G$  is
 also smooth.
\end{theorem}
For the proof, see Theorem \ref{universality}.

\bigskip

Now we can generalize the optimal reduction procedure from Poisson
manifolds (see \cite{OrRa04}, Theorem 9.1.1) to closed Dirac manifolds. As we shall see,
with appropriately extended definitions this important desingularization
method works also for  Dirac manifolds.

\begin{theorem}[Optimal point reduction by Dirac actions]\label{red_theorem}
Let $(M,\mathsf D)$ be an integrable Dirac manifold and $G$ a Lie group acting
freely and properly on $M$ and leaving the Dirac structure invariant. Assume
that $\mathsf D\cap\mathcal{K}^\perp$ is constant dimensional and let
$\mathcal{J}: M\to M/\mathcal{D}_G$ be the optimal (Dirac) momentum map
associated to this action. Then, for any $\rho \in M/\mathcal{D}_G$ whose
isotropy subgroup $G_\rho$ acts properly on $\mathcal{J}^{-1}(\rho)$, the orbit space $M_\rho=\mathcal{J}^{-1}(\rho)/G_\rho$ is a smooth
presymplectic regular quotient manifold with presymplectic form 
$\omega_\rho \in \Omega^2(M _\rho)$ defined by 
\begin{equation}\label{presymp}
\left(\pi_\rho^*\omega_\rho \right)(m)(X(m),Y(m))=\alpha_m(Y(m))=-\beta_m(X(m))
\end{equation}
for any $m \in \mathcal{J}^{-1}(\rho)$ and any $X,Y\in
\Gamma(\mathcal{D}_G)$ defined on an open set containing $m$, where 
$\alpha, \beta \in \Gamma(\mathcal{V}^\circ)$ are such that $(X,\alpha), 
(Y,\beta) \in \Gamma(\mathsf D\cap\mathcal{K}^\perp)$, and $\pi_\rho:\mathcal{J}^{-1}(\rho)\to M_\rho$ is the projection. The pair $(M_\rho,\mathsf D_\rho)$ is called the \emph{(Dirac optimal) point reduced space} of $(M,\mathsf D)$ at $\rho$, where $\mathsf D_\rho$ is the graph of the presymplectic form $\omega_\rho$.
\end{theorem}

Note that if $\mathsf D$ is the graph of a Poisson structure on $M$, the
distribution $\mathsf{G}_0$ is $\{0\}$, all functions in $C^\infty(M)$ are admissible,
and we are in the setting of the Optimal point reduction by Poisson actions
 (see \cite{OrRa04}, Theorem 9.1.1).

\medskip

Recall that, since $\mathsf D\cap \mathcal{K}^\perp$ is assumed to have constant
dimensional fibers, one can build the reduced Dirac manifold 
$(\bar{M}, \bar{\mathsf D})$ as in Theorem \ref{conj-orbits-red}.
 The following theorem gives the relation between
the reduced manifold $\bar M$ and the reduced manifolds $M_\rho$ given by the
optimal reduction theorem. 

\begin{theorem}\label{presympl} If $m\in\mathcal{J}^{-1}(\rho)\subseteq M$, the reduced
  manifold $M_\rho$ is diffeomorphic to the presymplectic leaf $\bar N$ through
  $\pi(m)$ of the reduced Dirac manifold $(\bar M,\mathsf D_{\rm red})$ via the map
  $\Theta: M_\rho\to \bar N$, $\pi_\rho(x)\mapsto (\pi\circ i_\rho)(x)$. Furthermore, $\Theta^*\omega_{\bar N}=\omega_\rho$, where $\omega_{\bar N}$ is the
  presymplectic form on $\bar N$.
\end{theorem}

\begin{example}
Consider a closed Dirac manifold $(M,\mathsf D)$ and the trivial Lie group 
$G=\{e\}$. Then the trivial action of $G$
on $M$ is Dirac and its vertical space is just the zero section in $TM$.
Thus, the intersection $\mathsf D\cap \mathcal K^\perp$
is equal to $\mathsf D\cap (TM\oplus T^*M)=\mathsf D$.
The projection $\pi_{TM}$ of 
$\mathsf D\cap \mathcal K^\perp$ 
is hence just the smooth distribution $\mathsf{G_1}$, 
which is known to be completely integrable in the sense
of Stefan and Sussmann (see \cite{Courant90a}).

In this situation, we get consequently the leaves of the presymplectic foliation of the 
Dirac manifold $(M,\mathsf D)$
as reduced spaces and we recover also the statement of the preceding theorem
in the case of a trivial action.
\end{example}

In the following, we will show how these results generalize to the non-free case.
The main difficulty is the fact that the distribution $\mathcal K$ doesn't have constant rank in this case,
and $\mathcal K$ and $\mathcal K^\perp$ are not necessarily spanned by sections that descend to the 
quotient. Also, the sets spanned by the pairs of vector fields and one-forms that 
descend to the quotient, and by the vector fields and one-forms that descend to the quotient 
and are $G$-invariant, are different, in general. Taking the intersection of the 
Dirac structures with each one of them yield two different singular foliations, 
that are related to the stratifications by orbit types, and by isotropy types, respectively.

\section{The optimal distributions}
\label{sec:optimal_distributions}

From now on we assume that $(M,\mathsf{D})$ is an integrable
Dirac manifold and  $G$ a
symmetry Lie group of $\mathsf{D}$ acting properly on $M$.
Let $\T$ be the tangent distribution on $M$ spanned  pointwise by the  vector fields that
\emph{descend} to $\bar{M}: = M/G$, i.e., $\mathcal{T}$ is
the distribution spanned pointwise by
\begin{align*}
F:=&\{X\in\mx(M)\mid X=X^G+X^\V,\text{ with } X^G\in\mx(M)^G,\text{ and }
X^\V\in\Gamma(\V)\}\\
=&\{X\in\mx(M)\mid [X,\Gamma(\V)]\subseteq \Gamma(\V)\}.
\end{align*} 
We have shown in \cite{JoRaSn11} that $\T$ is integrable and its leaves are 
the connected components of the orbit type manifolds. 
In the same manner, let $\T_G$ be the distribution on $M$ spanned by the 
set $F_G$ of $G$-invariant vector fields on $M$:
\[
F_G=\{X\in\mx(M)\mid \Phi^*_gX=X\quad  \forall g\in G\}.
\]
It is shown in \cite{OrRa04} that this smooth distribution is also completely
integrable in the sense of Stefan and Sussmann; its leaves are the 
connected components of the isotropy type manifolds. Note that the 
integrability of $\T$ follows from the integrability of $\T_G$ and Proposition 
3.4.6 in \cite{OrRa04}.
In particular, both distributions coincide in the case of a free action. This is why 
the constructions made here simplify in a significant manner in this special case reviewed in Section \ref{section_free_case}.

By the considerations at the end of Subsection \ref{subsec_descending_sections}, the
generalized distribution $\T\oplus \V_G^\circ$ is  the generalized
distribution spanned pointwise by the descending sections of $\mathsf P_M$. 
The descending sections of $\mathsf{D}$  
will be of great importance in the rest of this paper. These sections 
necessarily  lie in  $\mathsf{D}\cap(\T\oplus \V_G^\circ)$.

In the free case, $\mathsf{D}\cap(\T\oplus \V_G^\circ)=
\mathsf{D}\cap\mathcal{K}^\perp$, which is assumed 
to be of constant rank. Since in the general case, 
$\mathcal{T}\oplus \mathcal{V}_G^\circ$
does not have constant rank, it does not make much sense 
to assume that its intersection with $\mathsf{D}$ has constant 
rank. In the first subsection, we assume that $\mathsf{D}\cap
(\mathcal{T}\oplus \mathcal{V}_G^\circ)$
is  spanned by its descending sections, as is required for the standard singular reduction 
(see Theorems \ref{singred}, \ref{singred2}).
We show that, under this hypothesis, 
the distributions are both algebraically involutive and 
the integrability of
$\mathcal D_G$ follows from the integrability of $\mathcal D$.
In the second subsection, we conclude  the integrability
of $\mathcal D_G$ from an ``exactness'' condition. The proof is  in the same spirit as the proof 
of the integrability of the characteristic distribution associated to a Poisson structure.

\begin{definition}
Assume that $\mathsf{D}\cap(\T\oplus \V_G^\circ)$
and $\mathsf{D}\cap(\T_G\oplus \V_G^\circ)$
are smooth generalized distributions.
The smooth tangent distributions
\[\D:=\pi_{TM}(\mathsf{D}\cap(\T\oplus \V_G^\circ))\quad
\text{ and }
\quad \D_G:=\pi_{TM}(\mathsf{D}\cap(\T_G\oplus \V_G^\circ))\]
are
 called the \emph{orbit type optimal
  distribution} and the \emph{isotropy type optimal
  distribution}.
\end{definition}

\subsection{The  ``locally finitely generated'' assumption}
In this section, we \textbf{assume that $\mathsf{D}\cap
(\T\oplus \V_G^\circ)$ is  spanned by its descending sections}, 
hence it is smooth. Recall the notation $\K:=\V\oplus \{0\}$.
We want to show that, under this assumption,
$\D$ and $\D_G$ are  algebraically involutive. Furthermore, 
the integrability of $\D_G$ will be a consequence of the 
integrability of $\D$. 

Note that in the free case, we have $\mathcal D=\mathcal D_G=\pi_{TM}(\mathsf D\cap \mathcal K^\perp)$.
If $\mathsf D\cap \mathcal K^\perp$ has constant rank on the manifold $M$, it has the structure of a
Lie  algebroid
and, by a standard result in \cite{Courant90a}, $\mathcal D$ is completely integrable.
The main result in \cite{JoRaZa11} implies that  
$\mathsf D\cap \mathcal K^\perp$ is then spanned by its descending sections, and hence  
$\mathcal D$ has the same property. 

\begin{lemma}\label{lemma42}
If $(X,\alpha)$ and $(Y,\beta)$ are descending sections of $\mathsf{D}$, then the
Courant bracket $[(X,\alpha),(Y,\beta)]$ is also a descending section of $\mathsf{D}$.
\end{lemma}

\begin{proof}
Since $\mathsf{D}$ is closed, we have automatically 
\[[(X,\alpha),(Y,\beta)]=\left([X,Y],\boldsymbol{\pounds}_{X}\beta-\ip{Y}\dr\alpha\right)\in
\Gamma(\mathsf{D}).\] 
The inclusion $[[X,Y],\Gamma(\V)]\subseteq \Gamma(\V)$ is easy to see. 
Indeed if $V$ is a section of $\V$, using the Jacobi  identity we get
\[ [[X,Y],V]=[[X,V],Y]+[X,[Y,V]]\in\Gamma(\V),\]
since $[[X,V],Y]$, $[X,[Y,V]]\in\Gamma(\V)$.

Next, we have to show that $\boldsymbol{\pounds}_{X}\beta-\ip{Y}\dr\alpha$ is a $G$-invariant
section of $\V^\circ$. Write $X=X^G+X^\V$ and $Y=Y^G+Y^\V$ with $X^G$,
$Y^G\in\mx(M)^G$ and $X^\V$, $Y^\V\in\Gamma(\V)$ (see \cite{JoRaSn11}). For each $\xi\in \mathfrak{g}$ we get
\begin{align*}
(\boldsymbol{\pounds}_{X}\beta-\ip{Y}\dr\alpha)(\xi_M)
={}&(\ip{X}\dr\beta)(\xi_M)+\dr(\ip{X}\beta)(\xi_M)-(\ip{Y}\dr\alpha)(\xi_M)\\
={}&X(\beta(\xi_M))-\xi_M(\beta(X))-\beta([X,\xi_M])+\xi_M(\beta(X))\\
&-Y(\alpha(\xi_M))+\xi_M(\alpha(Y^G+Y^\V))+\alpha([Y,\xi_M])\\
={}&X(0)-\xi_M(\beta(X))-0+\xi_M(\beta(X))-Y(0)+\xi_M(\alpha(Y^G))+0=0,
\end{align*}
where we have used the the $G$-invariance of of the function 
$\alpha(Y^G)$, $[X,V]$, $[Y,V] \in \mathcal{V}$ for all $V\in \mathcal{V}$, 
and $\beta(V)=\alpha(V)=0$ for all $V\in\Gamma(\V)$. Since the
fundamental vector fields $\xi_M$, $\xi\in\lie g$, span 
 $\Gamma(\V)$ as a $C^\infty(M)$-module,  we get
$\boldsymbol{\pounds}_{X}\beta-\ip{Y}\dr\alpha\in\Gamma(\V^\circ)$. 
It remains to show that $\boldsymbol{\pounds}_{X}\beta-\ip{Y}\dr\alpha$ is 
$G$-equivariant. To see this, let $\xi\in \mathfrak{g}$ and compute
\begin{align*}
\boldsymbol{\pounds}_{\xi_M}(\boldsymbol{\pounds}_{X}\beta-\ip{Y}\dr\alpha)&=\boldsymbol{\pounds}_{\xi_M}\boldsymbol{\pounds}_{X}\beta-\boldsymbol{\pounds}_{\xi_M}\ip{Y}\dr\alpha\\
&=\boldsymbol{\pounds}_{[\xi_M,X]}\beta+\boldsymbol{\pounds}_{X}\boldsymbol{\pounds}_{\xi_M}\beta-\ip{[\xi_M,Y]}\dr\alpha-\ip{Y}\boldsymbol{\pounds}_{\xi_M}\dr\alpha
=0+0-\ip{[\xi_M,Y]}\dr\alpha-0,
\end{align*}
since $[\xi_M,X]\in\Gamma(\V)$ and $\beta$, $\dr\alpha$ are
$G$-equivariant. Because  $[\xi_M,Y]\in \Gamma(\V)$, we can write
$[\xi_M,Y]=\sum_{i=1}^kf_i\xi^i_M$  with $\xi^1,\ldots,\xi^k\in\lie g$ and
$f_1,\ldots,f_k\in C^\infty(M)$, and hence we have  
\begin{align*}
\boldsymbol{\pounds}_{\xi_M}(\boldsymbol{\pounds}_{X}\beta-\ip{Y}\dr\alpha)&=-\ip{[\xi_M,Y]}\dr\alpha=-\ip{\sum_{i=1}^kf_i\xi^i_M}\dr\alpha\\
&=-\sum_{i=1}^kf_i\ip{\xi^i_M}\dr\alpha
=-\sum_{i=1}^kf_i\boldsymbol{\pounds}_{\xi^i_M}\alpha+\sum_{i=1}^kf_i\dr(\ip{\xi^i_M}\alpha)
=-\sum_{i=1}^if_i\cdot 0+\sum_{i=1}^if_i\dr(0)=0.
\end{align*} 
Hence we have 
\[\frac{d}{dt}\Phi_{\exp(t\xi)}^*(\boldsymbol{\pounds}_{X}\beta-\ip{Y}\dr\alpha)
=\Phi_{\exp(t\xi)}^*\boldsymbol{\pounds}_{\xi_M}(\boldsymbol{\pounds}_{X}\beta-\ip{Y}\dr\alpha)=0,\]
for all $\xi\in\lie g$. This yields
$\Phi_{\exp(t\xi)}^*(\boldsymbol{\pounds}_{X}\beta-\ip{Y}\dr\alpha)=\boldsymbol{\pounds}_{X}\beta-\ip{Y}\dr\alpha$
for all $t\in \R$ and $\xi\in\lie g$. Since $G$ is a connected Lie group and 
hence generated, as a group, by $\exp\lie g$, the one-form $\boldsymbol{\pounds}_{X}\beta-\ip{Y}\dr\alpha$ is $G$-invariant.
\end{proof}
The following  proposition shows \textit{the algebraic involutivity of $\D$}.

\begin{proposition}
Assume that $\mathsf D\cap(\T\oplus\V_G^\circ)$ is spanned
by its descending sections.
Let $(X,\alpha)$ and $(Y,\beta)$ be sections of 
$\mathsf{D}\cap(\T\oplus\V_G^\circ)$. Then
the Courant bracket 
\[[(X,\alpha),(Y,\beta)]=\left([X,Y],\boldsymbol{\pounds}_{X}\beta-\ip{Y}\dr\alpha\right)\]
is also an element of $\Gamma(\mathsf{D}\cap(\T\oplus\V_G^\circ))$.
Hence, if $X, Y \in \Gamma( \mathcal{D})$ then 
$[X,Y]\in \Gamma(\D)$, and $\mathcal D$ is algebraically involutive.
\end{proposition}

\begin{proof}
Write $(X,\alpha)$ and $(Y,\beta)$ as sums
\[
(X,\alpha)=\sum_{i=1}^kf_i(X_i,\alpha_i)\quad \text{ and }\quad (Y,\beta)
=\sum_{j=1}^lg_j(Y_j,\beta_j),
\]
with $f_1,\ldots,f_k,g_1,\ldots,g_l\in C^\infty(M)$ and
$(X_1,\alpha_1),\ldots,(X_k,\alpha_k),(Y_1,\beta_1),\ldots,(Y_l,\beta_l)$
descending sections of $\mathsf{D}$. Since
\[
[(X_1,\alpha_1),f(X_2,\alpha_2)]=f[(X_1,\alpha_1),(X_2,\alpha_2)]+X_1(f)(X_2,\alpha_2)
\]
for all $(X_1,\alpha_1), (X_2,\alpha_2)\in\Gamma(\mathsf{D})$ and $f\in 
C^\infty(M)$, we get
\begin{align*}
[(X,\alpha),(Y,\beta)]&=\left[\sum_{i=1}^kf_i(X_i,\alpha_i),\sum_{j=1}^lg_j(Y_j,\beta_j)\right]\\
&=\sum_{i,j=1}^{k,l}\bigl(g_jf_i\left[(X_i,\alpha_i),(Y_j,\beta_j)
  \right]-g_jY_j(f_i)(X_i,\alpha_i)+f_iX_i(g_j)(Y_j,\beta_j)\bigr).
\end{align*}
This shows that $[(X,\alpha),(Y,\beta)]$ can be written  as a 
$C^\infty(M)$-combination of  the descending sections  $[(X_j,\alpha_j),
(Y_i,\beta_i)]$, $(X_j,\alpha_j)$,
and $(Y_i,\beta_i)$ of $\mathsf{D}$ ($i=1,\ldots,k$, $j=1,\ldots,l$). 
The bracket
 $[(X,\alpha),(Y,\beta)]$ is thus a section of 
$\mathsf{D}\cap(\T\oplus\V_G^\circ)$.
\end{proof}

We study now properties of the smooth distribution $\D_G$.
\begin{proposition}\label{invgen}
If $\mathsf{D}\cap(\T\oplus\V^\circ_G)$ is spanned pointwise by its descending sections, 
then
$\mathsf{D}\cap(\T_G\oplus\V^\circ_G)$ is spanned \emph{pointwise}
by its $G$-invariant 
descending
sections and it is, in particular, smooth.
\end{proposition}

\begin{proof}
Choose  $m\in M$, and \[(v_m,l_m)\in
\bigl(\mathsf{D}\cap(\T_G\oplus\V^\circ_G)\bigr)(m)
\subseteq\bigl(\mathsf{D}\cap(\T\oplus\V^\circ_G)\bigr)(m).\]  
Since $\mathsf{D}\cap(\T\oplus\V^\circ_G)$ is spanned by its 
descending sections, we find a smooth section $(X,\alpha)$ of 
$\mathsf{D}\cap(\T\oplus\V^\circ_G)$ such that
the vector field $X$ is descending, 
$\alpha\in\Gamma(\V^\circ)^G$, $X(m)=v_m$, and $\alpha(m)=l_m$. 
Since the action of $G$ on $M$ is proper,
there exists a tube $U$ for the action at $m$ (see Theorem
\ref{tube_thm}). Assume that 
$(X,\alpha)$ is defined on the whole of $U$; otherwise, 
multiply $(X,\alpha)$ by a bump function that is equal to 
$1$ on a neighborhood $U_1\subsetneq U$ of $m$, and equal to
$0$ on the complement of a neighborhood $U_2$ of $m$ with 
$U_1\subsetneq U_2\subsetneq U$. Consider the $G$-invariant 
averaging $(X_G,\alpha_G)$ at $m$ of the pair
$(X,\alpha)$. Since $(X,\alpha)$ is a section of $\mathsf{D}
\cap(\T\oplus\V^\circ_G)$,
the pair $(X_G,\alpha_G)$ is also a section of this intersection, and we get
\begin{align*}
X_G(m)&=X_G([e,0]_H)=\int_HT_{[h,0]_H}\Phi_{h^{-1}}X([h,0]_H)dh,
\\
\alpha_G(m)&=\alpha_G([e,0]_H)=\int_H\alpha_{[h,0]_H}\circ T_{[e,0]_H}\Phi_{h}dh,
\end{align*}
where $H=G_m$ is the isotropy group of the point $m$. Note that $X_G(m)$ and
$\alpha_G(m)$ only depend on the values of $X$ and $\alpha$ at $m$, and thus,
the multiplication with the bump function doesn't change the situation if the
section $(X,\alpha)$ was not defined on the whole of $U$. Since $\alpha$ is
$G$-invariant and $H$ is the isotropy group of $m$, 
we have, in particular, $\alpha(m)\circ T_m\Phi_h=\alpha(m)$ for
all $h\in H$ (note that again, the multiplication with the bump function
doesn't change anything). Thus, since $[h,0]_H=[hh^{-1},h\cdot 0]_H=[e,0]_H=m$
for all $h\in H$, we get
\[X_G(m)=\int_HT_{m}\Phi_{h^{-1}}X(m)dh\]
and 
\[\alpha_G(m)=\int_H\alpha(m)dh=\alpha(m)=l_m.\]
Since $X(m)=v_m\in\T_G(m)$, it is tangent to the isotropy type manifold
through $m$ and hence, using Lemma 2.4 in \cite{CuSn01}, we conclude that
$X_G(m)= \int_HT_{m}\Phi_{h^{-1}}X(m)dh=X(m)=v_m$. 

The section $(X_G,\alpha_G)$ is hence a $G$-invariant section of
$\mathsf{D}\cap(\T_G\cap\V^\circ_G)$ taking the value $(v_m,l_m)$ at $m$.
\end{proof}
\begin{theorem} \label{integrability_implication}
Assume that $\mathsf{D}\cap(\T\oplus\V^\circ_G)$ is spanned pointwise by its descending sections.
If $\D$ is an integrable distribution, then $\D_G$ is also integrable. 
\end{theorem}
\begin{proof}
We have seen in the preceding proposition that $\mathsf{D}\cap(\T_G\oplus\V^\circ_G)$
is spanned pointwise by its $G$-invariant sections 
\[\left\{(X,\alpha)\in\Gamma(\mathsf{D})\mid X\in\mx(M)^G 
\text{ and }\alpha\in\Gamma(\V^\circ)^G \right\} .\]
Let $(X,\alpha)$ be a descending section of $\mathsf{D}$ and  
$\phi$ the flow of the vector field $X$. Since $\D$ is 
integrable, we know that $\phi_t^*Y$ is is a section
of $\mathsf{D}\cap(\T\oplus\V^\circ_G)$ provided 
$(Y,\beta)$ is a descending section of $\mathsf{D}$ for all 
$t $ for which $\phi_t$ is defined (see Theorem 
\ref{frobenius with vector fields} about  the integrability of smooth
distributions spanned pointwise by families of vector fields).

Consider, in particular, a $G$-invariant pair $(X,\alpha)$ of
$\mathsf{D}\cap(\T_G\oplus\V^\circ_G)$. Then $(X,\alpha)$ is a descending section of $\mathsf{D}$
and, denoting again by $\phi$ the flow of $X$, we know that  $\phi_t^*Y$ is a
section of $\mathsf{D}\cap(\T\oplus\V^\circ_G)$ if $(Y,\beta)$ is a $G$-invariant
descending section of $\mathsf{D}$. But since $X$ and $Y$  are $G$-invariant vector fields, the
flow $\phi$ of $X$ is $G$-equivariant and the vector field $\phi_t^*Y$ is,
consequently, $G$-invariant. Let $\beta^t$ be a section of $\V_G^\circ$ such
that $(\phi_t^*Y,\beta^t)\in\Gamma(\mathsf{D}\cap(\T\oplus\V^\circ_G))$. Since
$\phi_t^*Y\in\mx(M)^G$, we even have 
$(\phi_t^*Y,\beta^t)\in\Gamma(\mathsf{D}\cap(\T_G\oplus\V^\circ_G))$, and hence
$\phi_t^*Y\in\Gamma(\D_G)$.

By  Theorem \ref{frobenius with vector fields}, we conclude that 
$\D_G$ is completely integrable in the sense of Stefan and Sussmann.
\end{proof}

\subsection{The ``exactness'' condition}\label{subsec_exactness}
We make now another assumption: the intersection
$\mathsf{D}\cap(\T\oplus\V_G^\circ)$ is spanned pointwise
by its set of \emph{exact} descending
sections
\[\left\{(X_f,\dr f)\in \Gamma(\mathsf{D})\mid f\in C^\infty(M)^G, \,
  [X_f,\Gamma(\V)]\subseteq \Gamma(\V)\right\}.\]
We will show that the (isotropy type) optimal distribution is also integrable
in this case. 

\begin{theorem}
If the generalized distribution $\mathsf{D}\cap(\T\oplus\V_G^\circ)$ is
spanned pointwise by its set of \emph{exact} descending sections, then the distribution $\D_G$
is smooth and integrable.
\end{theorem}

\begin{proof}
Since $\mathsf{D}\cap(\T\oplus\V^\circ_G)$ is spanned pointwise by its exact descending sections,
we can show, as in the proof of Proposition \ref{invgen},
 that  $\mathsf{D}\cap(\T_G\oplus\V^\circ_G)$ is spanned pointwise by its exact
$G$-invariant descending sections: choose $m\in M$ and 
 $(v_m,\alpha_m)\in(\mathsf D\cap (\T_G\oplus\V_G^\circ))(m)$. Then 
$(v_m,\alpha_m)\in(\mathsf D\cap (\T\oplus\V_G^\circ))(m)$
and, by hypothesis, we find a smooth exact descending section
$(X_f,\dr f)$ of $\mathsf D$ defined in a neighborhood of $m$
and such that $(X_f,\dr f)(m)=(v_m,\alpha_m)$. Consider the
$G$-invariant average $(X_G,\alpha_G)$ of $(X_f,\dr f)$ in a tube 
centered at $m$. The pair  $(X_G,\alpha_G)$ is a section of $\mathsf D$ 
because $\mathsf D$ is $G$-invariant. Since $\dr f\in\Gamma(\V^\circ)^G$, 
we find $\alpha_G=\dr f$ and hence $(X_G,\dr f)$ is a $G$-invariant 
descending section of $\mathsf D$. Since $v_m\in\T_G(m)$, $v_m$ is 
tangent to the isotropy type manifold through $m$ and we find
\[
X_G(m)=\int_HT_m\Phi_{h\inv}X_f(m)dh=\int_HT_m\Phi_{h\inv}v_mdh=v_m,
\]
where $H=G_m$. Thus, $(X_G,\dr f)$ satisfies $(X_G,\dr f)(m)=
(v_m,\alpha_m)$.

Hence, the distribution $\D_G$ is spanned pointwise by the family
\[
F^G:=\{X\in \mx(M)^G\mid \text{there exists }  f\in C^\infty(M)^G \text{ such that } (X,\dr
f)\in\Gamma(\mathsf{D})\} \subseteq \mathfrak{X}(M).
\]
We write $X_f$ for a $G$-invariant vector field
corresponding to the admissible function $f\in C^\infty(M)$. 
Let $X_f$ be an element of $F^G$ and denote by $\phi$ the flow of $X_f$. Let
$X_g$ be an element of $F^G$, corresponding to the admissible function $g\in
C^\infty(M)^G$. By Theorem \ref{invDirac}, we know that
$(\phi_t^*X_g,\phi_t^*\dr g)$ is a section of $\mathsf{D}$. Furthermore, since $X_f$ is
$G$-invariant, its flow $\phi_t$ is $G$-equivariant and thus,
$\phi_t^*X_g\in\mx(M)^G$ and $\phi_t^*g\in C^\infty(M)^G$. This shows that
$\phi_t^*X_g$ is an element of $F^G$ and hence that $\D_G$ is completely
integrable in the sense of Stefan and Sussmann.
\end{proof}

The following proposition is not needed in the rest of the paper; we add it here for the sake of completeness.

\begin{proposition}
If the intersection $\mathsf{D}\cap(\T_G\oplus\V_G^\circ)$ is smooth, then 
it is spanned pointwise by its $G$-invariant descending sections.
\end{proposition}

\begin{proof}
If $m\in M$ and  $(v_m,l_m)\in \mathsf{D}\cap(\T_G\oplus\V_G^\circ)(m)$, 
then we find a smooth section 
$(X,\alpha)$ of $\mathsf{D}\cap(\T_G\oplus\V_G^\circ)$ 
defined on a neighborhood $U'$ of $m$ in $M$ such that $(X(m),
\alpha(m))=(v_m,l_m)$. Let $H=G_m$ be the isotropy group of the point 
$m$ and let $U$ be a tube for the action of $G$ at $m$; assume that 
$(X,\alpha)$ is defined on the whole of $U$ (otherwise, we multiply  
$(X,\alpha)$ with a bump function as in the proof of Proposition 
\ref{invgen}). The average $(X_G,\alpha_G)$ is a $G$-invariant section 
of $ \mathsf{D}\cap(\T_G\oplus\V_G^\circ)$,
hence a $G$-invariant descending section of $\mathsf{D}$.

As in the proof of Proposition
\ref{invgen}, we deduce that $X_G(m)=X(m)=v_m$.
Since $\alpha$ is a section of $\V_G^\circ$, it can be written as
$\alpha=\sum_{i=1}^kf_i\alpha_i$, where $f_i\in C^\infty(M)$
and $\alpha_i\in\Gamma(\V^\circ)^G$.
Again, as in the proof of Proposition \ref{invgen}, we compute
\begin{align*}
\alpha_G(m)&=\sum_{i=1}^k\int_Hf_i(hm)\alpha_i(hm)\circ T_m\Phi_hdh
=\sum_{i=1}^k\int_Hf_i(m)\alpha_i(m)dh=\sum_{i=1}^kf_i(m)\alpha_i(m)=\alpha(m)=l_m.
\end{align*}
Hence, we have found a smooth $G$-invariant descending section
$(X_G,\alpha_G)$ of $\mathsf{D}$ taking the value $(v_m,l_m)$ at
 the point $m$.
\end{proof}

\begin{remark}
Let  $(M,\pois)$ be a smooth Poisson manifold with a canonical and
proper smooth action of a Lie group $G$ on it. Let $\mathsf{D}$ be 
the Dirac structure associated to the Poisson bracket on $M$. Since 
$\V^\circ_G$ is generated by the differentials of the $G$-invariant smooth 
functions, the intersection $\mathsf{D}\cap(TM\oplus\V^\circ_G)$ is
spanned pointwise by the pairs $(X_f,\dr f)$, where $f\in C^\infty(M)^G$. The
vector field  $X_f$ corresponding to a $G$-invariant function $f\in
C^\infty(M)^G$ is $G$-invariant and we get 
\[\mathsf{D}\cap(TM\oplus\V^\circ_G)=\mathsf{D}\cap(\T\oplus\V^\circ_G)
=\mathsf{D}\cap(\T_G\oplus\V^\circ_G),\]
which is spanned by its exact $G$-invariant descending sections.
Hence, we are now in the situation of the exactness condition; the distributions $\D_G$ and $\D$ are equal and completely integrable
 in the sense of Stefan and Sussmann.
 \end{remark}

\section{The optimal momentum maps}
\label{sec:optimal_momentum}
In this section, we define the two optimal momentum maps and show that there is an action 
of the Lie group on the leaf spaces of the two distributions, such that the momentum maps 
are $G$-equivariant. We study also the isotropy subgroups of these actions.

\medskip
Assume  that the distributions $\D$ and $\D_G$ are spanned pointwise by their
descending sections and are completely integrable in the sense of Stefan
and Sussmann.  Let $\mathcal{J}: M\to M/\D$ and $\mathcal{J}_G: M\to M/\D_G$ be
the projections  on the leaf spaces of $\D$ and $\D_G$, respectively. The 
map $\mathcal{J}$  (respectively $\J_G$) is called the \emph{(orbit type)
Dirac optimal momentum map} (respectively the \emph{(isotropy type)  
Dirac optimal momentum map}).  
In this section, we will study these two momentum maps separately.

\subsection{The orbit type Dirac optimal momentum map}
We construct an optimal momentum map induced by the distribution $\mathcal{D}$.

\begin{proposition}\label{inducedaction}
If $m$ and $m'$ are in the same leaf of $\mathcal{D}$ then $\Phi_g(m)$ and
$\Phi_g(m')$ 
are in the same leaf of $\mathcal{D}$ for all $g \in
G$. Hence there is a well defined action $\bar{\Phi}:G\times M/\mathcal{D} \to M/\mathcal{D}$ given by
\begin{equation}
\bar{\Phi}_g(\mathcal{J}(m)):=\mathcal{J}(\Phi _g(m)) \label{actionleafsporbit}
\end{equation}
\end{proposition}

\begin{proof}
Let $m$ and $m'$ be in the same leaf of $\mathcal{D}$. Without 
loss of generality, we can assume that there exists a vector field $X\in\Gamma(\D)$ 
with flow $\phi$ such that  $\phi_t(m)=m'$ for
some $t$ (in reality, $m$ and $m'$ can be joined by finitely many such
curves). Since $(X,\alpha)\in \Gamma(\mathsf{D}\cap(\T\oplus\V_G^\circ))$ 
for some $\alpha\in \Gamma(\mathcal{V}_G^\circ)$ and 
$\mathsf{D}\cap(\T\oplus\V_G^\circ) $ is $G$-invariant, it follows that
$(\Phi^*_gX,\Phi^*_g\alpha)\in\Gamma(\mathsf{D}\cap(\T\oplus\V_G^\circ))$ 
for all $g \in G$.  Hence, $\Phi_g^*X\in \Gamma(\mathcal{D})$ for all 
$g \in G$. For all $s \in [0,t]$ we have
\begin{align*}
\frac{d}{ds} \left( \Phi_g\circ \phi_s\right) (m)
&=T_{\phi_s(m)}\Phi_g \left( X(\phi_s(m)) \right) 
= (\Phi_{g^{-1}}^*X)(\Phi_{g}(\phi_s(m)))\in \mathcal{D}((\Phi_{g}\circ\phi_s)(m)).
\end{align*}
Thus the curve $c:[0,t]\to M$, $s\mapsto\left(\Phi_g\circ \phi_{s}\right) (m)$  connecting
$c(0)=\Phi_g(m)$ to $c(t)= \Phi _g (\phi_t(m)) = \Phi_g(m')$ has all its
tangent vectors in the distribution $ \mathcal{D} $ and hence it lies
entirely in the leaf of $\mathcal{D}$ through the
point $\Phi_g(m)$.
\end{proof}

Denote by  $G_{(\rho)}$ the isotropy subgroup of 
$\rho\in M/\mathcal{D}$ for this induced action. If 
$g\in G_{(\rho)}$ and $m\in \mathcal{J}^{-1}(\rho)$, 
then
\[
\mathcal{J}(\Phi_g(m))=\bar{\Phi}_g(\mathcal{J}(m))
=\bar{\Phi}_g(\rho)=\rho=\mathcal{J}(m)
\] 
and we conclude, as usual, that $G_{(\rho)}$ leaves 
$\mathcal{J}^{-1}(\rho)$ invariant. Thus we get an induced
action $\Phi^{(\rho)}: G_{(\rho)} \times \mathcal{J}^{-1}(\rho) 
\rightarrow \mathcal{J}^{-1}(\rho)$.

Since $\D$ is, by definition,  a subdistribution of
the integral tangent distribution $\T$, each leaf of 
$\mathcal{D}$ is be contained
in a leaf of $\mathcal{T}$. Since the leaves of $\T$ are the connected components 
of the orbit type submanifolds of $M$, the induced action $\Phi^{(\rho)}$ 
of $G_{(\rho)}$ on $\mathcal{J}^{-1}(\rho)$ has  isotropy
subgroups that are conjugated in $G$ for each $\rho\in M/\D_G$, as is 
shown in  the next proposition.

\begin{proposition}\label{iso-subgroups}
Let $\J^{-1}(\rho)$ be a leaf of $\D$ and $m\in\J^{-1}(\rho)$. Then the
isotropy subgroup $\left(G_{(\rho)}\right)_m$ of $m$ by the action of $G_{(\rho)}$ on
$\mathcal{J}^{-1}(\rho)$ is equal to the isotropy subgroup $G_m$.
Hence we have the inclusion 
\[\underset{\J(m)=\rho}\bigcup G_m\subseteq G_{(\rho)}.\]
\end{proposition}
\begin{proof}
Choose $m\in\J^{-1}(\rho) $ and $g\in\left(G_{(\rho)}\right)_m$. Then we have $g\cdot m=m$,
which leads to $g\in G_m$. Conversely, choose $g\in G_m$ and compute
\[\bar\Phi_g(\rho)=\bar{\Phi}_g(\mathcal{J}(m))=\mathcal{J}(g\cdot
m)=\mathcal{J}(m)=\rho.\]
Thus, we have shown that $g\in G_{(\rho)}$. But since $g\cdot m=m$, we have in
particular $g\in (G_{(\rho)})_m$.
\end{proof}
We will see later that the isotropy subgroups of the action of $G_{(\rho)}$ on
$\J^{-1}(\rho)$ are conjugated in $G_{(\rho)}$.

In general, the action of $G_{(\rho)}$ on
$\mathcal{J}^{-1}(\rho)$ is not proper. A sufficient  condition for this
is, for example, the closedness of $G_{(\rho)}$ in $G$, which is not true, in general.

\subsection{The isotropy type Dirac optimal momentum map}

The optimal distribution $\mathcal{D}_G$ gives rise to a second
Dirac optimal momentum map. The results are analogous to 
those in the previous subsection.

\begin{proposition}\label{inducedaction2}
If $m$ and $m'$ are in the same leaf of $\mathcal{D}_G$ then $\Phi_g(m)$ and
$\Phi_g(m')$ 
are in the same leaf of $\mathcal{D}_G$ for all $g \in
G$. Hence there is a well defined action $\bar{\Phi}:G\times M/\mathcal{D}_G \to M/\mathcal{D}_G$ given by
\begin{equation}
\bar{\Phi}_g(\mathcal{J}_G(m)):=\mathcal{J}_G(\Phi _g(m)) \label{actionleafsp}
\end{equation}
\end{proposition}
\begin{proof}
The proof is almost identical to that of Proposition \ref{inducedaction}.
The only difference is that we have to consider the flow of
$X\in\Gamma(\D_G)$. Then we use $G$-invariance of the 
distribution $\mathsf{D}\cap(\T_G\oplus\V_G^\circ)$.
\end{proof}

Let $G_\sigma$ be the isotropy subgroup of $\sigma\in M/\D_G$. 
Then there exists a unique smooth structure on $G _\sigma$ with respect to 
which $G _\sigma$ is an initial Lie subgroup of $G$; the  Lie algebra of $G _
\sigma$ is
\[
\mathfrak{g}_\sigma: = \{ \xi\in \mathfrak{g} \mid \xi_M(m) \in 
\mathcal{D}_G(m), \text{ for all } m \in \mathcal{J}_G ^{-1}(\sigma) \}
\]
 (see \cite{OrRa04},  Proposition 3.4.4). In particular
 \begin{equation}
\label{dim_g_sigma}
 \dim \mathfrak{g}_\sigma = \dim \left(\mathcal{V}(m) \cap \mathcal{D}_G(m) 
 \right) + \dim G_m, \quad \text{for any} \quad m \in \mathcal{J}_G ^{-1}
 (\sigma).
\end{equation}
Indeed, the surjective linear map $\xi\in \mathfrak{g}_\sigma \mapsto 
\xi_M(m) \in \mathcal{V}(m) \cap \mathcal{D}_G(m) $ has kernel $\{ \xi\in 
\mathfrak{g} \mid \xi_M(m) = 0 \}$.

Since $G_\sigma$ leaves $\J_G^{-1}(\sigma)$ invariant,  we get 
an induced action $\Phi^\sigma$ of $G_\sigma$ on 
$\J_G^{-1}(\sigma)$. 
Since $\D_G\subseteq \T_G$, the points of $\J_G^{-1}(\sigma)$ are all 
in the same connected component of an isotropy type submanifold of 
$M$ and hence $G_m=H_\sigma$ for all $m\in \J_G^{-1}(\sigma)$
and some compact subgroup $H_\sigma\subseteq G$.

\begin{proposition}
Let $\J^{-1}_G(\sigma)$ be a leaf of $\D_G$ and $m\in\J^{-1}_G(\sigma)$. 
The isotropy subgroup $\left(G_\sigma\right)_m$ of $m$ by the action 
of $G_\sigma$ on $\mathcal{J}^{-1}_G(\sigma)$ is equal to the compact subgroup $H_\sigma$ which is
automatically a subset of $G_\sigma$.
\end{proposition}
\begin{proof}
Here also, the proof of the analogue in the orbit type case (Proposition
\ref{iso-subgroups}) can be adapted to this particular situation.
\end{proof}

With this last proposition, we can show the missing detail in the preceding
subsection.

\begin{proposition}\label{conj_isotropy_subgroups}
Choose  $\rho\in M/\D$ and let $G_{(\rho)}$ be its isotropy subgroup by 
the action of $G$ on $M/\D$. Then we find for all $m, m'\in\J^{-1}(\rho)$ 
an element $g\in G_{(\rho)}$ such that $G_m=gG_{m'}g^{-1}$.
\end{proposition}

\begin{proof}
We have $\D_G\subseteq \D\subseteq \V+\D_G$, by definition.
For $m\in M$, the leaf $\J^{-1}(\rho)$ of $\D$ through $m$ is thus contained 
in  $G\cdot\J_G^{-1}(\sigma)$, where $\sigma=\J_G(m)$ and $\rho=\J(m)$, 
and the leaf $\J_G^{-1}(\sigma)$ is contained in $\J^{-1}(\rho)$. 

Choose $m'\in \J^{-1}(\rho)$. Then $m'$ can be written as $m'=gm''$ with 
$m''\in \J_G^{-1}(\sigma)\subseteq\J^{-1}(\rho) $ and $g\in G$ . 
Then we have $G_{m'}=gG_{m''}g^{-1}$. But since $m''$ is
an element of $\J_G^{-1}(\J_G(m))\subseteq M_{G_m}$, 
we have $G_{m''}=G_m$ and hence $G_{m'}=gG_mg^{-1}$. Furthermore,
$g\cdot\rho=g\cdot\J(m)=g\cdot\J(m'')=\J(g\cdot m'')=\J(m')
=\rho$, which shows that $g\in G_{(\rho)}$.
\end{proof}

\subsection{Universality of the optimal map $\mathcal J_G$ under the exactness condition 
(\S\ref{subsec_exactness})}

We assume here that the intersection
$\mathsf{D}\cap(\T\oplus\V_G^\circ)$ is spanned pointwise by its set of \emph{exact} descending
sections
\[\left\{(X_f,\dr f)\in \Gamma(\mathsf{D})\mid f\in C^\infty(M)^G, \,
  [X_f,\Gamma(\V)]\subseteq \Gamma(\V)\right\}.\]
We have shown in Subsection 
\ref{subsec_exactness} that the distribution
$\mathcal D_G$ is then completely integrable in the sense of Stefan and Sussmann.
Hence, the isotropy type optimal momentum map is defined.

Note that in this particular case,
the smooth distribution $\mathcal D_G$ is spanned pointwise
by the following family
of vector fields:
\[\mathcal F=\{X_f\in\mx(M)^G\mid (X_f,\dr f)\in \Gamma(\mathsf D), 
f\in C^\infty(M)^G\},
\]
and its leaves are hence the accessible sets of this family of vector fields.
Using this, we prove a universality property of the isotropy type optimal momentum map.
This theorem suggests that the isotropy type optimal momentum map 
should be the more ``natural'' one, provided that the exactness condition above is satisfied.
Recall that in the Poisson case, the optimal distribution is always spanned 
by the family $\mathcal F$, and the following statement is hence true
(see \cite{OrRa04}).

\begin{theorem}\label{universality}
Let $G$ be a symmetry Lie group of the Dirac manifold $(M,\mathsf D)$ and 
$\mathbf{J}:M\to P$ a function with the Noether property 
{\rm (}see Definition \ref{noether_property}{\rm )}.
Then there exists a unique map
$\phi_G:M/\mathcal{D}_G\to P$
such that the following diagram commutes:
\[\begin{xy}
\xymatrix@!R{
      M \ar[rr]^{\mathbf{J}} \ar[rd]_{\mathcal{J}}  &     &  P  \\
                             &  M/\mathcal{D}_G \ar[ur]_{\phi_G} &
  }\end{xy}.\]
If $\mathbf{J}$ is  $G$-equivariant with respect to some $G$-action on $P$,
then  $\phi_G$  is
 also $G$-equivariant.
If $\mathbf{J}$ is smooth
 and $M/\mathcal D_G$ is a smooth manifold, then $\phi_G$  is
 also smooth.
\end{theorem}

\begin{proof} The proof is the same as for Poisson manifolds (see \cite{OrRa04}, Theorem 5.5.15). 
Define $ \phi : M / \mathcal{D}_G \rightarrow  P $ by 
$ \phi ( \rho ) : = \mathbf{J} (m) $, where $ \rho = \mathcal{J} (m) $. 
The map $ \phi $ is well defined since if $m'\in\mathcal{J}^{-1}(\rho)$,
 then there is a finite composition $F$  of flows of elements of $\mathcal F$
 such that $m'=F(m)$. Since $\mathbf{J}$ is a Noether momentum map we have
\[\mathbf{J}(m')=\mathbf{J}(F(m))=\mathbf{J}(m)=\phi(\rho).\]
The definition immediately implies that the diagram commutes. Uniqueness of $ \phi $ 
follows from the requirement that the diagram commutes and the surjectivity of $\mathcal{J}$. 
Equivariance of $ \phi $ is a direct consequence of the definition \eqref{actionleafsp} of the 
$ G $-action on $ M /\mathcal{D}_G $. Finally, if all objects are smooth manifolds and 
$ \mathcal{J} $, $ \mathbf{J} $ are smooth maps then $ \phi $ is a smooth map as the 
quotient of the smooth map $ \mathbf{J} $ by the projection $ \mathcal{J} $ (see \cite{Bourbaki67}). 
\end{proof}

\section{Optimal reduction}
\label{sec:optimal_reduction}
In this section we generalize the optimal reduction theorem for 
Poisson manifolds (see \cite{OrRa04}, Theorem 9.1.1) 
to closed Dirac manifolds. As we shall see,
with necessary assumptions and appropriately extended 
definitions, this important desingularization method works 
also for  Dirac manifolds.

We shall assume throughout this section that the distributions 
$\D$ (respectively $\D_G$) are spanned pointwise by their descending (respectively $G$-invariant) 
sections and are
completely integrable in the sense of Stefan and Sussmann. 
Recall from Subsection \ref{int_sing_dis} that their leaves are then the accessible
sets of these special  families of vector fields.

As we have seen above, there are two optimal momentum maps that we 
can consider. We have hence  two optimal point reduction theorems. We will 
also prove an orbit reduction theorem and shall see that the two optimal 
momentum maps give rise to the same orbit reduction theorem. At the end 
of this section it will be shown that  the three optimally reduced 
manifolds at $\J(m)$ and $\J_G(m)$, for a point $m\in M$,  are isomorphic presymplectic manifolds.

\subsection{The reduction theorems: Optimal point reduction by Dirac 
actions}
\label{optimal_point_reduction_thm}
Let $(M,\mathsf{D})$ be a smooth integrable Dirac manifold and $G$ a Lie group 
acting smoothly and properly on $M$ and  leaving the Dirac structure 
invariant.  Assume  that $\mathsf{D}\cap(\T\oplus\V_G^\circ)$ is 
spanned pointwise by the descending sections of $\mathsf{D}$  and that $\D$ is 
integrable. Let $\mathcal{J}: M\to M/\mathcal{D}$ be the orbit type Dirac 
optimal momentum map associated to this action. If $\rho\in M/\D$, 
denote by $\inc:\mathcal{J}^{-1}(\rho)\hookrightarrow M$ the regular 
immersion.  Recall that $\mathcal{J}^{-1}(\rho)$ is an initial submanifold 
of $M$.

\begin{theorem}\label{red_theorem1}
For any $\rho \in M/\mathcal{D}$ with
isotropy subgroup $G_{(\rho)}$ acting properly on $\mathcal{J}^{-1}(\rho)$, 
the orbit space $M_{(\rho)}:=\mathcal{J}^{-1}(\rho)/G_{(\rho)}$ is a regular
quotient manifold such that  the projection 
$\pi_{(\rho)}:\mathcal{J}^{-1}(\rho)\to M_{(\rho)}$ is a smooth submersion.
Define $\omega_{(\rho)} \in \Omega^2\left(M _{(\rho)}\right)$ by 
\begin{equation}\label{presymp1}
\left(\pi_{(\rho)}^*\omega_{(\rho)}\right)(m)(v_m,w_m):=
\alpha_{\inc(m)}(Y(\inc(m)))
=-\beta_{\inc(m)}(X(\inc(m)))
\end{equation}
for any $m \in \mathcal{J}^{-1}(\rho)$ and any $X,Y\in
\Gamma(\mathcal{D})$ defined on an open set around $\inc(m)$, where  
$(X,\alpha), (Y,\beta) \in \Gamma(\mathsf{D}\cap(\T\oplus\V_G^\circ))$ 
are such that $T_m\inc v_m=X(\inc(m))$ and $T_m\inc w_m=Y(\inc(m))$.

Then $\left(M_{(\rho)},\omega_{(\rho)}\right)$ is a presymplectic manifold.
The pair $\left(M_{(\rho)},\mathsf{D}_{(\rho)}\right)$ is called the 
\emph{orbit type Dirac optimal point reduced space} of 
$(M,\mathsf{D})$ at $\rho$, where $\mathsf{D}_{(\rho)}$ is 
the graph of the presymplectic form $\omega_{(\rho)}$.
\end{theorem}

Since $\mathsf D\cap(\T\oplus\V_G^\circ)$
is spanned pointwise by its descending sections 
and $\D$ is integrable,   $\mathcal{D}_G = \pi_{TM} 
\left(\mathsf{D}\cap(\T_G\oplus\V_G^\circ)\right)$ is spanned
pointwise by its $G$-invariant descending sections  and is completely integrable
in the sense of Stefan and Sussmann by Proposition \ref{invgen} and Theorem 
\ref{integrability_implication}. Let
$\mathcal{J}_G: M\to M/\mathcal{D}_G$ be the isotropy type 
Dirac optimal momentum map. For $\sigma\in M/\D_G$, denote by
$\incg:\mathcal{J}_G^{-1}(\sigma)\hookrightarrow M$ the regular immersion.

\begin{theorem}\label{red_theorem2}
For any $\sigma \in M/\mathcal{D}_G$ with
isotropy subgroup $G_\sigma$ acting properly on $\mathcal{J}_G^{-1}
(\sigma)$, the orbit space 
$M_\sigma:=\mathcal{J}_G^{-1}(\sigma)/G_\sigma$ is a regular
quotient manifold such that  the projection 
$\pi_\sigma:\mathcal{J}^{-1}(\sigma)\to M_\sigma$ is a smooth submersion.
Define  $\omega_\sigma \in \Omega^2\left(M _\sigma\right)$ by 
\begin{equation}\label{presymp2}
\left(\pi_{\sigma}^*\omega_{\sigma}\right)(m)(v_m,w_m):=
\alpha_{\incg(m)}(Y(\incg(m)))
=-\beta_{\incg(m)}(X(\incg(m)))
\end{equation}
for any $m \in \mathcal{J}^{-1}(\sigma)$ and any $X,Y\in
\Gamma(\mathcal{D}_G)$ defined on an open set around $\incg(m)$, where  $(X,\alpha), 
(Y,\beta) \in \Gamma(\mathsf{D}\cap(\T_G\oplus\V_G^\circ))$ are such that 
$T_m\incg v_m=X(\incg(m))$ and
$T_m\incg w_m=Y(\incg(m))$.

Then $(M_{\sigma},\omega_{\sigma})$ is a presymplectic manifold.
The pair
$(M_{\sigma},\mathsf{D}_{\sigma})$ is called the 
\emph{isotropy type Dirac optimal point reduced space} of 
$(M,\mathsf{D})$ at $\sigma$, where $\mathsf{D}_{\sigma}$ is the graph of the presymplectic form 
$\omega_{\sigma}$.
\end{theorem}

\begin{proof}[of Theorem \ref{red_theorem1}]
Recall the notation $\Phi^ {(\rho)}: G _{(\rho)} \times \mathcal{J} ^{-1}(\rho) 
\rightarrow \mathcal{J} ^{-1}(\rho) $ for the restriction of the original
$G$-action 
on $M$ to the Lie group $G _{(\rho)}$ and the manifold 
$\mathcal{J}^{-1}( \rho)$. Since the $G_{(\rho)}$-action on
$\mathcal{J}^{-1}(\rho)$ is, by hypothesis, proper and its isotropy subgroups
are conjugated by Proposition \ref{conj_isotropy_subgroups}, 
the quotient $M_{(\rho)}:=\mathcal{J}^{-1}(\rho)/G_{(\rho)}$ is a regular
quotient manifold and the projection
$\pi_{(\rho)}:\mathcal{J}^{-1}(\rho)\to\mathcal{J}^{-1}(\rho)/G_{(\rho)}$ 
is a smooth surjective submersion.

We show that $\omega_{(\rho)}$ given by (\ref{presymp1}) is well-defined.
Let $m,m'\in \mathcal{J}^{-1}(\rho)$ be such that 
$\pi_{(\rho)}(m)=\pi_{(\rho)}(m')$
and let $v,w\in T_m\mathcal{J}^{-1}(\rho)$,
$v',w'\in T_{m'}\mathcal{J}^{-1}(\rho)$ be such that $T_m\pi_{(\rho)}(v)=T_{m'}\pi_{(\rho)}(v')$,
$T_m\pi_{(\rho)}(w)=T_{m'}\pi_{(\rho)}(w')$. Let $(X,\alpha)$,$(X',\alpha')$, $(Y,\beta)$,
$(Y',\beta')$ be sections of $\mathsf{D}\cap(\T\oplus\V_G^\circ)$  such that
\begin{equation*}
\begin{array}{ll}
X(\iota_{(\rho)}m)=T_m\iota_{(\rho)}v,& X'(\inc(m'))=T_{m'}\inc v',\\
Y(\inc(m))=T_m\inc w,& Y'(\inc(m'))=T_{m'}\inc w'.
\end{array}
\end{equation*}
 The condition $\pi_{(\rho)}(m)=\pi_{(\rho)}(m')$ implies the existence of an
element $g\in G_{(\rho)}\subseteq G$ such that $m'=\Phi_g^{(\rho)}(m)$. We 
have then $\pi_{(\rho)}=\pi_{(\rho)}\circ \Phi_g^{(\rho)}$ and thus 
$T_m\pi_{(\rho)}=T_{m'}\pi_{(\rho)}\circ T_m\Phi_g^{(\rho)}$. 
Furthermore, because of the equalities $T_m\pi_{(\rho)}(v)=T_{m'}
\pi_{(\rho)}(v')$,
$T_m\pi_{(\rho)}(w)=T_{m'}\pi_{(\rho)}(w')$, we have 
\[
T_{m'}\pi_{(\rho)}(T_m\Phi_g^{(\rho)}(v)-v')=0\qquad
\text{and} \qquad T_{m'}\pi_{(\rho)}(T_m\Phi_g^{(\rho)}(w)-w')=0
\]
and there exist elements $\xi^1$, $\xi^2 \in 
\lie g$
 such that $\xi^1_M(\iota_{(\rho)}(m')), \xi^2_M(\iota_{(\rho)}(m'))
\in\D(\iota_{(\rho)}(m'))$, 
\begin{align*}
X'(\inc(m'))-T_{\inc(m)}\Phi_g(X(\inc(m)))
=&T_{\inc(m')}(v'-T_m\Phi_g^{(\rho)}(v))
=T_{\inc(m')}\xi^1_{\mathcal{J}^{-1}(\rho)}(m')
=\xi^1_M(\inc(m'))
\end{align*}
and
\begin{align*}
Y'(\inc(m'))-T_{\inc(m)}\Phi_g(Y(\inc(m)))
=&T_{\inc(m')}(w'-T_m\Phi_g^{(\rho)}(w))
=T_{\inc(m')}\xi^2_{\mathcal{J}^{-1}(\rho)}(m')=\xi^2_M(\inc(m')),
\end{align*}
where we have used the equality $\inc\circ\Phi^{(\rho)}_g=\Phi_g\circ\inc$.
This yields 
\[
X'(n')=\left((\Phi_{g^{-1}})^*X\right)(n')+\xi^1_M(n')\qquad 
\text{and} \qquad Y'(n')=\left((\Phi_{g^{-1}})^*Y\right)(n')+\xi^2_M(n'),
\]
where we let $n:=\inc(m)$ and $n':=\inc(m')$.
Since  $(X', \alpha') $ and $((\Phi_{g^{-1}})^*Y,(\Phi_{g^{-1}})^*\beta)$
are sections of $\mathsf{D}$ in a neighborhood of the point $n'$, we have 
\begin{equation}\label{phik}
((\Phi_{g^{-1}})^*\beta)(X')=-\alpha'\left((\Phi_{g^{-1}})^*Y\right),
\end{equation}
and thus we conclude 
\begin{align*}
 \omega_{(\rho)}(\pi_{(\rho)}(m'))(T_{m'}\pi_{(\rho)}(v'),T_{m'}\pi_{(\rho)}(w'))
=&\,(\pi_{(\rho)}^*\omega_{(\rho)})(m')(v',w')\\
=&\,\alpha'(n')(Y'(n'))=\alpha'(n')\left(\left((\Phi_{g^{-1}})^*Y\right)(n')
+\xi^2_M(n')\right)\\
=&\alpha'(n')\left(\left((\Phi_{g^{-1}})^*Y\right)(n')\right)
+\alpha'(n')\left(\xi^2_M(n')\right)\\
\overset{(*)}=&\,\alpha'(n')\left(\left((\Phi_{g^{-1}})^*Y\right)(n')\right)
\overset{\makebox[0pt][c]{\tiny\eqref{phik}}}{=}
-((\Phi_{g^{-1}})^*\beta)(n')(X'(n'))\\
=&-((\Phi_{g^{-1}})^*\beta)(n')\left(\left((\Phi_{g^{-1}})^*X\right)(n')
+\xi^1_M(n')\right)\\
\overset{(*)}=&-((\Phi_{g^{-1}})^*\beta)(n')
\left(\left((\Phi_{g^{-1}})^*X\right)(n')\right)
=-\beta(n)(X(n))\\
=&\,\omega_{(\rho)}(\pi_{(\rho)}(m))(T_{m}\pi_{(\rho)}(v),T_{m}\pi_{(\rho)}(w)).
\end{align*}
For the equalities $(*)$, we use the fact that $\alpha'$ and
$\Phi^*_{g^{-1}}\beta$  are sections of $\V^\circ$.

Finally, we show that $\omega_{(\rho)}$ is closed.
Let $m\in \mathcal{J}^{-1}(\rho)$ and choose $\tilde X,\tilde Y,\tilde Z\in
\mx(\mathcal{J}^{-1}(\rho))$, defined on a neighborhood of $m$ in
$\mathcal{J}^{-1}(\rho)$. Then there exist sections
$(X,\alpha)$, $(Y,\beta)$, $(Z,\gamma)\in \Gamma(\mathsf{D}\cap(\T\oplus\V_G^\circ))$
defined on a neighborhood of $\inc(m)$ in $M$ such that 
\[
\tilde X\sim_{\inc}X,\quad \tilde Y\sim_{\inc}Y, \quad \text{and}\quad
\tilde Z\sim_{\inc}Z.
\] 
Since $[\tilde X,\tilde Y]\sim_{\inc}[X,Y]$ and
\[
\left([X,Y],  \boldsymbol{\pounds}_{X}\beta-\ip{Y}\dr\alpha\right)\in\Gamma(\mathsf{D}),
\]
we have (by definition \eqref{presymp1})
\begin{equation}\label{form_on_bracket}
\left(\pi_{(\rho)}^*\omega_{(\rho)}\right)\left([\tilde X,\tilde Y],\tilde Z\right)
=-\gamma([X,Y])\circ\inc=\left(\boldsymbol{\pounds}_{X}\beta-\ip{Y}\dr\alpha\right)(Z)\circ\inc.
\end{equation}
Thus, recalling the definition \eqref{presymp1}, we get
\begin{align*}
\mathbf{d}\left(\pi_{(\rho)}^*\omega_{(\rho)}\right)(\tilde X,\tilde Y,\tilde Z)
&=\tilde X\left[(\pi_{(\rho)}^*\omega_{(\rho)})(\tilde Y,\tilde Z)\right]
-\tilde Y\left[(\pi_{(\rho)}^*\omega_{(\rho)})(\tilde X,\tilde Z)\right]\\
&\qquad +\tilde Z\left[(\pi_{(\rho)}^*\omega_{(\rho)})(\tilde X,\tilde
  Y)\right]-(\pi_{(\rho)}^*\omega_{(\rho)})\left([\tilde X,\tilde Y],\tilde Z\right)\\
& \qquad +(\pi_{(\rho)}^*\omega_{(\rho)})\left([\tilde X,\tilde Z],\tilde
  Y\right)-(\pi_{(\rho)}^*\omega_{(\rho)})\left([\tilde Y,\tilde Z],\tilde X\right)\\
&\overset{\eqref{form_on_bracket}}=
\tilde X\left[\beta(Z)\circ\inc\right]+\tilde
Y\left[\gamma(X)\circ\inc\right]+\tilde Z\left[\alpha(Y)\circ\inc\right]\\
&\qquad +\gamma([X,Y])\circ\inc
+(\pounds_{X}\gamma-{\mathbf{i}}_{Z}\mathbf{d}\alpha)(Y)\circ\inc
 +\alpha([Y,Z])\circ\inc\\
& =\Bigl(X\left[\beta(Z)\right]+Y\left[\gamma(X)\right]+Z\left[\alpha(Y)\right]
+\gamma\left([X,Y]\right)-\gamma\left([X,Y]\right)\\
&\qquad 
+X\left[\gamma(Y)\right]-Z\left[\alpha(Y)\right] +Y\left[\alpha(Z)\right]+\alpha\left([Z,Y]\right)
+\alpha\left([Y,Z]\right)\Bigr)\circ \inc\\
& =
\Bigl(X\left[\beta(Z)+\gamma(Y)\right]+Y\left[\gamma(X)+\alpha(Z)\right] 
\Bigr)\circ \inc =0,
\end{align*} 
where we used the fact that $\gamma(X)+\alpha(Z)=0$ and $\gamma(Y)+\beta(Z)=0$
(this follows directly from $(X,\alpha), (Y,\beta), (Z,\gamma)\in\Gamma(\mathsf{D})$).
Thus,
$\pi_{(\rho)}^*\mathbf{d}\omega_{(\rho)}=\mathbf{d}(\pi_{(\rho)}^*\omega_{(\rho)})=0$
and, because $\pi_{(\rho)}$ is a surjective submersion, this yields 
$\mathbf{d}\omega_{(\rho)}=0$. Therefore, $\omega_{(\rho)}$ is a well-defined 
presymplectic form on $M_{(\rho)}$.
\end{proof}
Theorem \ref{red_theorem2} has a similar proof.

\subsection{Optimal orbit reduction}
Let $(M,\mathsf{D})$ be a smooth integrable Dirac manifold with a smooth and 
proper canonical action of a Lie group $G$ on it. Assume that the same 
conditions on $\mathsf D\cap(\T\oplus\V_G^\circ)$
as in the preceding subsection are satisfied.
Let  $\J:M\to M/\D$, $\J_G:M\to M/\D_G$ be the optimal momentum maps.
Consider the distribution $\D_G+\V\subseteq TM$. By
Proposition 3.4.6 in \cite{OrRa04}, it is integrable. 
Since $G$ is connected, the leaves of $\D_G+\V$ are the sets
 $G\cdot\J_G\inv(\sigma)=G\cdot\J\inv(\J(m))$ for any 
 $m\in\J_G\inv(\sigma)$ (recall that $\J\inv(\J(m))\subseteq G\cdot\J_G
 \inv(\sigma)$ for any $m\in\J_G\inv(\sigma)$). The leaves of $\D_G+\V$
are initial submanifolds of $M$; that is, the maps $\iota_{\sigma,G}:G\cdot
\J_G\inv(\sigma)\hookrightarrow M$ are regular immersions for all $\sigma\in M/\D_G$.

\begin{lemma}
Choose $\sigma \in M/\D_G$ such that the action of $G_\sigma$ on
$\J_G\inv(\sigma)$ is proper.
The integral leaf $G\cdot \J_G\inv(\sigma)$ of $\D_G+\V$ is diffeomorphic to
the regular quotient manifold 
\[
G\times_{G_\sigma}\J_G\inv(\sigma):=\left(G\times\J_G\inv(\sigma)\right)/
G_\sigma,
\]
where the action $A_\sigma$ of $G_\sigma$ on $G\times\J_G\inv(\sigma)$ 
is the twisted action
\[ 
\begin{array}{lccc}
A_\sigma:&G_\sigma\times \left(G\times\J_G\inv(\sigma)\right)
&\to&G\times\J_G\inv(\sigma)\\
&(h,(g,m))&\mapsto & (g h\inv,h\cdot m).
\end{array}\] 
\end{lemma}
\begin{remark}
The properness of the action of $G_\sigma$ on $G\times \J_G\inv(\sigma)$
follows from the properness of the action of $G_\sigma$ on $\J_G\inv(\sigma)$.
\end{remark}

\begin{proof}
Define $F:G\times_{G_\sigma}\J_G\inv(\sigma)\to G\cdot \J_G\inv(\sigma)$ 
by  $F([g,x]_{G_\sigma})=gx\in G\cdot \J_G\inv(\sigma)$. The map $F$ is
well-defined. To see this, note that if $[g,x]_{G_\sigma}=[g',x']_{G_\sigma}$, 
then there exists $h\in G_\sigma$ such that $g'=gh\inv$ and $x'=hx$. But then we have
$F([g',x']_{G_\sigma})=g'x'=gh\inv hx=gx=F([g,x]_{G_\sigma})$.

The inverse of the function $F$ is given by $F^{-1}:G\cdot \J_G\inv(\sigma)\to
G\times_{G_\sigma}\J_G\inv(\sigma)$, $F^{-1}(g\cdot x)=[g,x]_{G_\sigma}$ for any
$g\cdot x\in G\cdot \J_G\inv(\sigma)$. Indeed, $F ^{-1}$ is well-defined since
$gx=g'x'$ for $x,x'\in \J_G\inv(\sigma)$ and $g,g'\in G$ implies
$g\inv g'\in G_\sigma$ by definition of $G_\sigma$ and hence
$[g',x']_{G_\sigma}=[g'(g\inv g')\inv,(g\inv g')
x']_{G_\sigma}=[g,x]_{G_\sigma}$. We have obviously 
$F\circ F^{-1}=\Id_{G\cdot \J_G\inv(\sigma)}$ and 
$F^{-1}\circ F=\Id_{G\times_{G_\sigma}\J_G\inv(\sigma)}$.

It remains hence to show that $F$ and $F^{-1}$ are smooth functions.
We use the commutative diagram
\begin{displaymath}
\begin{xy}
\xymatrix{
G\times\J_G\inv(\sigma)\ar[d]_{\pi_{G_\sigma}}\ar[drrr]^{\Phi\an{G\times\J_G\inv(\sigma)}} \ar[dr]_{\tilde{ \Phi}}& &&\\
G\times_{G_\sigma}\J_G\inv(\sigma)\ar[r]_F&G\cdot \J_G\inv(\sigma)\ar[rr]_{ \iota_{G, \sigma}}&& M
}
\end{xy}
\end{displaymath}
for the smoothness of $F$. Recall that $G \cdot \mathcal{J}_G ^{-1}
(\sigma) $ is an initial submanifold of $M$ because it is a leaf of the 
integrable tangent distribution $\mathcal{D}_G + \mathcal{V}$, that is,  
the inclusion $\iota_{G, \sigma} $ is regular.  The map
$\Phi\an{G\times\J_G\inv(\sigma)}=\Phi\circ(\Id_G\times \iota_\sigma):G
\times \J_G\inv(\sigma) \rightarrow M$ is smooth. Since its image is $G 
\cdot \mathcal{J}_G ^{-1}( \sigma)$, the map $\tilde{ \Phi}: G \times 
\mathcal{J}_G ^{-1}( \sigma) \rightarrow G \cdot  \mathcal{J}_G ^{-1}
( \sigma)$ defined by $\tilde{ \Phi}(g, m) : = \Phi(g, m) $ for all $g \in G $, 
$m \in \mathcal{J}_G ^{-1}( \sigma)$, is well defined and it is smooth 
because $\iota_{G, \sigma} $ is regular. Therefore $F$ is smooth by the
properties of the quotient map $\pi_{G_\sigma}$.

To show that $F ^{-1}$ is smooth we shall prove that $\dim\left(G \times _{G _\sigma} 
\mathcal{J}_G ^{-1}( \sigma)\right) 
= \dim \left(G \cdot \mathcal{J}_G ^{-1}( \sigma) \right)$ and that 
\begin{equation}
\label{injectivity}
T_{\pi_{G_ \sigma}(g, m)} \left( \iota_{G, \sigma} \circ F\right) : 
T_{\pi_{G_ \sigma}(g, m)} 
\left(G \times _{G _\sigma} \mathcal{J}_G ^{-1}( \sigma)\right) 
\rightarrow T_{\iota_{G, \sigma}(F(g,m))} M 
\end{equation}
 is injective. Indeed, if this is known, then $F $ is a bijective smooth map which is a local diffeomorphism, hence a diffeomorphism.

Since $\mathcal{J}_G ^{-1}( \sigma) $ is a leaf of $\mathcal{D}_G $, we 
have $\dim \mathcal{J}_G ^{-1}( \sigma) = \dim \mathcal{D}_G(m) $, for 
any $m \in \mathcal{J}_G ^{-1}( \sigma)$. Therefore, since the $G _\sigma
$-action on $G \times \mathcal{J}_G ^{-1}( \sigma)$ is free and proper, 
using \eqref{dim_g_sigma}, we have 
\begin{align*}
\dim \left(G \times _{G _\sigma} 
\mathcal{J}_G ^{-1}( \sigma)\right) 
&= \dim \mathcal{D}_G(m) + \dim G - \dim  G_\sigma \\
&\overset{\eqref{dim_g_sigma}}
= \dim \mathcal{D}_G(m) + \dim G - \dim ( \mathcal{V}(m) \cap 
\mathcal{D}_G(m)) - \dim G_m \\
& = \dim ( \mathcal{D}_G(m) + \mathcal{V}
(m) )- \dim \mathcal{V}(m)  + \dim G  - \dim G_m \\
&= \dim ( \mathcal{D}_G(m) 
+ \mathcal{V}(m) ) = \dim \left(G \cdot \mathcal{J}_G ^{-1}( \sigma)\right).
\end{align*}
Next we show the injectivity of \eqref{injectivity}. Let $T_{(g,n)} \pi_{G_ \sigma} (v_g, w_n) \in T_{\pi_{G_ \sigma}(g, n)} \left(G \times _{G _\sigma} \mathcal{J}_G ^{-1}( \sigma)\right)$ be such that $T_{\pi_{G_ \sigma}(g, n)} F \left(T_{(g,n)} \pi_{G_ \sigma} (v_g, w_n) \right) = 0$. From the diagram it follows that $0 = T_{(g,n)} \Phi(v _g, w _n) = T_g \Phi^n(v _g) + T_n \Phi _g( w _n)$, where $\Phi^n(h) : = h \cdot n $, for all $h \in G $. Therefore, choosing $\xi\in \mathfrak{g}$ such that  $v_g = T_e L_g \xi$, where $L_g(h) : = gh$, for all $h \in G $,  we have
\begin{align*}
w_n = - T_{g \cdot n} \Phi_{ g ^{-1}} \left(T_g \Phi^n( T_e L_g \xi) \right)
= -T_e \left(\Phi^n \circ L_{ g ^{-1}} \circ L _g\right)( \xi) = -\xi_M(n).
\end{align*}
Hence $w _n = - \xi_M(n) \in \mathcal{V}(n) \cap \mathcal{D}_G (n) $ which implies that $\xi \in \mathfrak{g}_\sigma$ by \eqref{dim_g_sigma}. Thus $\exp ( t \xi) \in G _\sigma$ for small $|t| $ and we get
\[
T_{(g, n)} \pi_{G_ \sigma}(v_g, w_n) = \left.\frac{d}{dt}\right|_{t=0} \pi_{G_ \sigma} ( g \exp ( t \xi), \exp (- t \xi) \cdot  n )  = 0 
\]
which proves the injectivity of \eqref{injectivity}.
\end{proof}

A leaf  $G\cdot \J_G\inv(\sigma)$  of $\D_G+\V$ is contained in $M_{(H)}$,
where $H\subseteq G$ is the compact subgroup such that
$\J_G\inv(\sigma)\subseteq M_H$. The induced action of $G$ on $G\cdot
\J_G\inv(\sigma)$ has hence conjugated isotropy subgroups. Using the fact that
the topology on $G\cdot \J_G\inv(\sigma)$ is stronger that
the topology induced on it by the topology of $M$, it is easy
to show that the action of $G$ on $G\cdot \J_G\inv(\sigma)$ is proper.

We have the following \emph{Dirac Optimal Orbit Reduction Theorem}, 
which is proved in the same manner as Theorems \ref{red_theorem1} 
and \ref{red_theorem2}. 

\begin{theorem}\label{red_theorem3}
Let $\sigma \in M/\mathcal{D}_G$. The
orbit space $M^G_\sigma:=\left(G\cdot\mathcal{J}_G^{-1}(\sigma)\right)/G$ 
is a regular quotient manifold such that  the projection 
$\pi:G\cdot\mathcal{J}^{-1}(\sigma)\to M^G_\sigma$ is a smooth 
submersion. Define 
$\omega^G_\sigma \in \Omega^2\left(M^G_\sigma\right)$ by 
\begin{equation}\label{presymp3}
\left({\pi}^*\omega^G_{\sigma}
\right)(m)(v_m,w_m)=\alpha(Y)\left(\iota_{\sigma,G}(m)\right)
=-\beta(X)\left(\iota_{\sigma,G}(m)\right)
\end{equation}
for any $m \in \mathcal{J}^{-1}(\sigma)$ and any $X,Y\in
\Gamma(\mathcal{D}_G)$ defined on an open set around 
$\iota_{\sigma,G}(m)$, where  $(X,\alpha), 
(Y,\beta) \in \Gamma(\mathsf{D}\cap(\T_G\oplus\V_G^\circ))$ are such that 
$T_m\iota_{\sigma,G} v_m=(X+V_1)(\iota_{\sigma,G}(m))$ and
$T_m\iota_{\sigma,G} w_m=(Y+V_2)(\iota_{\sigma,G}(m))$ for some 
smooth sections  $V_1,V_2\in\Gamma(\V)$.

Then $\left(M^G_{\sigma},\omega^G_{\sigma}\right)$ is a presymplectic 
manifold. The pair
$(M_{\sigma}^G,\mathsf{D}^G_{\sigma})$ is called the 
\emph{Dirac optimal orbit reduced space} of 
$(M,\mathsf{D})$ at $\sigma$, where $\mathsf{D}^G_{\sigma}$ 
is the graph of the presymplectic form $\omega^G_{\sigma}$.
\end{theorem}

\subsection{Comparison of the three methods}
We  show next that if the hypotheses of Theorems 
\ref{red_theorem1}, \ref{red_theorem2}, \ref{red_theorem3} 
hold, then the three methods yield the same reduced objects.

\begin{theorem}
Choose  $m\in M$ and set $\sigma:=\J_G(m)$, $\rho:=\J(m)$. Assume 
that the three optimal reduced Dirac manifolds  
$\left(M_\sigma,\mathsf D_\sigma\right)$, $\left(M_{(\rho)},\mathsf D_{(\rho)}
\right)$, and $\left(M_{\sigma}^G,\mathsf{D}^G_{\sigma}\right)$
are defined. The reduced presymplectic spaces $\left(M_\sigma,
\omega_\sigma\right)$, $\left(M_{(\rho)},\omega_{(\rho)}\right)$, and 
$\left(M^G_\sigma,\omega^G_\sigma\right)$ are presymplectomorphic.
\end{theorem}

\begin{proof}
Define the maps 
\begin{equation*}
\begin{array}{lccc}
\Psi:&M^G_\sigma&\longrightarrow& M_\sigma\\
&\left(\pi\circ\pi_{G_\sigma}\right)(g,m)&\longmapsto &\pi_\sigma(m)
\end{array}
\quad\text{ and }\quad
\begin{array}{lccc}
\Theta:&M_\sigma& \longrightarrow & M^G_\sigma\\
&\pi_{\sigma}(m)&\longmapsto &\left(\pi\circ\iota_\sigma^G\right)(m)
\end{array}
\end{equation*}
by the following commutative diagrams.
\begin{displaymath}\begin{xy}
\xymatrix{ 
& G\times \mathcal{J}_G^{-1}(\sigma)\ar[ld]_{\pi_{G_\sigma}}\ar[rd]^{{\rm p}_2}&  \\
G\cdot\mathcal{J}_G^{-1}(\sigma)\ar[d]_{\pi}
& & \mathcal{J}_G^{-1}(\sigma)\ar[d]^{\pi_\sigma}\\
M_\sigma^G\ar[rr]_{\Psi}&&M_\sigma
}
\end{xy}
\qquad \qquad 
\begin{xy}
\xymatrix{ &\\
\mathcal{J}_G^{-1}(\sigma) \ar[r]^{\iota^G_{\sigma}}\ar[d]_{\pi_\sigma}
& G\cdot\mathcal{J}_G^{-1}(\sigma)\ar[d]^\pi \\
M_\sigma\ar[r]_{\Theta}&M_\sigma^G\\
}
\end{xy}
\end{displaymath}
We use the diagrams to check that the maps $\Psi$ and $\Theta$ are 
inverses of each other and hence both bijective. Indeed,  for all 
$(g,m)\in G\times\J^{-1}_G(\sigma)$ we have
\begin{align*}
(\Theta\circ\Psi)\big((\pi\circ\pi_{G_\sigma})(g,m)\big)
&=\Theta(\pi_\sigma(m))
=\left(\pi\circ\iota_\sigma^G\right)(m)=\pi(g\cdot m)
=\left(\pi\circ\pi_{G_\sigma}\right)(g,m),
\end{align*}
and for all $m\in \J^{-1}_G(\sigma)$,
\begin{align*}
(\Psi\circ\Theta)(\pi_\sigma(m))=\Psi\left(\pi\circ\iota^G_\sigma(m)\right)
=\Psi\left((\pi\circ\pi_{G_\sigma})(g,m)\right)
=\pi_\sigma(m).
\end{align*}
The diagrams are also used to show that both maps are smooth.
The equality 
\[\Psi\circ\pi\circ\pi_{G_\sigma}=\pi_\sigma\circ {\rm p}_2\]
shows that $\Psi$ is smooth, since $\pi_\sigma\circ {\rm p}_2$ is smooth and
$\pi\circ\pi_{G_\sigma}$ is a smooth open map. The equality 
\[ \Theta\circ\pi_\sigma=\pi\circ\iota_\sigma^G\]
shows that $\Theta$ is smooth since $\pi_\sigma$ is a smooth open map and
$\pi\circ\iota_\sigma^G$ is smooth.

We define in the same manner the maps
\begin{equation*}
\begin{array}{lccc}
\Lambda:&M_{(\rho)}&\longrightarrow & M^G_\sigma\\
&\pi_{(\rho)}(m)&\longmapsto &(\pi\circ\iota)(m)
\end{array}
\quad\text{ and }\quad
\begin{array}{lccc}
\Phi:&M_\sigma& \longrightarrow& M_{(\rho)}\\
&\pi_{\sigma}(m)&\longmapsto 
&\left(\pi_{(\rho)}\circ\iota_{\sigma,\rho}\right)(m)
\end{array}
\end{equation*}
by the following commutative diagrams.

\begin{displaymath}
\begin{xy}
\xymatrix{ 
 & M   &  \\
\mathcal{J}^{-1}(\rho)\ar[d]_{\pi_{(\rho)}}\ar[ur]^{\inc}\ar[rr]^{\iota}
&&G\cdot\mathcal{J}_G^{-1}(\sigma)\ar[ul]_{\iota_{\sigma,G}}\ar[d]^{\pi}\\
M_{(\rho)}\ar[rr]_{\Lambda}&&M_\sigma^G
}
\end{xy}
\qquad
\begin{xy}
\xymatrix{ 
 & M   &  \\
\mathcal{J}_G^{-1}(\sigma)\ar[d]_{\pi_\sigma}\ar[ur]^{\incg}\ar[rr]^{\iota_{\sigma,\rho}}
&&\mathcal{J}^{-1}(\rho)\ar[ul]_{\inc}\ar[d]^{\pi_{(\rho)}}\\
M_\sigma\ar[rr]_{\Phi}&&M_{(\rho)}
}
\end{xy}
\end{displaymath}

We have then the commutative diagram
\begin{displaymath}
\begin{xy}
\xymatrix{ 
M_{(\rho)}\ar[dr]_{\Lambda} &  &M_\sigma\ar[ll]_\Phi\ar[dl]^{\Theta}  \\
&M^G_\sigma\ar[ur]^{\Psi}&
}
\end{xy}
\end{displaymath}
which shows that $\Phi$ and $\Lambda$ are bijective. We have
$\Lambda^{-1}=\Phi\circ\Psi$ and $\Phi^{-1}=\Psi\circ\Lambda$.
Thus, we have only to show that $\Lambda$ and $\Phi$ are smooth. 
But using the commutative diagrams, we get 
$\Lambda\circ\pi_{(\rho)}=\pi\circ\iota$. In addition, using
$\iota_{\sigma,G}\circ\iota=\iota_{(\rho)}$ and smoothness of
$\iota_{(\rho)}$, we conclude the smoothness of
$\iota$ since $ \iota_{\sigma,G}$ is a regular immersion. The 
map $\pi\circ\iota$ is consequently smooth and  $\Lambda$ is 
thus smooth because $\pi_{(\rho)}$ is a smooth open map. 

Analogously, we have $\Phi\circ\pi_\sigma=\pi_{(\rho)}\circ\iota_{\sigma,\rho}$. An argument similar to the one above
shows that the inclusion $\iota_{\sigma,\rho}$ is smooth.
Thus, $\Phi$ is smooth, using the fact that $\pi_\sigma$ is a
smooth open map.

Finally we prove  the equalities 
\[
\Phi^*\omega_{(\rho)}=\omega_\sigma,
\quad \text{and }\quad \Lambda^*\omega_\sigma^G=\omega_{(\rho)}
\]
which immediately imply that $\Theta$ preserves the presymplectic forms:
\[\Theta^*\omega_\sigma^G=(\Psi^{-1})^*\omega_\sigma^G
=(\Lambda\circ\Phi)^*\omega_\sigma^G=\Phi^*\Lambda^*\omega_\sigma^G
=\Phi^*\omega_{(\rho)}=\omega_\sigma.
\]
Choose $m\in\J_G\inv(\sigma)$ and vectors $v,w\in T_m\J_G\inv(\sigma)$. 
Then there exist $G$-invariant descending sections $(X,\alpha)$ and $(Y,\beta)$ 
of $\mathsf{D}$ such that $T_m\incg v=X(\incg(m))$ and $T_m\incg
w=Y(\incg(m))$.
We have 
\begin{align*}
\left(\pi_\sigma^*\Phi^*\omega_{(\rho)}\right)(m)(v,w)
&=
\left(\iota_{\sigma,\rho}^*\pi_{(\rho)}^*\omega_{(\rho)}\right)(m)(v,w)
=\left(\pi_{(\rho)}^*\omega_{(\rho)}\right)(\iota_{\sigma,\rho}(m))
(T_m\iota_{\sigma,\rho}v,T_m\iota_{\sigma,\rho}w).
\end{align*}
Since 
\[T_{\iota_{\sigma,\rho}(m)}\inc T_m\iota_{\sigma,\rho}
v=T_m\incg v=X(\incg(m))=X((\inc\circ\iota_{\sigma,\rho})(m))\]
and
\[T_{\iota_{\sigma,\rho}(m)}\inc T_m\iota_{\sigma,\rho}
w=T_m\incg v=Y(\incg(m))=Y((\inc\circ\iota_{\sigma,\rho})(m)),\]
formula \eqref{presymp1} and the fact that $(X,\alpha)$ and $(Y,\beta)$ are descending
sections of $\mathsf{D}$, imply
\begin{align*}
\left(\pi_\sigma^*\Phi^*\omega_{(\rho)}\right)(m)(v,w)
&=\left(\alpha(Y)\circ\inc\right)(\iota_{\sigma,\rho}(m))
=\left(\alpha(Y)\circ\incg\right)(m)\overset{\eqref{presymp2}}
=(\pi_\sigma^*\omega_\sigma)(m)(v,w),
\end{align*}
i.e., $\pi_\sigma^*\Phi^*\omega_{(\rho)}=\pi_\sigma^*\omega_\sigma$. Since $\pi_\sigma$ 
is a smooth surjective submersion, the equality
$\Phi^*\omega_{(\rho)}=\omega_\sigma$ is proved.
\medskip

Next, we prove the equality $\Lambda^*\omega_\sigma^G=
\omega_{(\rho)}$. Choose $m\in\J\inv(\rho)$ and $v,w\in
T_m\J\inv(\rho)$.
Then there exist descending sections $(X,\alpha)$ and $(Y,\beta)$ of
$\mathsf{D}$ defined on a neighborhood of $\inc(m)$ in $M$ such that 
$T_m\inc v=X(\inc(m))$ and $T_m\inc
w=Y(\inc(m))$. Since $X$ and $Y$ are descending vector fields, they
can be written $X=X^G+V$ and $Y=Y^G+W$ with $X^G,
Y^G\in\mx(M)^G$ and $V, W\in\Gamma(\V)$.
Assume that $(X,\alpha)$ and $(Y,\beta)$ are defined on a whole tube $U$
for the action of $G$ at $\inc(m)$; otherwise, we multiply $(X,\alpha)$ and
$(Y,\beta)$ with a bump function that is equal to $1$ on a
neighborhood $U_1\subsetneq U$ of $\inc(m)$, and equal to $0$ outside from
a neighborhood $U_2$ of $\inc(m)$ such that $U_1\subsetneq U_2\subsetneq U$. 

Consider the $G$-invariant averages $(X_G,\alpha_G), 
(Y_G,\beta_G)\in\Gamma(\mathsf{D})$ of $(X,\alpha)$ and $(Y,\beta)$ at $\inc(m)=:n$. 
We have, with $H=G_n$,  
\[X_G(n)=\int_H(T_n\Phi_{h^{-1}}X^G(n)+T_n\Phi_{h^{-1}}V(n))dh
=X^G(n)+V_G(n)=X(n)+(V_G-V)(n)
\]
and 
\[\alpha_G(n)=\int_H\alpha(h\cdot n)\circ T_n\Phi_h dh=\alpha(n).
\]
Hence, the sections $(X_G,\alpha_G)$ and $(Y_G,\beta_G)$ are
$G$-invariant descending sections of $\mathsf D$ such that
\begin{align*}
T_{\iota(m)}\iota_{\sigma,G}(T_m\iota\, v)
&=T_m\inc v=X(\inc(m))\\
&=X(n)=X_G(n)+(V-V_G)(n)=(X_G+(V-V_G))
\bigl(\iota_{\sigma,G}(\iota(m))\bigr)
\end{align*}
and analogously 
\[T_{\iota(m)}\iota_{\sigma,G}(T_m\iota\, w)
=(Y_G+(W-W_G))\bigl(\iota_{\sigma,G}(\iota(m))\bigr).\]
We get, using this and definitions \eqref{presymp3} and \eqref{presymp1}
\begin{align*}
\left(\pi_{(\rho)}^*\Lambda^*\omega_\sigma^G\right)(m)(v,w)
&=\left(\iota^*\pi^*\omega_\sigma^G\right)(m)(v,w)
=\left(\pi^*\omega_\sigma^G\right)(\iota(m))(T_m\iota\,
v,T_m\iota\, w)\\
&\overset{\makebox[0pt][c]{\tiny\eqref{presymp3}}}{=}-\beta_G(X_G)\bigl(\iota_{\sigma,G}(\iota(m))\bigr)
=-\beta_G(X_G+(V-V_G))\bigl(\inc(m)\bigr)\\
&=-\beta_G(n)(X_G(n)+(V-V_G)(n))
=-\beta(n)(X(n))\\
&=-\beta(X)\bigl(\inc(m)\bigr)\overset{\small\eqref{presymp1}}
=\left(\pi_{(\rho)}^*\omega_{(\rho)}\right)(m)(v,w),
\end{align*}
that is,  $\pi_{(\rho)}^*\Lambda^*\omega_\sigma^G
=\pi_{(\rho)}^*\omega_{(\rho)}$.
Since $\pi_{(\rho)}$ is a smooth surjective
submersion, we conclude
$\Lambda^*\omega_\sigma^G=\omega_{(\rho)}$.
\end{proof}

\subsection{Reduction of dynamics}
In this subsection, we write $\pois$ for the Poisson bracket
$\pois_{\mathsf D}$ on the admissible functions 
of $(M,\mathsf D)$. We assume that the hypotheses of the preceding subsections
are satisfied and study the reduction of dynamics.

\begin{theorem}
Let $(M,\mathsf D)$ be a smooth Dirac manifold with a proper Lie group action, 
such that the orbit optimal momentum map $\mathcal J$ is defined. 
Choose $m\in M$ such that $G_{(\rho)}$ 
acts properly on $\J\inv(\rho)$, where $\rho=\J(m)$.
Let $h\in C^\infty(M)^G$ be a $G$-invariant 
admissible smooth function on
$M$ defined on a neighborhood $U$ of $m$ in $M$. Then:
\begin{enumerate}
\item There
exists a $G$-invariant vector field $X_h$ defined on $U$ such that 
$(X_h,\dr h)$ is a 
($G$-invariant descending) section of $\mathsf{D}$. 
\item The flow $\phi$ of $X_h$ commutes with the $G$-action and  leaves
  $\J\inv(\rho)$ invariant. Thus it restricts to a flow $\tilde \phi$
on $\J\inv(\rho)$, that is, with $\phi_t\circ\inc=\inc\circ\tilde \phi_t$
for all 
  $t\in \mathbb{R}$ where the left hand side  is defined.
The flow $\tilde \phi$ commutes then with the $G_{(\rho)}$-action
 and induces  therefore a flow $\phi^{(\rho)}$ on
  $M_{(\rho)}$ uniquely determined by the relation $\pi_{(\rho)}\circ
  \tilde \phi_t= \phi_t^{(\rho)}\circ \pi_{(\rho)}$ for all 
  $t\in \mathbb{R}$ where the left hand side  is defined. 
\item The vector field $X_h^{(\rho)}$ defined by the flow $\phi^{(\rho)}$ on
  $M_{(\rho)}$ is a section of $\mathsf{G}_1^{(\rho)}$; more precisely, we
  have 
\[ (X_h^{(\rho)},\dr h_{(\rho)})\in\Gamma(\mathsf{D}_{(\rho)}),\]
where $h_{(\rho)}$ is the smooth function on $M_{(\rho)}$ defined by 
$h_{(\rho)}\circ\pi_{(\rho)}=h\circ\iota_{(\rho)}$.
\item Let $k\in C^\infty(M)^G$ be another admissible function, and
  $\{\cdot,\cdot\}_{(\rho)}$ the bracket on admissible functions on
  $M_{(\rho)}$ defined by $\mathsf{D}_{(\rho)}$. Then we have 
\[\left(\{h,k\}\right)_{(\rho)}=\{h_{(\rho)},k_{(\rho)}\}_{(\rho)},\]
where the function $\left(\{h,k\}\right)_{(\rho)}$ is defined by 
\[\left(\{h,k\}\right)_{(\rho)}\circ\pi_{(\rho)}=\{h,k\}\circ\iota_{(\rho)}.\]
This makes sense because $\{h,k\}=-X_h(k)$ is $G$-invariant. 
\end{enumerate}
\end{theorem}

\begin{proof}
Since $h$ is admissible, there exists a vector field $X$ such that $(X,\dr h)$
is a section of $\mathsf{D}$. The one-form $\dr h$ is a $G$-invariant section 
of $\V^\circ$. Consider the $G$-invariant average $(X_G,(\dr h)_G)$ of 
$(X,\dr h)$ in a tube for the action of $G$ centered at the point $m$. Since 
$\dr h$ is $G$-invariant, we have $\dr h=(\dr h)_G$. Set $X_h:=X_G$; then 
$(X_h,\dr h)$ is a $G$-invariant descending section of $\mathsf{D}$.

Hence we have $X_h\in\Gamma(\D_G)\subset\Gamma(\D)$
and, consequently, the leaves $\J\inv(\rho)$ and 
$\J_G\inv(\J_G(m))$
of $\D$ and $\D_G$ are left invariant by the
flow $\phi$ of $X_h$; thus we can define the restriction $\tilde \phi_t$ of 
$\phi_t$ to $\J\inv(\rho)$ by $\iota_{(\rho)}\circ\tilde\phi_t=\phi_t\circ
\iota_{(\rho)}$.  Since $X_h$ is $G$-invariant, $\phi_t$
commutes with the $G$-action and, consequently,  $\tilde \phi_t$
commutes with the $G_\rho$-action. Define $\phi^{(\rho)}_t$ on $M_{(\rho)}$
by $\pi_{(\rho)}\circ \tilde \phi_t= \phi_t^{(\rho)}\circ 
\pi_{(\rho)}$ for all $t$ where the left hand side  is defined. 

Let $X_h^{(\rho)}$ be the vector field defined by the flow
$\phi^{(\rho)}$. Then we have 
\[
T\pi_{(\rho)}\tilde X_h=X_h^{(\rho)}\circ \pi_{(\rho)},
\]
where $\tilde X_h$ is the vector field on $\J\inv(\rho)$ that is
$\iota_{(\rho)}$-related to $X_h$ (that is, $\tilde X_h$ is the vector field
defined by the flow $\tilde \phi$ on $\J\inv(\rho)$).
For any $n\in \J\inv(\rho)$ where $h\circ\iota_{(\rho)}$ is defined and any
$\tilde X\in \mx(\J\inv(\rho))$ with flow $\phi^{\tilde X}$ defined on a
neighborhood of $n$, there exists $X\in\Gamma(\D)$ with flow $\phi^X$ such that
$T\iota_{(\rho)}\circ \tilde X=X\circ\iota_{(\rho)}$. We compute
\begin{align*}
\left(\ip{X_h^{(\rho)}}\omega_{(\rho)}\right)
\left(\pi_{(\rho)}(n)\right)(T_n\pi_{(\rho)}\tilde X(n))
&=\omega_{(\rho)}({\pi_{(\rho)}(n)})
\left(X_h^{(\rho)}(\pi_{(\rho)}(n)),T_{n}\pi_{(\rho)}\tilde X(n)\right)\\
&=\omega_{(\rho)}\left({\pi_{(\rho)}(n)}\right)
\left(T_{n}\pi_{(\rho)}\tilde X_h(n),T_n\pi_{(\rho)}\tilde X(n) \right)\\
&=(\pi_{(\rho)}^*\omega_{(\rho)})(n)\left(\tilde X_h(n),\tilde X(n) \right)\\
&\overset{\eqref{presymp1}}=\mathbf{d}h\left(\iota_{(\rho)}(n)\right)(X(\iota_{(\rho)}(n)))
=\mathbf{d}h\left(\iota_{(\rho)}(n)\right)(T_n\iota_{(\rho)}\tilde X(n))\\
&=\mathbf{d}(\iota_{(\rho)}^*h)(n)(
\tilde{X}(n))
=\mathbf{d}(\pi_{(\rho)}^*h_{(\rho)})(n)(\tilde X(n))\\
&=\mathbf{d}h_{(\rho)}\left(\pi_{(\rho)}(n)\right)(T_n\pi_{(\rho)}\tilde X(n)).
\end{align*}
Hence, we have shown the equality $\ip{X_h^{(\rho)}}\omega_{(\rho)}=\dr
h_{(\rho)}$ and, by the definition of $\mathsf{D}_{(\rho)}$, we get 
\[\left(X_h^{(\rho)},\dr h_{(\rho)}\right)\in\Gamma\left(\mathsf{D}_{(\rho)}\right),\]
which yields also the fact that $ h_{(\rho)}\in C^\infty\left(M_{(\rho)}\right)$ is
admissible.

We show the last statement of the theorem in the same manner. Let $X_k$ be the
$G$-invariant vector field such that $(X_k,\dr k)$ is a $G$-invariant
descending section of $\mathsf{D}$. Then we have
\begin{align*}
\{h_{(\rho)},k_{(\rho)}\}_{(\rho)}\circ\pi_{(\rho)}
&=-\left(\dr h_{(\rho)}\left(X_k^{(\rho)}\right)\right)\circ\pi_{(\rho)}
=-\left(\omega_{(\rho)}\circ {\pi_{(\rho)}}\right)
\left(X_h^{(\rho)}\circ\pi_{(\rho)},X_k^{(\rho)}\circ\pi_{(\rho)}\right)\\
&=-\left(\pi_{(\rho)}^*\omega_{(\rho)}\right)(\tilde X_h,\tilde X_k)
=-(\dr h)(X_k)\circ\iota_{(\rho)}=\{h,k\}\circ\iota_{(\rho)}\\
&=:(\{h,k\})_{(\rho)}\circ\pi_{(\rho)}.
\end{align*}
\end{proof}
\begin{remark}
As can be easily seen from this proof, an analogous theorem is true
for the optimal reduced Dirac spaces 
$(M_\sigma,\mathsf D_\sigma)$ and $(M_\sigma^G,\mathsf D_\sigma^G)$
if they are defined.
\end{remark}

\section{Comparison of Optimal and standard Dirac reduction}
\label{sec:comparison}
In this section, we compare the reduced Dirac manifolds 
obtained by the standard reduction method in 
\cite{JoRaSn11} with those obtained by optimal reduction, under the
assumption that all necessary conditions on the smooth generalized 
distribution $\mathsf D\cap(\T\oplus\V_G^\circ)$ are satisfied.
Since we know that the three optimal reduction methods are equivalent if 
they are all possible, we only consider the optimal (orbit type) point reduction
method in this section.

Assume that $\mathsf{D}\cap(\T\oplus\V_G^\circ)$ is spanned pointwise 
by its descending sections and that $\D$ is integrable. Let $\J:M\to M/\D$
be the corresponding orbit optimal momentum map.
Since  $\mathsf{D}\cap(\T\oplus\V_G^\circ)$ is spanned pointwise by the
descending sections of $\mathsf{D}$, the Dirac structure $\mathsf{D}$ 
induces a  Dirac structure on each stratum of the quotient space $M/G$
(Theorems \ref{singred} and \ref{singred2}). 
The following theorem gives the relation between the strata of $M/G$ 
endowed with these reduced structures and the  reduced manifolds
$(M_{(\rho)},\mathsf{D}_{(\rho)})$ given by the (orbit type) Dirac optimal 
reduction theorem (under the assumption that $G_{(\rho)}$ acts properly on
$\mathcal{J}^{-1}(\rho)$). 

\begin{theorem}\label{presympl} Let $m\in\mathcal{J}^{-1}(\rho)$ for some 
$\rho\in M/\D$ such that $G_{(\rho)}$ acts properly on $\mathcal{J}^{-1}
(\rho)$. Then, if $P$ is the connected component through $m$ of the orbit
type manifold $M_{(G_m)}$, we have $\mathcal{J}^{-1}(\rho)\subseteq
P$. The reduced manifold $M_{(\rho)}$ is diffeomorphic to the 
presymplectic leaf $\bar N$ through  $\pi(m)$ of the reduced Dirac 
manifold $(\bar P,\mathsf{D}_{\bar P})$, where $\bar{P}=\pi(P)$ is the 
stratum of $M/G$ through $\pi(m)$, via the map
  $\Pi: M_{(\rho)}\to \bar N$, $\pi_{(\rho)}(x)\mapsto (\pi\circ
  \iota_{(\rho)})(x)$. Furthermore, $\Pi^*\omega_{\bar N}=\omega_{(\rho)}$, where $\omega_{\bar N}$ is the presymplectic form on $\bar N$.
\end{theorem}

\begin{proof} 
We begin by showing
that the map $\Pi$ is well-defined. Let $x, y \in
\mathcal{J}^{-1}(\rho)$ be such that $\pi_{(\rho)}(x)=\pi_{(\rho)}(y)$. Then there
exists $g\in G_{(\rho)}\subseteq G$ such   that $\Phi^{(\rho)}_g(x)=y$ which implies 
that $\Phi_g(\iota_{(\rho)}(x))=\iota_{(\rho)}(y)$ and 
$\pi(\iota_{(\rho)}(x))=\pi(\iota_{(\rho)}(y))$.
Thus, it remains to show that $\pi(\iota_{(\rho)}(x))\in \bar{N}$. 
Since $x\in\mathcal{J}^{-1}(\rho)$, and
by definition of the integral leaves of $\mathcal{D}$, 
the points $\iota_{(\rho)}(m)$ and $\iota_{(\rho)}(x)$ can be joined by a broken path 
consisting of finitely many pieces of integral curves of descending
sections of $\mathcal{D}$ belonging to descending pairs of $\mathsf{D}$. 
To simplify the notation, we shall write in what
follows simply $x'$ for $\iota_{(\rho)}(x)$ and $m'$ for $\iota_{(\rho)}(m)$. 
Assume, without
loss of generality, that one such curve suffices,
i.e., that $x'=\phi_t(m')$, where $\phi$ is the flow of
a vector field $X\in\Gamma(\mathcal{D})$ for which
there exists $\alpha\in\Omega^1(M)$ such that  $(X,\alpha)$ is a descending
section of $\mathsf{D}$. Since
$X$ is a descending vector field, it can be written as a sum 
$X=V+X^G$ with $X\in\mx(M)^G$ and $V\in\Gamma(\V)$. Then
$[X^G,V]=0$ and we have $\phi_t=\phi_t^{G}\circ
\phi_t^{V}=\phi_t^{V} \circ \phi_t^{G}$, where
$\phi_t^{G}$ and $\phi_t^{V}$ are the flows of $X^G$ and
$V$, respectively. Let
$\bar{\phi}$ be the flow on $\bar{M}$ induced by $\phi $, i.e., $\pi\circ
\phi_s=\pi\circ \phi_s^{V} \circ \phi_s^{G}=\pi\circ
\phi_s^{G}=\bar{\phi}_s\circ \pi$ for all $s$. 
This flow $\bar{\phi}$ generates a vector field $\bar{X}$ on $\bar{M}$ such that
$X\sim_\pi \bar{X}$. Since $(X,\alpha)$ is 
a descending section of $\mathsf{D}\cap(\T\oplus\V^\circ_G)$, we know 
by the definition of the reduced Dirac structure on $\bar P$  that
\begin{equation}\label{sectionofredDirac}
(X_{\bar P},\alpha_{\bar P})\in\Gamma(\mathsf{D}_{\bar P}),
\end{equation}
 where $X_\bp$ is the
restriction to $\bar P$ of $\bar X$ and $\alpha_\bp$ is the restriction of
$\bar \alpha$ to $\bp$, the ``one-form''  $\bar \alpha\in\Omega^1(\bar M)$ being such that
$\pi^*\bar\alpha=\alpha$. Here, we know that $X_\bp\in\mx(\bp)$ because the
flow of $\bar X$ through points in $\bp$ remains in $\bp$. 
We have 
\[(\bar{\phi}_t\circ\pi)(m')=(\pi\circ\phi_t)(m')=\pi(x'),\]
 which yields, using \eqref{sectionofredDirac} and the fact that $\bar\phi_t\an{\bp}$ 
is the
flow of $X_\bp$, that $\pi(\iota_{(\rho)}(x))$ 
and $\pi(\iota_{(\rho)}(m))$ lie in the same presymplectic leaf $\bar{N}$ of 
$(\bar{P},\mathsf{D}_\bp)$. This concludes the proof that $ \Pi : M _\rho
\rightarrow  \bar{N}$ is well defined.

To prove that $\Pi$ is injective, choose $x, y\in\mathcal{J}^{-1}(\rho)$ 
such that $\pi(\iota_{(\rho)}(x))=\pi(\iota_{(\rho)}(y))$. 
Then there 
exists $g\in G$ satisfying \[\Phi_g(\iota_{(\rho)}(x))=\iota_{(\rho)}(y).\] 
This shows that $g\in G_\rho$ 
and $\Phi^{(\rho)}_g(x)=y$, so we get $\pi_{(\rho)}(x)=\pi_{(\rho)}(y)$. 

For the surjectivity of $\Pi$ choose $\pi(x)\in\bar{N}$ and assume, again 
without loss of generality, that
\[
\pi(x)=\phi^\bp_t(\pi(\iota_{(\rho)}(m))),
\] 
where  $\phi^\bp$ is the flow of some $X_\bp\in
\mx(\bar P)$, such that there exists $\alpha_\bp\in\Omega^1(\bar P)$
with 
$(X_\bp,\alpha_\bp)\in\Gamma(\mathsf{D}_\bp)$. Choose a descending 
section $(X,\alpha)\in\Gamma(\mathsf{D})$  and $(\bar{X},\bar\alpha)
\in\bar \D$  such that $X\sim_\pi \bar X$, $\alpha=\pi^*\bar\alpha$, and 
$(\bar X,\bar\alpha)\an{\bp}=(X_\bp,\alpha_\bp)$. 
The pairs $(X,\alpha)$ and $(\bar X,\bar\alpha)$ exist by the proof of the
reduction Theorem \ref{singred}. Then
the flows $\phi$ of $X$ and $\bar \phi$ of $\bar X$ satisfy $\pi\circ \phi_s
=\bar \phi_s\circ \pi$ for all
$s$ and $\phi_s$ restricts to 
$\mathcal{J}^{-1}(\rho)$ since $X$ is a descending
section of $\D$. If we define
$x'\in\J\inv(\rho)$ by $\iota_{(\rho)}(x')=\phi_t(\iota_{(\rho)}(m))$, we get,
using the fact that $\bar \phi_t\an{\bp}=\phi_t^\bp$,
\[\pi(\iota_{(\rho)}(x'))=(\pi\circ\phi_t)(\iota_{(\rho)}(m))
=(\bar \phi_t\circ\pi)(\iota_{(\rho)}(m))=\pi(x)\] 
and hence  $\Pi(\pi_{(\rho)}(x'))=\pi(x)$. 

Note that we have simultaneously shown that
$\pi(\mathcal{J}^{-1}(\rho))\subseteq \bar P$ is equal, as a set, to $\bar
N$. Moreover, we claim that the topology of $\bar N$ (which is in general
\emph{not} the relative topology induced from the topology on $\bar P$) 
is the quotient topology defined by the topology of $\mathcal{J}^{-1}(\rho)$, 
that is, a set is open in $\bar{N}$ if and only if its preimage under
$\pi\circ\iota_{(\rho)}$ is open in $\mathcal{J}^{-1}(\rho)$. This is proved in the following way.

Denote by $\bar{\mathsf{G}}_1\subseteq \mx(\bar M)$ the set of vector 
fields $\bar{X}$ on $\bar{M}$ such that there exists $\bar\alpha\in
\Omega^1(\bar M)$ with  $(\bar X,\bar\alpha)\in\bar\D$.
The presymplectic leaf $\bar{N}$ containing $\bar{m}$ can be seen as 
the accessible set of $\bar{\mathsf{G}}_1$ through $\bar m$, since 
$\bar{P}$  is the accessible set through $\bar m$ of the family of all 
vector fields on $\bar{M}$. The topology on $\bar{N}$ is the relative 
topology induced on $\bar{N}$ by a topology we call the
$\bar{\mathsf{G}}_1$-topology on $\bar M$: this is the strongest topology on
$\bar{M}$ such that all the maps
\begin{equation*}
\begin{array}{ccc}
U&\longrightarrow& \bar M\\
(t_1,\ldots,t_k)&\longmapsto &\left(\bar \phi_{t_1}^{\bar X_1}\circ\dots\circ
\bar{\phi}_{t_k}^{\bar{X}_k}\right)(\bar{m})
\end{array}
\end{equation*} 
are continuous, where $\bar{m}\in\bar{M}$, $\bar{\phi}_{t_i}^{\bar{X}_i}$ 
is the flow of a vector field $\bar{X}_i\in\bar{\mathsf{G}}_1$ for 
$i=1,\ldots, k$, and  $U\subseteq \mathbb{R}^k$ is an
appropriate open set in $\mathbb{R}^k$. 
In the same manner, because $\mathcal{J}^{-1}(\rho)$ is an accessible 
set of the family
\[
F:=\{X\in\mathfrak{X}(M)\mid\exists\alpha\in\Omega^1(M)\text{ such that } (X,\alpha)
\text{ is a descending section of } \mathsf{D}\},  
\]
the topology on $\mathcal{J}^{-1}(\rho)$ is  the relative topology
induced on $\mathcal{J}^{-1}(\rho)$ by the
topology we call the $\D$-topology on $M$: 
this is the strongest topology on
$M$ such that all the maps
\begin{equation*}
\begin{array}{ccc}
U&\longrightarrow& M\\
(t_1,\ldots,t_k)&\longmapsto &\left(\phi_{t_1}^{X_1}\circ\dots\circ
\phi_{t_k}^{ X_k}\right)(m)
\end{array}
\end{equation*} are continuous,
where $m\in M$, $\phi_{t_i}^{X_i}$ is the flow of a vector
field $X_i\in F$ for $i=1,\ldots, k$, and 
$U\subseteq \mathbb{R}^k$ is an
appropriate open set in $\mathbb{R}^k$.

Now our claim is easy to show, using the fact that for each section $X_\bp$
in $\bar{\mathsf{G}}_1$, there exists a descending section $(X,\alpha)$ of
$\mathsf{D}$ such that $X\sim_{\pi}\bar X$ and hence $\bar
\phi_t^{\bar X}\circ\pi=\pi\circ \phi_t^X$. Conversely, for each descending section
$(X,\alpha)$ of $\mathsf{D}$, the vector field  $\bar X$ satisfying
$X\sim_\pi\bar X$ is an element  of $\bar{\mathsf{G}}_1$ and we have $ \bar
\phi_t^{\bar X}\circ\pi=\pi\circ \phi_t^X$.
Hence, a map $f:\bar N\to Q$ is smooth if and only if
$f\circ(\pi\circ\iota_{(\rho)}):{\mathcal{J}^{-1}(\rho)}\to Q$
is smooth, where $Q$ is an arbitrary smooth manifold. 
Thus we have shown that $\bar{N} = \pi( \mathcal{J} ^{-1}( \rho)) \subseteq \bar{P}$ as topological spaces.

Finally, the smoothness of $\Pi$ and of its inverse $\Pi^{-1}:\bar{N}\to 
M_\rho$, $\pi(\iota_{(\rho)}(x))\mapsto \pi_\rho(x)$ follow from  the following 
commutative 
diagrams:
\begin{displaymath}
\begin{xy}
\xymatrix{ 
\mathcal{J}^{-1}(\rho) \ar[r]^{\iota_{(\rho)}}\ar[d]_{\pi_{(\rho)}}& M\ar[d]^{\pi} \\
M_{(\rho)} \ar[r]_{\iota_{\bar N}\circ\Pi}& \bar M}
\end{xy}
\qquad\qquad\qquad
\begin{xy}
  \xymatrix{ 
\mathcal{J}^{-1}(\rho)\ar[rd]^{\pi_{(\rho)}}  
\ar[d]_{\pi\circ\iota_{(\rho)}}   &               \\
      \bar N \ar[r]_{\Pi^{-1}}   &     M_{(\rho)}
                      }
\end{xy}
\end{displaymath}
Consider the first diagram.  Let
$\iota_{\bp,\bar M}:\bp\hookrightarrow \bar M$ and 
$\iota_{\bar N,\bp}:\bar N\hookrightarrow \bp$ be the inclusions.
Since $\pi\circ \iota_{(\rho)}$ is smooth, we have
automatically (by the quotient manifold structure on $M_{(\rho)}$) that
$\iota_{\bar{N}}\circ \Pi=\iota_{\bp,\bar M}\circ\iota_{\bar N,\bp}\circ\Pi$
is smooth.  Since $\bar{N}$ is an initial submanifold of $\bar{P}$ and 
$\bar{P}$ is a stratum of $\bar{M}$, the smoothness of 
$\Pi$ follows.  With the
considerations above and, because $\pi_{(\rho)}$ is smooth,
we get the smoothness of $ \Pi^{-1}$ with the second diagram.

Now we show that $\Pi$ is a \emph{presymplectomorphism}, i.e., 
$\Pi^*\omega_{\bar{N}}=\omega_\rho$. Let $\pi_\rho(x)\in M_\rho$, 
$x\in\mathcal{J}^{-1}(\rho)$, and $v,w\in T_x\mathcal{J}^{-1}(\rho)$,
i.e., 
we have $T_x\iota_{(\rho)} v$, $T_x\iota_{(\rho)}
w\in\mathcal{D}(\iota_{(\rho)}(x))$. Then find $X_\bp,Y_\bp 
\in \mx(\bp)$ and $\alpha_\bp, \beta_\bp \in 
\Omega^1(\bar{P})$ such that $(X_\bp,\alpha_\bp), (Y_\bp,\beta_\bp)
\in \Gamma(\mathsf{D}_\bp)$,  $T_x(\Pi\circ\pi_{(\rho)})v=
T_x(\pi\circ \iota_{(\rho)}) v=X_\bp (\pi(\iota_{(\rho)}(x)))$, and
$T_x(\Pi\circ\pi_{(\rho)})w=T_x(\pi\circ \iota_{(\rho)})
w=Y_\bp(\pi(\iota_{(\rho)}(x)))$. 
Choose $(\bar X,\bar \alpha), (\bar Y,\bar \beta)\in \bar\D$ and 
$X,Y\in\mx(M)$ such that $(X,\pi^*\bar\alpha), (Y,\pi^*\bar\beta)$ are
descending sections of $\mathsf{D}$,
$X\sim_\pi \bar{X}$, 
$Y\sim_\pi\bar{Y}$, and $(X_\bp,\alpha_\bp)$ and $(Y_\bp,\beta_\bp)$ are the
restrictions to $\bar P$ of $(\bar X,\bar \alpha)$ and $ (\bar Y,\bar \beta)$. 
Then we get
\begin{align*}
(\Pi^*\omega_{\bar{N}})(\pi_{(\rho)}(x)) \left( T_x\pi_{(\rho)} v,T_x\pi_{(\rho)} w \right) 
&={\omega_{\bar{N}}} \left((\Pi\circ \pi_{(\rho)})(x) \right) 
(T_x(\Pi\circ\pi_{(\rho)} )v,T_x(\Pi\circ\pi_{(\rho)}) w)\\
&={\omega_{\bar{N}}}((\pi\circ \iota_{(\rho)})(x))
\left(T_x(\pi\circ \iota_{(\rho)} )v,T_x(\pi\circ \iota_{(\rho)}) w\right) \\
&={\omega_{\bar{N}}}((\pi\circ \iota_{(\rho)})(x)) 
\left(X_\bp(\pi(\iota_{(\rho)}(x))),Y_\bp(\pi(\iota_{(\rho)}(x))) \right) \\
&=\alpha_\bp((\pi\circ \iota_{(\rho)})(x)) \left(Y_\bp(\pi(\iota_{(\rho)}(x))
\right) 
=\alpha(\iota_{(\rho)}(x))\left(Y(\iota_{(\rho)}(x)) \right) \\
&={\omega_{(\rho)}}(\pi_{(\rho)}(x)) \left(T_x\pi_{(\rho)} v,T_x\pi_{(\rho)} w\right), 
\end{align*}
where the last equality is the definition of $\omega_{(\rho)}$.
\end{proof}

\section{Examples}
\label{sec:examples}
\begin{example}\label{ex1}
We consider the example of the proper action $\Phi$ of $G:=\mathbb{S}^1\simeq 
\R/(2\pi\Z)$ 
on  $M:=\R^3$ given by  
\[
\Phi_ \alpha(x,y,z)=\alpha\cdot(x,y,z)=(x\cos\alpha-y\sin\alpha,
x\sin\alpha+y\cos\alpha,z).
\]
The orbit and isotropy types of this action coincide 
since the Lie group is Abelian. They are $P_1=\{0\}\times\{0\}\times\R$,
$P_1=M_{H_1}$ with $H_1=\mathbb{S}^1$, and $P_2=\R^3\setminus
P_1$, so $P_2=M_{H_2}$ with $H_2=\{e\}$. 
The orbit of a point $(x,y,z)\in\R^3$ is $\{(x',y',z')\in\R^3\mid
x'^2+y'^2=x^2+y^2\text{ and } z'=z\}$. Thus the reduced space $\bar{M}$
can be identified with $[0,+\infty)\times \R$ with the projection $\pi$ given by
\[
[0, \infty) \times \mathbb{R} \ni (\bar{x}, \bar{z}) : =\pi(x,y,z)=(x^2+y^2,z).
\] 
It is easy to compute, for each $\alpha\in\mathbb{S}^1$:
\begin{align*}
\Phi_\alpha^*(\partial_x)&=\cos\alpha\partial_x-\sin\alpha\partial_y,\quad
\Phi_\alpha^*(\partial_y)=\sin\alpha\partial_x+\cos\alpha\partial_y,\quad
\Phi_\alpha^*(\partial_z)=\partial_z
\end{align*}
and
\begin{align*}
\Phi_\alpha^*(\dr x)&=\cos\alpha\dr x-\sin\alpha\dr y,\quad
\Phi_\alpha^*(\dr y)=\sin\alpha\dr x+\cos\alpha\dr y,\quad
\Phi_\alpha^*(\dr z)=\dr z.
\end{align*}
Hence, the Dirac structure $\mathsf{D}$ given as the span of the sections 
\[ 
(\partial_x,\dr y), \quad (\partial_y,-\dr x), \quad (\partial_z,0)\]
is $\mathbb{S}^1$-invariant, that is, the Lie group 
$\mathbb{S}^1$ acts on $(M,\mathsf{D})$ in a Dirac manner.

In Example \ref{exap1}, we carried out
the explicit computation showing that
$$\T(m)=\T_G(m)=\erz_{\R}\{\partial_z,
x\partial_{x}+y\partial_{y},
x\partial_{y}-y\partial_{x}\}$$
for all $m=(x,y)\in\R^2$.
The set $\mathsf{D}\cap(\T\oplus\V_G^\circ)$ is equal to the 
set $\mathsf{D}\cap(\T_G\oplus\V_G^\circ)$, since the orbit type manifolds
coincide with the isotropy type manifolds, and  is spanned by the sections 
\[ 
(y\partial_x-x\partial_y,x\dr x+y\dr y) \quad \text{and} \quad 
(\partial_z,0)\quad \text{of} \quad  \mathsf{D}.   
\]
These sections are exact $G$-invariant descending sections of 
$\mathsf{D}$ and hence $\D_G=\D$ is completely integrable. Indeed, it 
is easy to see that since $\D$
is spanned pointwise by the vector fields  $\partial_z$ and $y\partial_x-x\partial_y$,
it is integrable and its leaves are the cylinders $M_r:=\{(x,y,z)\in\R^3\mid x^2+y^2=r\}$ for all
$r>0$ and the line $\{(0,0,z)\mid z\in\R\}=:M_0$. 

We identify the leaf space of $\D$ with the closed set $[0,\infty)$;
$\J(x,y,z)=x^2+y^2$ for all $(x,y,z)\in\R^3$. It is easy to see that
$G_r=\sfe$ for all $r\in [0,\infty)$. The Lie group $\sfe$ acts on $M_r$ by
rotations for $r>0$ and trivially for $r=0$. The reduced spaces are then all
lines with trivial presymplectic form.

\medskip

It has been shown in \cite{JoRaSn11}  that the Dirac structures on 
$\bar{M}_1=P_1/\sfe\simeq \R$ and $\bar{M}_2=P_2/\sfe\simeq \R^2$
are given by $\mathsf{D}_{\bar M_1}(\bar{z})=
\operatorname{span}\{(\partial_{\bar{z}},0)\}$ and
$\mathsf{D}_{\bar M_2}(\bar{x},\bar{z})= 
\operatorname{span}\{(\partial_{\bar{z}},0),
(0,\bar{x}\mathbf{d}\bar{x})\}$. Thus, the 
symplectic leaves are all lines with trivial Dirac structures and we recover 
the result of the correspondence theorem.

\end{example}

\begin{example}\label{ex2}
We consider here another example from \cite{JoRaSn11}. Let $M=\R^3\times\R^3$ with the 
(automatically proper) diagonal action of $G=\sfe\simeq \R/(2\pi\Z)$ on it, 
i.e., 
\[
\begin{array}{cccc}
\Phi:&\sfe\times(\R^3\times\R^3)&\longrightarrow&\R^3\times\R^3\\
&\left(\alpha,\begin{pmatrix}x_1\\
y_1\\z_1
\end{pmatrix},\begin{pmatrix}x_2\\
y_2\\z_2
\end{pmatrix}\right)&\longmapsto&
\left(\begin{pmatrix}x_1\cos\alpha-y_1\sin\alpha\\
x_1\sin\alpha+y_1\cos\alpha\\
z_1\end{pmatrix},
\begin{pmatrix}x_2\cos\alpha-y_2\sin\alpha\\
x_2\sin\alpha+y_2\cos\alpha\\z_2
\end{pmatrix}\right).
\end{array}
\]
The functions 
\begin{align*}
&r_1(v,w)=x_1^2+y_1^2=\left\|\begin{pmatrix}
x_1\\
y_1
\end{pmatrix}\right\|^2,\qquad 
r_2(v,w)=x_2^2+y_2^2=\left\|\begin{pmatrix}
x_2\\
y_2
\end{pmatrix}\right\|^2,\\
&d(v,w)=x_1y_2-y_1x_2=\det\begin{pmatrix}
x_1&x_2\\
y_1&y_2
\end{pmatrix},\qquad 
s(v,w)=x_1x_2+y_1y_2=\left\langle\begin{pmatrix}
x_1\\
y_1
\end{pmatrix},\begin{pmatrix}
x_2\\
y_2
\end{pmatrix}\right\rangle,\\
&z_1(v,w)=z_1,\qquad z_2(v,w)=z_2\\
\end{align*}
are $\sfe$-invariant. They also characterize the
$\sfe$-orbits of the action since $d$ and $s$ determine
in a unique way the angle between the vectors $(x_1,y_1)$ and $(x_2,y_2)$. 
Hence, the reduced
manifold is the stratified space $\bar M=\pi(\R^3\times\R^3)\subseteq \R^6$,
where $\pi:\R^3\times\R^3\to\R^6$ is given by
$\pi(v,w)=(r_1,r_2,d,s,z_1,z_2)(v,w)$. We conclude that $\bar{M}$ is the 
semi-algebraic set $\bar{M}=\{(r_1,r_2,d,s,z_1,z_2)
\in\R^6\mid r_1,r_2\geq 0 \text{ and } s^2+d^2=r_1r_2\}$.

The two strata of $\bar{M}$ are $\bar M_0=\{(0,0,0,0,z_1,z_2)\in
\R^6\}$, corresponding to the orbit (isotropy) type manifold
$M_{\sfe}=M_{(\sfe)}= \{(0,0,z_1,0,0,z_2)\in\R^6\}$ with trivial 
$\sfe $-action on it, and 
$\bar M_1=\{(r_1,r_2,d,s,z_1,z_2)\in\R^6\mid (r_1,r_2)\neq(0,0)
\text{ and } d^2+s^2=r_1r_2\}$, 
corresponding to the orbit (isotropy) type manifold
$M_{\{0\}}=M_{(\{0\})}= \{(x_1,y_1,z_1,x_2,y_2,z_2)\in \R^6\mid
(x_1,y_1)\neq (0,0) \text{ or } (x_2,y_2)\neq(0,0)\}$.

Define $U: = \R_{>0}\times\R^4 \subset \R^5$.
Since the points $(r_1,r_2,d,s,z_1,z_2)$ in $\bar M_1$ satisfy
$r_1>0$ or $r_2>0$, we have two charts for $\bar{M}_1$, namely 
$(\psi_1(U), \psi_1^{-1})$ and $(\psi_2(U), \psi_2^{-1})$, where
\[\begin{array}{lccc}
\psi_1:&\R_{>0}\times\R^4 &\rightarrow& \bar M_1\\
&(r_1,d,s,z_1,z_2,)&\mapsto& 
\left(r_1,\frac{d^2+s^2}{r_1},d,s,z_1,z_2\right)\end{array},
\begin{array}{lccc}
\psi_1^{-1}:&\psi_1(U)\subseteq \bar M_1& \rightarrow
&  \R_{>0}\times\R^4\\
&(r_1,r_2,d,s,z_1,z_2)& \mapsto& (r_1,d,s,z_1,z_2)
\end{array}
\]
and 
\[\begin{array}{lccc}
\psi_2:& \R_{>0}\times\R^4& \rightarrow& \bar M_1\\
&(r_2,d,s,z_1,z_2)& \mapsto& 
\left(\frac{d^2+s^2}{r_2},r_2,d,s,z_1,z_2\right)\end{array},
\begin{array}{lccc}
\psi_2^{-1}:&\psi_2(U)\subseteq \bar M_1& \rightarrow
&  \R_{>0}\times\R^4\\
&(r_1,r_2,d,s,z_1,z_2)& \mapsto& (r_2,d,s,z_1,z_2).
\end{array}
\]
Since $\mathcal{V}_G^\circ (m) 
= \left\{\mathbf{d}f(m) \mid f \in C ^{\infty}(M)^G \right\}$ (Lemma 5.8 in \cite{JoRaSn11}),  we have 
for all $m=(x_1,y_1,z_1,x_2,y_2,z_2)\in \R^6$
\[
\V_G^\circ(m) =\erz_{\R}
\left\{\begin{array}{c}
\dr z_1,\quad \dr z_2,\quad x_1\dr x_1+y_1\dr y_1,\quad
x_2\dr x_2+y_2\dr y_2, \\
x_1\dr y_2+y_2\dr x_1-x_2\dr y_1-y_1\dr x_2,\quad
x_1\dr x_2+x_2\dr x_1+y_1\dr y_2+y_2\dr y_1 
\end{array}
\right\},
\]
and, as shown in Example \ref{exap2}, 
\[
\T(m)=\T_G(m)=\erz_{\R}
\left\{\begin{array}{l}
X_1:=\partial_{ z_1},\quad X_2:=\partial_{ z_2},\\
 X_3:=x_1\partial_{x_1}+y_1\partial_{y_1}, \quad
X_4:=x_2\partial_{x_2}+y_2\partial_{y_2}, \\
X_5:=y_1\partial_{x_2}-x_1\partial_{y_2},\quad
X_6:=y_2\partial_{x_1}-x_2\partial_{y_1},\\
X_7:=x_1\partial_{x_2}+y_1\partial_{y_2},\quad
X_8:=x_2\partial_{x_1}+y_2\partial_{y_1},\\
X_9:=x_1\partial_{y_1}-y_1\partial_{x_1},\quad 
X_{10}:=x_2\partial_{y_2}-y_2\partial_{x_2} 
\end{array}
\right\}.
\]
Note that $\V$ is spanned on $M$ by
$X_9+X_{10}=x_1\partial_{y_1}-y_1\partial_{x_1}+x_2\partial_{y_2}-
y_2\partial_{x_2}$.

Consider the Dirac structure $\mathsf{D}\subseteq TM\oplus T^*M$ 
spanned by the pairs
\[
(\partial_{x_1},\dr y_1), \quad (\partial_{y_1},-\dr x_1), \quad 
(\partial_{z_1},0), \quad
(\partial_{x_2},-\dr y_2), \quad (\partial_{y_2},\dr x_2),
\quad(0,\dr z_2).
\]
Comparing this with the sections of $\T=\T_G$ and $\V_G^\circ$ given above, we find
a set $\D^{\sfe}$ of \emph{exact $G$-invariant descending sections} spanning pointwise
the intersection $\mathsf{D}\cap(\T\oplus\V_G^\circ)
= \mathsf{D}\cap(\T_G\oplus\V_G^\circ)$:
\begin{align*}
\D^{\sfe}&=\left\{
\begin{array}{c}
(\partial_{z_1},0), (0,\dr z_2),\\
(-x_1\partial_{y_1}+y_1\partial_{x_1},x_1\dr x_1+y_1\dr y_1),
(x_2\partial_{y_2}-y_2\partial_{x_2},x_2\dr x_2+y_2\dr y_2),\\
(-x_1\partial_{x_2}-y_2\partial_{y_1}-x_2\partial_{x_1}-y_1\partial_{y_2},
x_1\dr y_2+y_2\dr x_1-x_2\dr y_1-y_1\dr x_2),\\
(x_1\partial_{y_2}-x_2\partial_{y_1}-y_1\partial_{x_2}+y_2\partial_{x_1},
x_1\dr x_2+x_2\dr x_1+y_1\dr y_2+y_2\dr y_1)
\end{array}
\right\}\\
&=\left\{
\begin{array}{c}
(\partial_{z_1},0), (0,\dr z_2),
\left(-X_9,\frac{1}{2}\dr r_1\right),\\
\left(X_{10},\frac{1}{2}\dr r_2\right),
(-X_7-X_8,\dr d),
(X_6-X_5,\dr s)
\end{array}
\right\}.
\end{align*}
Thus, we get  for all $m\in \R^6$
\begin{align*}
\D(m)&=\D_G(m)=\erz\left\{
\begin{array}{c}
\partial_{z_1},\qquad X_9=x_1\partial_{y_1}-y_1\partial_{x_1},\\
X_{10}=x_2\partial_{y_2}-y_2\partial_{x_2},\\
X_7+X_8=x_1\partial_{x_2}+y_2\partial_{y_1}+x_2\partial_{x_1}+y_1\partial_{y_2},\\
X_6-X_5=x_1\partial_{y_2}-x_2\partial_{y_1}-y_1\partial_{x_2}+y_2\partial_{x_1}
\end{array}\right\}(m)\\
&=
\left(\erz\left\{\dr z_2,x_1\dr x_1+y_1\dr y_1-x_2\dr x_2-y_2\dr y_2
\right\}\right)^\circ=
\left(\erz\left\{\dr z_2,\dr(r_1-r_2)
\right\}\right)^\circ
\end{align*}
The distribution $\D=\D_G$ is hence integrable and its leaf through a point 
$p=(x_1,y_1,z_1,x_2,y_2,z_2)\in M $ is 
\begin{enumerate}
\item $\{(0,0,t,0,0,z_2)\mid t\in \R \}$
if $r_1(p)=r_2(p)=0$,
\end{enumerate}
and otherwise the level set of the functions $z_2$ and $r_1-r_2$ through 
the point  $p=(x_1,y_1,z_1,x_2,y_2,z_2)\in M $, 
that is,
\begin{enumerate}
\setcounter{enumi}{1}
\item $\{(r\cos\alpha,r\sin\alpha,t,r\cos\beta,r\sin\beta,z_2)\mid r>0,\alpha,\beta,t\in\R
\}$
if $r_1(p)=r_2(p)>0$,
\item $\{(\sqrt{x^2+y^2+k}\cos\alpha,\sqrt{x^2+y^2+k}\sin\alpha,t,x,y,z_2)
\mid x,y,\alpha,t\in\R
\}
$
if $k:=(r_1-r_2)(p)>0$ and
\item $\{(x,y,t,\sqrt{x^2+y^2-k}\cos\alpha,\sqrt{x^2+y^2-k}\sin\alpha,z_2)
\mid x,y,\alpha,t\in\R
\}$
if $k:=(r_1-r_2)(p)<0$.
\end{enumerate} 

The singularity at points where $r_1$ and $r_2$ both vanish can also be seen 
considering the flows $\phi^1$, $\phi^9$, $\phi^{10}$, $\phi^{7+8}$, 
$\phi^{6-5}$ of the vector fields 
$\partial_{z_1},X_9,X_{10},X_7+X_8,X_6-X_5$:
\begin{align*}
\phi^1\left(\begin{pmatrix}x_1\\
y_1\\
z_1
\end{pmatrix},
\begin{pmatrix}x_2\\y_2\\z_2
\end{pmatrix}\right)&=\left(\begin{pmatrix}x_1\\
y_1\\
z_1+t
\end{pmatrix},
\begin{pmatrix}x_2\\y_2\\z_2
\end{pmatrix}\right),\\
\phi^9\left(\begin{pmatrix}x_1\\
y_1\\
z_1
\end{pmatrix},
\begin{pmatrix}x_2\\y_2\\z_2
\end{pmatrix}\right)&=\left(\begin{pmatrix}x_1\cos t-y_1\sin t\\
x_1\sin t+y_1\cos
t\\
z_1
\end{pmatrix},
\begin{pmatrix}x_2\\y_2\\z_2
\end{pmatrix}\right),\\
\phi^{10}\left(\begin{pmatrix}x_1\\
y_1\\
z_1
\end{pmatrix},
\begin{pmatrix}x_2\\y_2\\z_2
\end{pmatrix}\right)&=\left(\begin{pmatrix}x_1\\y_1\\z_1\end{pmatrix},
\begin{pmatrix}x_2\cos t-y_2\sin t\\x_2\sin t+y_2\cos
t\\z_2\end{pmatrix}\right),\\
\phi^{7+8}\left(\begin{pmatrix}x_1\\
y_1\\
z_1
\end{pmatrix},
\begin{pmatrix}x_2\\y_2\\z_2
\end{pmatrix}\right)&=\left(\begin{pmatrix}x_1\cosh t+x_2\sinh t\\y_1\cosh
t+y_2\sinh t\\z_2\end{pmatrix},
\begin{pmatrix}x_2\cosh t+x_1\sinh t\\y_2\cosh t+y_1\sinh t \\z_1
\end{pmatrix}\right),\\
 \phi^{6-5}\left(\begin{pmatrix}x_1\\
y_1\\
z_1
\end{pmatrix},
\begin{pmatrix}x_2\\y_2\\z_2
\end{pmatrix}\right)&=\left(\begin{pmatrix}x_1\cosh t+y_2\sinh t\\y_1\cosh
t-x_2\sinh t\\z_2\end{pmatrix},
\begin{pmatrix}
x_2\cosh t-y_1\sinh t\\y_2\cosh t+x_1\sinh t\\z_1\end{pmatrix}\right).
\end{align*}
Hence we can identify the leaf space $M/\D_G$ with the set
\[\{(r_1,r_2,r_1-r_2,z_2)(p)\mid p\in M\}/\sim,
\]
where $\sim$ is the equivalence relation given on
$\{(r_1,r_2,r_1-r_2,z_2)(p)\mid p\in M\}\subseteq
\R_{>0}\times\R_{>0}\times\R\times \R$ by
$(r_1,r_2,k,t)\sim (r_1',r_2',k',t')$
if and only if $t=t'$ and 
\[k=k'\neq 0\] or \[(k=k'=0) \text{ and }(r_1>0 \text{ or }r_2>0)\text{  and 
}(r_1'>0 \text{ or }r_2'>0)\]
 or 
\[k=k'=0 \text{ and }r_1=r_2=r_1'=r_2'=0.\]

Since $\V\subseteq\D_G$, we find $G_{\sigma}=G=\sfe$ for all $\sigma\in
M/\D_G$. The action of $\sfe$ on each of the leaves is the restriction to the
leaf of the action of $\sfe $on $M$.

We consider the four different cases:
\begin{enumerate}
\item If $\sigma=[0,0,0,a]\in M/\D_G$, we have 
$\J_G^{-1}(\sigma)=\{(0,0,t,0,0,a)\mid t\in\R\}\simeq \R$ and  the induced action of $\sfe$ on $\mathcal{J}_G^{-1}(\sigma)$ is trivial. Thus, the reduced space
$M_\sigma=\J_G^{-1}(\sigma)/\sfe=\J_G^{-1}(\sigma)=\R$
is a  line and the presymplectic form is necessarily trivial.
\item If $\sigma=[R,R,0,a]\in M/\D_G$ with $R>0$, we have 
\begin{align*}
\J_G^{-1}(\sigma)&=\{(r\cos\alpha,r\sin\alpha,t,r\cos\beta,r\sin\beta,a)
\mid r>0,\alpha,\beta,t\in\R
\}\end{align*}
Hence, if we consider it as a subspace of $M/\sfe$, the reduced space is equal to
\begin{align*}
M_\sigma=\J_G^{-1}(\sigma)/\sfe
&=\{(r^2,r^2,r^2\sin(\alpha-\beta),r^2\cos(\alpha-\beta),t,a)\mid r>0,
\alpha,\beta,t\in\R
\}\\
&\simeq \R_{>0}\times \sfe\times\R
\end{align*}
via the diffeomorphism
\[\begin{array}{lccc}
\psi_\sigma:& M_\sigma&\longrightarrow&\R_{>0}\times\sfe\times \R\\
&(r^2,r^2,r^2\sin(\alpha-\beta),r^2\cos(\alpha-\beta),t,a)
&\longmapsto&(r^2,\alpha-\beta,t)\\
\psi_\sigma^{-1}:& \R_{>0}\times\sfe\times \R& \longrightarrow&M_\sigma\\
&(r,\theta,t)&\longmapsto&(r,r,r\sin(\theta),r\cos(\theta),t,a).
\end{array}
\]
Let $\pi: M \rightarrow M/\sfe$ and $\pi_\sigma: 
\mathcal{J}_G ^{-1}( \sigma) \rightarrow M_ \sigma$ be the
canonical projections.

We use the coordinates $(r,\theta,t)$ on $\R_{>0}\times \sfe\times
\R$ and compute the presymplectic form $\omega_\sigma$.
We have \begin{align*}
\partial_r&\sim_{\psi_\sigma^{-1}}\partial_{r_1}+\partial_{r_2}
+\sin\theta\partial_s
+\cos\theta\partial_d=\frac{2r_1\partial_{r_1}+2r_2\partial_{r_2}
+2s\partial_s+2d\partial_d}{2r_1}\,
\text{ and }\,
\partial_\theta\sim_{\psi_\sigma^{-1}}s\partial_d-d
\partial_s
\end{align*}
since $r_1=r_2$ on $M_\sigma$. Since 
\begin{align*}
x_1\partial_{x_1}+y_1\partial_{y_1}+x_2\partial_{x_2}+y_2\partial_{y_2}
&\sim_{\pi}2r_1\partial_{r_1}+2r_2\partial_{r_2}
+2s\partial_s+2d\partial_d
\quad \text{ and }\quad 
X_9\sim_{\pi}-s\partial_d+d
\partial_s,
\end{align*}
this leads to
\begin{align*}
\omega_\sigma(\partial_r,\partial_\theta)&=-\frac{1}{2r_1}
(\pi_\sigma^*\omega_\sigma)
(x_1\partial_{x_1}+y_1\partial_{y_1}+x_2\partial_{x_2}+y_2\partial_{y_2},X_9)\\
&= -\frac{1}{2r_1}(x_1\dr x_1+y_1\dr
y_1)(x_1\partial_{x_1}+y_1\partial_{y_1}+x_2\partial_{x_2}+y_2\partial_{y_2})
=-\frac{1}{2r_1}\left(x_1^2+y_1^2\right)=-\frac{1}{2},\\
\omega_\sigma(\partial_r,\partial_t)&=
\frac{1}{2r_1}(\pi_\sigma^*\omega_\sigma)
(x_1\partial_{x_1}+y_1\partial_{y_1}+x_2\partial_{x_2}+y_2\partial_{y_2},
\partial_{z_1})=0\qquad\text{ and }\\
\omega_\sigma(\partial_\theta,\partial_t)&=
(\pi_\sigma^*\omega_\sigma)(-X_9,\partial_{z_1})=0.
\end{align*}
Thus, we find $\omega_\sigma(r,\theta,t)=\frac{1}{2}\dr \theta\wedge\dr r$.

Note that easy linear algebra arguments show that 
$x_1\partial_{x_1}+y_1\partial_{y_1}+x_2\partial_{x_2}+y_2\partial_{y_2}$
is an element of $\D_G(x_1,y_1,z_1,x_2,y_2,z_2)$ if and only if 
$x_1^2+y_1^2=x_2^2+y_2^2$, that is, if and only if
$(x_1,y_1,z_1,x_2,y_2,z_2)\in\J_G^{-1}(\sigma)$
for $\sigma=[R,R,0,a]$.
\item If $\sigma=[R_1,R_2,k,a]\in M/\D_G$ with $k>0$,
we have 
\begin{align*}
\J_G^{-1}(\sigma)&
=\left\{(\sqrt{x^2+y^2+k}\cos\alpha,\sqrt{x^2+y^2+k}\sin\alpha,t,x,y,a)
\mid x,y,\alpha,t\in\R
\right\}.
\end{align*}
The reduced space $M_\sigma$ is now 
\begin{align*}
M_\sigma=\J_G^{-1}(\sigma)/\sfe
&=\left\{(x^2+y^2+k,x^2+y^2,d,s,t,a)\left|\begin{array}{c}
x,y, t,d,s\in\R\\
d^2+s^2=(x^2+y^2+k)(x^2+y^2)
\end{array}\right.\right\}\simeq \R^3
\end{align*}
via the diffeomorphism
\[\begin{array}{lccc}
\psi_\sigma:& M_\sigma&\longrightarrow&\R^3\\
&(x^2+y^2+k,x^2+y^2,d,s,t,a)&\longmapsto&(d,s,t)\\
\psi_\sigma^{-1}:& \R^3& \longrightarrow&M_\sigma\\
&(d,s,t)&\longmapsto&\left(\frac{\sqrt{k^2+4(d^2+s^2)}+k}{2},
\frac{\sqrt{k^2+4(d^2+s^2)}-k}{2},
d,s,t,a\right).
\end{array}
\]
We use the coordinates $(d,s,t)$ on $\R^3$ 
and compute the presymplectic form $\omega_\sigma$.
We have 
\begin{align*}
\partial_d\sim_{\psi_\sigma\inv}&\frac{2d}{\sqrt{k^2+4(s^2+d^2)}}\partial_{r_1}
+\frac{2d}{\sqrt{k^2+4(s^2+d^2)}}\partial_{r_2}+\partial_d \quad \text{ and }\\
\partial_s\sim_{\psi_\sigma\inv}&\frac{2s}{\sqrt{k^2+4(s^2+d^2)}}\partial_{r_1}
+\frac{2s}{\sqrt{k^2+4(s^2+d^2)}}\partial_{r_2}+\partial_s.
\end{align*}
A computation (see \cite{JoRaSn11}) yields 
\begin{align*}
X_6-X_5\sim_\pi&2d\partial_{r_1}+2d\partial_{r_2}+(r_1+r_2)\partial_d \quad \text{ and
}\quad 
X_7+X_8\sim_\pi2s\partial_{r_1}+2s\partial_{r_2}+(r_1+r_2)\partial_s.
\end{align*}
With $r_1+r_2=\sqrt{k^2+4(s^2+d^2)}  $, this leads to
\begin{align*}
\omega_\sigma(\partial_d,\partial_s)&
=\frac{1}{k^2+4(s^2+d^2)}(\pi_\sigma^*\omega_\sigma)(X_6-X_5,X_7+X_8)\\
&=\frac{1}{k^2+4(s^2+d^2)}(x_1\dr x_2+x_2\dr x_1+y_1\dr y_2+y_2\dr
y_1)(x_1\partial_{x_2}+y_2\partial_{y_1}+x_2\partial_{x_1}
+y_1\partial_{y_2})\\
&=\frac{1}{k^2+4(s^2+d^2)}(x_1^2+x_2^2+y_1^2+y_2^2)=\frac{1}{\sqrt{k^2+4(s^2+d^2)}},\\
\omega_\sigma(\partial_d,\partial_t)&
=\frac{1}{\sqrt{k^2+4(s^2+d^2)}}(\pi_\sigma^*\omega_\sigma)(X_6-X_5,\partial_{z_1})=0\\
\text{ and }\qquad 
\omega_\sigma(\partial_s,\partial_t)&
=\frac{1}{\sqrt{k^2+4(s^2+d^2)}}(\pi_\sigma^*\omega_\sigma)(X_7+X_8,\partial_{z_1})=0,
\end{align*}
which leads to $\omega_\sigma(d,s,t)
=\frac{1}{\sqrt{k^2+4(s^2+d^2)}}\dr d\wedge\dr s$.
\item If $\sigma=[R_1,R_2,k,a]\in M/\D_G$ with $k< 0$,
we have 
\begin{align*}
\J_G^{-1}(\sigma)&
=\left\{(x,y,t,\sqrt{x^2+y^2-k}\cos\alpha,\sqrt{x^2+y^2-k}\sin\alpha,a)
\mid x,y,\alpha,t\in\R
\right\}.
\end{align*}
The reduced space $M_\sigma$ is then equal to
\begin{align*}
M_\sigma=\J_G^{-1}(\sigma)/\sfe
&=\left\{(x^2+y^2,x^2+y^2-k,d,s,t,a)\left|\begin{array}{c}
x,y, t,d,s\in\R\\
d^2+s^2=(x^2+y^2-k)(x^2+y^2)
\end{array}\right.\right\}\simeq \R^3
\end{align*}
via the diffeomorphism
\[\begin{array}{lccc}
\psi_\sigma:& M_\sigma& \longrightarrow&\R^3\\
&(x^2+y^2,x^2+y^2-k,d,s,t,a)&\longmapsto&(d,s,t)\\
\psi_\sigma^{-1}:& \R^3& \longrightarrow&M_\sigma\\
&(d,s,t)&\longmapsto&\left(\frac{\sqrt{k^2+4(d^2+s^2)}+k}{2},
\frac{\sqrt{k^2+4(d^2+s^2)}-k}{2},
d,s,t,a\right).
\end{array}
\]
We use the coordinates $(d,s,t)$ on $\R^3$ and get in the same manner as above\linebreak
$\omega_\sigma(d,s,t) = \frac{1}{\sqrt{k^2+4(s^2+d^2)}}\dr d\wedge\dr s$.
\end{enumerate}

\bigskip

We want to compare these reduced spaces with the presymplectic leaves
of the Dirac structures induced on the two strata $\bar M_0$ and 
$\bar M_1$ of $\bar M$ by standard
singular reduction. These are given by 
$\mathsf D_{\bar M_0}(\bar m)=\erz_\R
\{(\partial_{z_1}\an{\bar m},0), (0,\dr z_2(\bar{m}))\}$
for all $\bar m\in\bar M_0$, and by 
\begin{align*}
\mathsf D_{\bar{M}_1}(r_1,d,s,z_1,z_2)=\erz_\R\left\{
\begin{array}{c}
(\partial_{z_1},0), (0,\dr z_2),
\left(2s\partial_d-2d\partial_s,\dr r_1\right),\\
\left(-2s\partial_{r_1}
-\left(r_1+\frac{s^2+d^2}{r_1}\right)\partial_s,
\dr d\right),\\
\left(2d\partial_{r_1}
+\left(r_1+\frac{s^2+d^2}{r_1}\right)\partial_{d},\dr s\right)
\end{array}
\right\}(r_1,s,d,z_1,z_2)
\end{align*}
in the chart $(U,\psi_1)$ and 
\begin{align*}
\mathsf D_{\bar{M}_1}(r_2,s,d,z_1,z_2)=\erz_\R\left\{
\begin{array}{c}
(\partial_{z_1},0), (0,\dr z_2),
\left(2s\partial_d-2d\partial_s,\dr r_2\right),\\
\left(-2s\partial_{r_2}
-\left(r_2+\frac{s^2+d^2}{r_2}\right)\partial_s,
\dr d\right),\\
\left(2d\partial_{r_2}
+\left(r_2+\frac{s^2+d^2}{r_2}\right)\partial_d,\dr s\right)
\end{array}
\right\}(r_2,s,d,z_1,z_2).
\end{align*}
in the chart  $(U,\psi_2)$ (see \cite{JoRaSn11}).

Take $p\in M$. If $p\in M_{\sfe}$, that is, 
$\pi(p)\in\bar M_0$, we have $p=(0,0,z_1,0,0,z_2)$
and the reduced space 
$(M_\sigma,\omega_\sigma)$ for $\sigma=\J_G(p)$
is of the first type: $(M_\sigma,\omega_\sigma)\simeq (\R,0)$.
The presymplectic leaf of $(\bar M_0,\mathsf D_{\bar M_0})$
through $\pi(p)\in\bar M_0$
is obviously
$\bar N_p=\{(0,0,0,0,t,z_2)\mid t\in \R\}\simeq \R$
with the trivial presymplectic structure. It is easy to see
that $(M_\sigma,\omega_\sigma)\simeq (\bar N_p,0)$
via the diffeomorphism constructed in the proof 
of Theorem \ref{presympl}.

For $p\in M_{\{e\}}$, we have $\pi(p)\in\bar M_1$.
We study the presymplectic leaves of
$(\bar M_1,\mathsf D_{\bar M_1})$. The corresponding distribution
$\mathsf{G_1}$ is given by 
\begin{align*}
\mathsf{G_1}(\pi(p))&=\erz_{\R}\left\{\partial_{z_1}, \quad 
s\partial_{ d}- d\partial_{s},\quad 
2s\partial_{r_1}
+\left(r_1+\frac{s^2+ d^2}{r_1}\right)\partial_{s},\quad 
2 d\partial_{r_1}
+\left(r_1+\frac{s^2+ d^2}{r_1}\right)\partial_{ d}\right\}\\
&=\erz_{\R}\left\{\partial_{z_1}, 
2s\partial_{r_1}
+\left(r_1+\frac{s^2+ d^2}{r_1}\right)\partial_{s},\quad 
2 d\partial_{r_1}
+\left(r_1+\frac{s^2+ d^2}{r_1}\right)\partial_{ d}\right\}
\end{align*}
in the chart $\psi_1$.
We find that $\mathsf{G_1}$
is the smooth annihilator of the codistribution that is spanned pointwise by
$\left\{\dr z_2,  \dr\left(
    r_1-\frac{s^2+ d^2}{r_1}\right)\right\}$
and that has constant rank on $\psi_1(U)\subseteq\bar M_1$. With the same argument
we find that $\mathsf{G_1}$ has constant rank on $\psi_2(U)\subseteq\bar M_1$
and since $\psi_1(U)\cap\psi_2(U)$ is open and dense in $\bar M_1$, the
distribution $\mathsf{G_1}$ is a vector bundle over $\bar M_1$. We have shown
in \cite{JoRaSn11} that it is involutive and so it is
completely integrable in the sense of Frobenius.
We have again three cases:
\begin{enumerate}
\item Suppose that $r_1>0$, $\left(r_1-\frac{s^2+ d^2}{r_1}\right)(p)=0=k$,
  $z_2=a\in\R$. Then the leaf $N_{0,a}$
 of $\mathsf{G_1}$ through $\pi(p)$ is the subset
\begin{align*}
\bar M_1\supseteq \psi_1(U)\supseteq N_{0,a}&=\left\{(r_1,s, d,z_1,a) \mid r_1>0,\,
  r_1^2=s^2+ d^2,z_1 \in\R\right\}\\
&=\{(z,z\cos\alpha,z\sin\alpha,t,a)\mid z>0, \alpha\in\sfe,t\in\R\}
\simeq \R_{>0}\times\sfe\times\R
\end{align*}
via the diffeomorphism
\begin{equation*}
\begin{array}{cccc}
\psi_{0,a}:&N_{0,a}& \longrightarrow&\R_{>0}\times\sfe\times\R\\
&(z,z\cos\alpha,z\sin\alpha,t,a)&\longmapsto&(z,\alpha,t)\\
\psi_{0,a}^{-1}:&\R_{>0}\times\sfe\times\R& \longrightarrow&N_{0,a}\\
&(z,\alpha,t)&\longmapsto&(z,z\cos\alpha,z\sin\alpha,t,a).
\end{array}
\end{equation*}
Note that this leaf of $\mathsf{G_1}$ is included in the intersection
$\psi_1(U)\cap\psi_2(U)$;
the values of $r_1$ and $r_2$ are equal on the leaf. Thus, for instance,
if $r_1$ vanishes, then $r_2$ has to be zero too, which is not possible on 
$\bar{M}_1$.  We compute the presymplectic structure on $N_{0,a}$.
Since
\begin{align*}
\partial_z&\sim_{\psi_{0,a}^{-1}} 
\partial_{r_1}+\cos\alpha\partial_s+\sin\alpha\partial_{ d}
=\frac{1}{r_1}(r_1\partial_{r_1}+s\partial_s+ d\partial_{ d})\\
\partial_{\alpha}&\sim_{\psi_{0,a}^{-1}} s\partial_ d- d\partial_s
\qquad \text{ and } \qquad \partial_t\sim_{\psi_{0,a}^{-1}}\partial_{z_1},
\end{align*}
we have
\begin{align*}
\omega_{N_{0,a}}(\partial_t,\partial_z)&=0,\quad \omega_{N_{0,a}}(\partial_t,\partial_\alpha)=0\\
\text{and}\quad \omega_{N_{0,a}}(\partial_z,\partial_\alpha)&=-\frac{1}{2}\dr r_1
\left(\frac{1}{r_1}(r_1\partial_{r_1}+s\partial_s+ d\partial_{ d})\right)
=-\frac{1}{2}.
\end{align*}
Therefore, $\omega_{N_{0,a}}=\frac{1}{2}\dr \alpha\wedge\dr z$. This
shows that $(N_{0,a},\omega_{N_{0,a}})$ is presymplectomorphic
to $(M_\sigma,\omega_\sigma)$, where $\sigma=\mathcal{J}_G(p)=[r_1,r_1, 0,a]$.

\item Suppose that $r_1>0$, $\left(r_1-\frac{s^2+ d^2}{r_1}\right)(p)=k>0$,
$z_2=a\in\R$. Then the leaf $N_{k,a}$ of $\mathsf{G_1}$ through $\pi(p)$ is 
the subset
\begin{align*}
\bar M_1\supseteq \psi_1(U)\supseteq N_{k,a}
&=\left\{(r_1,s, d,z_1,a) \,\left|\; r_1>0,\,
  r_1-\frac{s^2+ d^2}{r_1}=k,z_1 \in\R\right.\right\}\\
&=\left\{\left.\left(\frac{\sqrt{4(s^2+ d^2)+k^2}+k}{2},s, d,z_1,a\right)\,\right|\;
z_1,s, d\in \R
\right\}.
\end{align*}
Note that since $r_1-r_2$ is equal to $k>0$ on $N_{k,a}$, 
we have  $r_1>0$ on $N_{k,a}$ and hence,
$N_{k,a}\subseteq \psi_1(U)$.
To compute the presymplectic structure on $N_{k,a}$, we study its graph, which
is the induced Dirac structure on the leaf (see \eqref{inducedDirac}).
We have 
\begin{align*}
\Gamma(\mathsf D_{N_{k,a}})
=\erz_{C^\infty(N_{k,a})}\left\{ \begin{array}{c}
(\partial_{z_1},0),\, \left(s\partial_ d- d\partial_s,
\frac{s}{\sqrt{4(s^2+ d^2)+k^2}}\dr s+\frac{ d}
{\sqrt{4(s^2+ d^2)+k^2}}\dr d\right),\\
(-\sqrt{4(s^2+ d^2)+k^2}\partial_s,\dr d),\\
(\sqrt{4(s^2+ d^2)+k^2}\partial_ d,\dr s)
\end{array}\right\}.
\end{align*}
We have used the fact that since $r_1-k=\frac{s^2+ d^2}{r_1}$, we have 
$r_1+\frac{s^2+ d^2}{r_1}=2r_1-k=\sqrt{4(s^2+ d^2)+k^2}$,
and the equality
\begin{align*}
\dr r_1&=8s\cdot \frac{1}{2\cdot2\cdot
  \sqrt{4(s^2+ d^2)+k^2}}\dr s+8 d\cdot \frac{1}{2\cdot2\cdot
  \sqrt{4(s^2+ d^2)+k^2}}\dr d\\
&=\frac{2s}{\sqrt{4(s^2+ d^2)+k^2}}\dr s+\frac{2 d}
{\sqrt{4(s^2+ d^2)+k^2}}\dr d.
\end{align*}
This yields
\begin{align*}
\Gamma(\mathsf D_{N_{k,a}})
=\erz_{C^\infty(N_{k,a})}\left\{ \begin{array}{c}
(\partial_{z_1},0),
(-\sqrt{4(s^2+ d^2)+k^2}\partial_s,\dr d),\\
(\sqrt{4(s^2+ d^2)+k^2}\partial_ d,\dr s)
\end{array}\right\}
\end{align*}
and thus
$\omega_{N_{k,a}}=\frac{1}{\sqrt{4(s^2+ d^2)+k^2}}\dr d\wedge\dr s$.
\item Suppose that $r_2>0$, $\left(\frac{s^2+ d^2}{r_2}-r_2\right)(p)=k<0$,
  $z_2=a\in\R$. Then the leaf $N_{k,a}$
 of $\mathsf{G_1}$ through $p$ is the subset
\begin{align*}
\bar M_1\supseteq \psi_2(U)\supseteq N_{k,a}&=\left\{(r_2,s, d,z_1,a) \left| r_2>0,\,
  \frac{s^2+ d^2}{r_2}-r_2=k,z_1 \in\R\right.\right\}\\
&=\left\{\left.\left(\frac{\sqrt{4(s^2+ d^2)+k^2}+k}{2},s, d,z_1,a\right)\right|
z_1,s, d\in \R
\right\}.
\end{align*}
Note that since $r_1-r_2$ is equal to $k<0$ on $N_{k,a}$, 
we have  $r_2>0$ on $N_{k,a}$ and hence,
$N_{k,a}\subseteq \psi_2(U)$.
To compute the presymplectic structure on $N_{k,a}$, we study its graph, which
is the induced Dirac structure on the leaf (see \eqref{inducedDirac}).
We have 
\begin{align*}
\Gamma(\mathsf D_{N_{k,a}})
=\erz_{C^\infty(N_{k,a})}\left\{ \begin{array}{c}
(\partial_{z_1},0),\, \left(s\partial_ d- d\partial_s,
\frac{s}{\sqrt{4(s^2+ d^2)+k^2}}\dr s+\frac{ d}
{\sqrt{4(s^2+ d^2)+k^2}}\dr d\right),\\
(-\sqrt{4(s^2+ d^2)+k^2}\partial_s,\dr d),\\
(\sqrt{4(s^2+ d^2)+k^2}\partial_ d,\dr s)
\end{array}\right\}
\end{align*}
as in the preceding case.
This yields
\begin{align*}
\Gamma(\mathsf D_{N_{k,a}})
=\erz_{C^\infty(N_{k,a})}\left\{ \begin{array}{c}
(\partial_{z_1},0),
(-\sqrt{4(s^2+ d^2)+k^2}\partial_s,\dr d),\\
(\sqrt{4(s^2+ d^2)+k^2}\partial_ d,\dr s)
\end{array}\right\}
\end{align*}
and thus
$\omega_{N_{k,a}}=\frac{1}{\sqrt{4(s^2+ d^2)+k^2}}\dr d\wedge\dr s$.
\end{enumerate}
\end{example}

\begin{appendix}
\section{The module of equivariant vector fields on a representation 
space}

Let $\Phi:G \times V \rightarrow V$ be a finite dimensional 
representation of a compact Lie group $G$. We review a method presented 
in \cite{ChLa00} to find $C^\infty(V)^G$-generators for the set 
$\mathfrak{X}(V)^G: = \{X \in \mathfrak{X}(V) \mid \Phi_g^\ast X = X\quad \forall
g\in G\}$ 
of $G$-equivariant vector fields on $V$. 

\medskip

Since $G$ is compact, averaging an arbitrary inner product 
on $V$ yields a $G $-invariant inner product 
$\left\langle\!\left\langle\,,\right\rangle\!\right\rangle$ 
on $V$. Thus we can assume, without loss  of generality, 
that the representation $\Phi$ is orthogonal. Denote
by $\left\langle\,, \right\rangle: V ^\ast \times V 
\rightarrow \mathbb{R}$ the nondegenerate duality pairing.

\medskip

The computation of the set of generators for $\mx(V)^G$  is facilitated by
the following two observations.
\medskip

\noindent \textbf{1.)} \textit{There is a bijective correspondence 
between $G$-equivariant vector fields on $V$ and 
$G$-equivariant  maps $V\to V$.}

 Indeed, since $TV\simeq V\times V$ and the 
tangent lift of the representation $\Phi$ to $TV$ is 
given by $g \cdot (v,w):= T_v\Phi_gw=(\Phi_gv,\Phi_gw)$, 
for all $g\in G$, $v, w \in V$, we can associate to 
each $X \in \mathfrak{X}(V)^G$ the smooth 
$G$-equivariant map  $f_X:V\to V$, $f_X(v):=
\operatorname{pr}_2(X(v))$, and vice versa; 
$\operatorname{pr}_2: V \times V \rightarrow V $ is 
the projection on the second factor.
\medskip

\noindent \textbf{2.)} \textit{There is a surjective 
map from the set of $G$-invariant real valued functions 
$C^\infty(V\times V)^G$ on $V \times V$ to the set 
$\mathfrak{X}(V)^G$ of $G$-equivariant vector fields 
on $V$.}

Indeed, if $\varphi:TV=V\times V\to\mathbb{R}$ is a 
smooth $G$-invariant function, define $f_\varphi(v)
:=\mathbf{d}_2\varphi(v,0) \in V ^\ast$ for 
all $v\in V$, where $\mathbf{d}_2 \varphi$ denotes the 
derivative relative to the second variable. Then 
$f_\varphi:V\to V^*$ is a $G$-equivariant 
function, where the $G $-representation on $V ^\ast$ is defined by 
$g \cdot l:=l \circ \Phi_{g ^{-1}} $, for all $l\in V ^\ast$ and $g \in G$. The 
$G$-invariant inner product $\left\langle \! \left\langle\,, 
\right\rangle \! \right\rangle$ on $V$ induces the isomorphism $v \in V 
\mapsto \left\langle \! \left\langle v , \cdot \right\rangle \! \right\rangle\in 
V ^\ast$ of  $G$-representations 
and hence $f _\varphi$ induces a $G$-equivariant map 
$\tilde{f}_\varphi:V\to V$ defined by 
$\left\langle \! \left\langle \tilde{f}_\varphi(v), w \right\rangle \! \right\rangle
: = \left\langle f _\varphi(v), w \right\rangle =
\left\langle \mathbf{d}_2\varphi(v, 0), w \right\rangle$ for all $v,w \in V$.
Therefore, we get $X_ \varphi \in \mathfrak{X}(V)^G$ 
defined by $X_\varphi(v): = (v, \tilde{f}_\varphi(v))$ for all $v \in V$. 

Conversely, each $X\in\mathfrak{X}(V)^G$, uniquely 
defines the smooth $G$-equivariant map $f_X: V 
\rightarrow V $ given by $X(v)=(v, f_X(v)) $ for all 
$v \in V$ and hence the smooth $G$-invariant
function $\varphi_X:V \times V \rightarrow \mathbb{R}$ 
defined by  $\varphi_X(v,w)=\left\langle\!\left\langle 
f_X(v), w\right\rangle\!\right\rangle$ which is linear 
in the second component, that is, $\varphi_X \in 
\mathcal{S}(V \times V)^G: = \{\varphi \in C^{\infty}
(V \times V)^G \mid \varphi(v, \cdot ) \in V ^\ast, 
\text{ for all } v \in V\}$. Note that if $\psi \in 
\mathcal{S}(V \times V)^G$,
then $\psi(v, w) = \left\langle\mathbf{d}_2\psi(v,0), 
w \right\rangle$. Using
this identity, it is easily seen that the correspondences 
$\varphi \in \mathcal{S}(V \times V)^G \mapsto X_ \varphi 
\in \mathfrak{X}(V)^G$, $X\in \mathfrak{X}(V)^G \mapsto 
\varphi_X \in \mathcal{S}(V \times V )^G$ are inverses to 
each other. 

So we have a \textit{bijective map
$\mx(V)^G\leftrightarrow \mathcal S(V\times V)^G$}. In 
particular \textit{the map $\varphi \in C ^{\infty}(V \times V)^G 
\mapsto X_ \varphi\in \mathfrak{X}(V)^G$ is surjective}.

\medskip

Let $\{p_1,\ldots,p_n\}$ be a Hilbert basis for the ring and
finitely generated $\mathbb{R}$-algebra of
$G$-invariant polynomials on $V\times V$. 
The \emph{Hilbert map} $\mathcal{H}: V \times V \rightarrow 
\mathbb{R}^n$, $\mathcal{H}(v, w): = \left(p _1(v,w), \ldots,
p _n(v,w)\right)$ is proper (inverse images of compact sets are 
compact), it separates orbits (if $(v',w') \neq g \cdot (v,w)$ for
all $g \in G$, then $\mathcal{H}(v',w') \neq \mathcal{H}(v,w)$), and 
there is a homeomorphism $\overline{\mathcal{H}}: (V \times V)/G 
\stackrel{\sim}\longrightarrow \mathcal{H}(V \times V) \subseteq \mathbb{R}^n$ 
such that $\overline{\mathcal{H}} \circ \pi = \mathcal{H}$,
where $\pi:V \times V \rightarrow (V \times V)/G$ is the projection on the
orbit space (see \cite{ChLa00}, Theorem 5.2.9). 
The theorem of Schwarz-Mather (see, for instance,
\cite{OrRa04}, Theorem 2.5.3) states that 
for each $G$-invariant function $\varphi\in C^\infty
(V\times V)^G$, there exists $F_\varphi\in C^\infty(\R^n)$ 
such that $\varphi=F_\varphi\circ(p_1,\ldots,p_n)$.
Since 
\[
f_\varphi(v):=\mathbf{d}_2\varphi(v,0)
=\sum_{i=1}^n\frac{\partial F_\varphi}{\partial x_i}(p_1(v,0),\ldots,p_n(v,0)) \mathbf{d}_2 p_i(v,0) \in V ^\ast,
\]
the $G$-equivariant vector fields 
$X_1,\ldots,X_n$, $X_i(v):=(v, \tilde{p}_i(v))$, 
$\left\langle\!\left\langle\tilde{p}_i(v), w 
\right\rangle\!\right\rangle = 
\left\langle\mathbf{d}_2p_i(v,0), w \right\rangle$, 
for all $w \in V $, $i=1, \ldots , n$, associated
to the Hilbert basis $p_1,\ldots,p_n$ are spanning 
vector fields for the $C^\infty(V)^G$-module 
$\mathfrak{X}(V)^G$.

In the examples below, we will need to know the Hilbert basis
for the diagonal actions of $\sfe$ and $\SO$ on $n$ copies of
$\R^2$, respectively $\R^3$.
These bases are given in the following proposition, which is 
proved in \cite{KrPr96}, Theorem 10.2.

\begin{proposition}\label{prop_hilbert}
\begin{enumerate}
\item Consider the diagonal action $\phi$ of $\sfe=\operatorname{SO}(2)$ on $n$ copies 
of $\R^2$, that is,
$\phi:\sfe\times\left(\R^2\right)^n\mapsto \left(\R^2\right)^n$, 
$\phi(\alpha,(v_1,\ldots,v_n))=(\alpha\cdot v_1,\ldots,\alpha\cdot v_n)$.
Write $\mathcal P\left(\left(\R^2\right)^n\right)$ as 
$\R[X_1,Y_1,\ldots,X_n,Y_n]$ and 
define $P_{ij}, Q_{ij}\in\mathcal P\left(\left(\R^2\right)^n\right)^{\sfe}$
by  $P_{ij}=X_iX_j+Y_iY_j$ and $Q_{ij}=X_iY_j-X_jY_i$ for all $i,j=1,\ldots,n$.
Then $\mathcal B_n:=\{P_{ij}, Q_{kl}\mid 1\leq i\leq j\leq n, 1\leq k<l\leq n\}$
is a Hilbert basis for $ \mathcal P\left(\left(\R^2\right)^n\right)^{\sfe}$.
\item Consider the diagonal action $\psi$ of $\SO$ on $n$ copies 
of $\R^3$, that is,
$\psi:\SO\times\left(\R^3\right)^n\mapsto \left(\R^3\right)^n$, 
$\psi(A,(v_1,\ldots,v_n))=(A\cdot v_1,\ldots,A\cdot v_n)$.
Write $\mathcal P\left(\left(\R^3\right)^n\right)$ as 
$\R[X_1,Y_1,Z_1,\ldots,X_n,Y_n,Z_n]$ and 
define $P_{ij}, Q_{ijk}\in\mathcal P\left(\left(\R^3\right)^n\right)^{\SO}$
by  $P_{ij}=X_iX_j+Y_iY_j+Z_iZ_j$ and 
$Q_{ijk}=X_iY_jZ_k+ Y_iZ_jX_k+Z_iX_jY_k-X_iY_kZ_j-Y_iX_jZ_k-Z_iY_jX_k$ 
for all $i,j,k=1,\ldots,n$.
Then $\mathcal C_n:=\{P_{ij}, Q_{klm}\mid 1\leq i\leq j\leq n, 1\leq k<l<m\leq n\}$
is a Hilbert basis for $ \mathcal P\left(\left(\R^3\right)^n\right)^{\SO}$.
\end{enumerate}
\end{proposition}

\begin{example}\label{exap1}
We consider Example \ref{ex1}: the (automatically proper) action $\Phi$ of $G:=\mathbb{S}^1\simeq 
\R/(2\pi\Z)$ 
on  $M:=\R^3$ given by  
\[
\Phi_ \alpha(x,y,z)=\alpha\cdot(x,y,z)=(x\cos\alpha-y\sin\alpha,
x\sin\alpha+y\cos\alpha,z).
\]

We want to find $C^\infty(M)^G$-generators for the 
$G$-invariant vector fields on $M = \mathbb{R}^3$, 
$G = \sfe$. Hence, we have to find the invariant functions
for the diagonal action of $\sfe$ on $M\times M$, that 
is, the action $\Psi:\sfe\times(\R^3\times\R^3)\to
(\R^3\times\R^3)$ given by $\Psi_\alpha(v,w)=
(\Phi_\alpha(v),\Phi_\alpha(w))$ for all
$v,w\in\R^3$. Write $(v,w)=(x_v,y_v,z_v,x_w,y_w,z_w)\in\R^6$. Then, 
by Proposition \ref{prop_hilbert},  the Hilbert basis for 
$\mathcal P^{\sfe}((\R^3)^2)$ is given by
$\{p_1,p_2,p_3,p_4,p_5,p_6\}$,
where
\begin{align*}
p_1(v,w)&=x_v^2+y_v^2\qquad p_2(v,w)=x_w^2+y_w^2\qquad p_3(v,w)=z_v\\
p_4(v,w)&=z_w\qquad p_5(v,w)=x_vx_w+y_vy_w\qquad p_6(v,w)=x_vy_w-y_vx_w.
\end{align*}

Since $\mathbf{d}_w p_1(v,0)=0$, 
$\mathbf{d}_w p_2(v,0)=0$, $\mathbf{d}_w p_3(v,0)=0$,
$\mathbf{d}_w p_4(v,0)=\mathbf{d}z_w$,
$\mathbf{d}_w p_5(v,0)=x_v\mathbf{d}{x_w}+
y_v\mathbf{d}{y_w}$, $\mathbf{d}_w p_6(v,0)=
x_v\mathbf{d}{y_w}-y_v\mathbf{d}{x_w}$, the 
method in \cite{ChLa00} reviewed above yields the generators
$X_1=X_2=X_3=0$, $X_4(v)=\partial_{z_v}$,
$X_5(v)=x_v\partial_{x_v}+y_v\partial_{y_v}$,
$X_6(v)=x_v\partial_{y_v}-y_v\partial_{x_v}$
of the $C^{\infty}(\mathbb{R}^{6})^{\sfe}$-module of 
equivariant vector fields $\mathfrak{X}
(\mathbb{R}^{6})^{\sfe}$ on $\mathbb{R}^{6}$. 
Note that $X_6$ is the fundamental vector field of the action of $\sfe$
on $\R^3$
defined by the Lie algebra element
$1\in T_{1}\sfe\simeq\R$.

Thus, we get
$\Gamma(\T)=\Gamma(\T_G)=\erz_{C^\infty(M)}\{\partial_z,
x\partial_{x}+y\partial_{y},
x\partial_{y}-y\partial_{x}\}$
as was used in Example \ref{ex1}.
\end{example}

\begin{example}\label{exap2}
We consider here the action of  Example \ref{ex2}. The vector space 
 $M=\R^3\times\R^3$ is endowed with the 
(automatically proper) diagonal action of $G=\sfe\simeq \R/(2\pi\Z)$ on it, 
i.e., 
\[
\begin{array}{cccc}
\Phi:&\sfe\times(\R^3\times\R^3)&\longrightarrow&\R^3\times\R^3\\
&\left(\alpha,\begin{pmatrix}x_1\\
y_1\\z_1
\end{pmatrix},\begin{pmatrix}x_2\\
y_2\\z_2
\end{pmatrix}\right)&\longmapsto&
\left(\begin{pmatrix}x_1\cos\alpha-y_1\sin\alpha\\
x_1\sin\alpha+y_1\cos\alpha\\
z_1\end{pmatrix},
\begin{pmatrix}x_2\cos\alpha-y_2\sin\alpha\\
x_2\sin\alpha+y_2\cos\alpha\\z_2
\end{pmatrix}\right).
\end{array}
\]
We have to consider the action $\Psi$ of $\sfe$ on 
$\R^{12}\simeq(\R^3\times\R^3)\times(\R^3\times\R^3)  $ 
defined by 
\[
\begin{array}{cccc}
\Psi:&\sfe\times\left((\R^3\times\R^3)\times(\R^3\times\R^3)\right)
&\longrightarrow&\left((\R^3\times\R^3)\times(\R^3\times\R^3)\right)\\
&(v,w,u,t)
&\longmapsto&
(\Phi_\alpha(v),\Phi_\alpha(w),\Phi_\alpha(u),\Phi_\alpha(t)).
\end{array}
\]
We write $v=(x_v,y_v,z_v)$, $w=(x_w,y_w,z_w)$, etc.
By proposition \ref{prop_hilbert}, 
we have $\mathcal H=\{p_i:\R^{12}\to \R\mid i=1,\ldots,20\}$, 
where 
\begin{align*}
p_1(v,w,u,t)&=z_v,\qquad p_2(v,w,u,t)=z_w,\qquad p_3(v,w,u,t)=z_u, 
\qquad p_4(v,w,u,t)=z_t,\\
p_5(v,w,u,t)&=x_v^2+y_v^2,\qquad p_6(v,w,u,t)=x_w^2+y_w^2,
\qquad p_7(v,w,u,t)=x_u^2+y_u^2, \\
p_8(v,w,u,t)&=x_t^2+y_t^2,
\qquad p_9(v,w,u,t)=x_vx_w+y_vy_w,
\qquad p_{10}(v,w,u,t)=x_vx_u+y_vy_u, \\
 p_{11}(v,w,u,t)&=x_vx_t+y_vy_t,
\qquad p_{12}(v,w,u,t)=x_wx_u+y_wy_u,
\qquad p_{13}(v,w,u,t)=x_wx_t+y_wy_t,
\\
 p_{14}(v,w,u,t)&=x_ux_t+y_uy_t,
\qquad p_{15}(v,w,u,t)=x_vy_w-y_vx_w,
\qquad p_{16}(v,w,u,t)=x_vy_u-y_vx_u,\\
p_{17}(v,w,u,t)&=x_vy_t-y_vx_t,
\qquad p_{18}(v,w,u,t)=x_wy_u-y_wx_u,
\qquad p_{19}(v,w,u,t)=x_wy_t-y_wx_t,\\
p_{20}(v,w,u,t)&=x_uy_t-y_ux_t.
\end{align*}
Since these functions determine the lenghts of the four vectors and the
angles between them, we have, as predicted,
$(p_1,\ldots,p_{20})(\R^{12})\simeq \R^{12}/\sfe$.
 We compute the vector fields
$X_1,\ldots,X_{20}$ associated to these polynomial functions in $\mathcal{H}$.  Since
\begin{align*}
\widetilde{p}_1(v,w)&=0,\qquad 
\widetilde{p}_2(v,w)=0,\qquad 
\widetilde{p}_3(v,w)=\dr z_u, \qquad 
\widetilde{p}_4(v,w)=\dr z_t,\\
\widetilde{p}_5(v,w)&=0,\qquad 
\widetilde{p}_6(v,w)=0, \qquad 
\widetilde{p}_7(v,w)=(2x_u\dr x_u+2y_u\dr y_u)\an{(v,w,0,0)}=0, \\
\widetilde{p}_8(v,w)&=(2x_t\dr x_t+2y_t\dr y_t)\an{(v,w,0,0)}=0, \qquad 
\widetilde{p}_9(v,w)=0, \qquad 
\widetilde{p}_{10}(v,w)=x_v\dr x_u+y_v\dr y_u, \\
\widetilde{p}_{11}(v,w)&=x_v\dr x_t+y_v\dr y_t,\qquad 
\widetilde{p}_{12}(v,w)=x_w\dr x_u+y_w\dr y_u,\qquad 
\widetilde{p}_{13}(v,w)=x_w\dr x_t+y_w\dr y_t, \\
\widetilde{p}_{14}(v,w)&=(x_u\dr x_t+x_t\dr x_u+y_u\dr y_t+y_t\dr x_u)\an{(v,w,0,0)}=0, \qquad 
\widetilde{p}_{15}(v,w)=0,\\
\widetilde{p}_{16}(v,w)&=x_v\dr y_u-y_v\dr x_u,
\widetilde{p}_{17}(v,w)=x_v\dr y_t-y_v\dr x_t,\\
\widetilde{p}_{18}(v,w)&=x_w\dr y_u-y_w\dr x_u,\qquad 
\widetilde{p}_{19}(v,w)=x_w\dr y_t-y_w\dr x_t,\\
\widetilde{p}_{20}(v,w)&=(x_u\dr y_t+y_t\dr x_u-y_u\dr x_t-x_t\dr y_u)\an{(v,w,0,0)}=0,
\end{align*}
we get 
\begin{align*}
X_1&=X_2=X_5=X_6=X_7=X_8=X_9=X_{14}=X_{15}=X_{20}=0,\\
X_3(v,w)&=\partial_{z_v}, 
\qquad X_4(v,w)=\partial_{z_w}, \qquad X_{10}(v,w)=x_v\partial_{x_v}+y_v\partial_{y_v},\\
 X_{11}(v,w)&=x_v\partial_{x_w}+y_v\partial_{y_w},
\qquad X_{12}(v,w)=x_w\partial_{x_v}+y_w\partial_{y_v},
\qquad X_{13}(v,w)=x_w\partial_{x_w}+y_w\partial_{y_w},
\\
X_{16}(v,w)&=x_v\partial_{y_v}-y_v\partial_{x_v},\qquad 
X_{17}(v,w)=x_v\partial_{y_w}-y_v\partial_{x_w},\\
X_{18}(v,w)&=x_w\partial_{y_v}-y_w\partial_{x_v},
\quad\text{ and }\quad X_{19}(v,w)=x_w\partial_{y_w}-y_w\partial_{x_w}.
\end{align*}
Thus, $\T_G=\T$ is spanned by 
\begin{align*}
\left\{\begin{array}{c}
X_3(v,w)=\partial_{z_v}, 
\quad X_4(v,w)=\partial_{z_w}, \quad X_{10}(v,w)=
x_v\partial_{x_v}+y_v\partial_{y_v},\\
 X_{11}(v,w)=x_v\partial_{x_w}+y_v\partial_{y_w},
\quad X_{12}(v,w)=x_w\partial_{x_v}+y_w\partial_{y_v},
\quad X_{13}(v,w)=x_w\partial_{x_w}+y_w\partial_{y_w},\\
X_{16}(v,w)=x_v\partial_{y_v}-y_v\partial_{x_v},\quad 
X_{17}(v,w)=x_v\partial_{y_w}-y_v\partial_{x_w},\\
X_{18}(v,w)=x_w\partial_{y_v}-y_w\partial_{x_v},
\quad X_{19}(v,w)=x_w\partial_{y_w}-y_w\partial_{x_w}
\end{array}
\right\}.
\end{align*}
Note that the vertical space  of the action 
is spanned by $V=X_{16}+X_{19}$.
\end{example}

\begin{example}\label{exap3}
Our last example is  an example in \cite{JoRaSn11}, inspired by 
\cite{Bierstone75}. We consider the diagonal action $\Phi$
of $G:=\operatorname{SO}(3)$ on $M:=\R^3\times \R^3$,
that is, $\Phi:\operatorname{SO}(3)\times ( \R^3\times \R^3)\to\R^3\times \R^3$,
$\Phi(A,v,w):= A \cdot (v, w): = (Av,Aw)$. 

Here, we have thus to consider the action $\Psi$ of 
$\operatorname{SO}(3)$ on $\R^{12}$ given by 
$\Psi(A,(v,w,u,t))=(Av,Aw,Au,At)$.
We  write again $v=(x_v,y_v,z_v)$, $w=(x_w,y_w,z_w)$, etc.
By Proposition \ref{prop_hilbert}, 
the Hilbert basis is $\mathcal H=\{p_1,\ldots,p_{14}\}$, 
where the polynomial functions $p_i:\R^{12}\to \R$, 
$i=1,\ldots,14$ are defined by
\begin{align*}
p_1(v,w,u,t)&=\|v\|^2=x_v^2+y_v^2+z_v^2,
\qquad p_2(v,w,u,t)=x_w^2+y_w^2+z_w^2,
\qquad p_3(v,w,u,t)=x_u^2+y_u^2+z_u^2,\\ 
p_4(v,w,u,t)&=x_t^2+y_t^2+z_t^2,\qquad 
p_5(v,w,u,t)=\langle v,w\rangle=
x_vx_w+y_vy_w+z_vz_w,\\
p_6(v,w,u,t)&=x_vx_u+y_vy_u+z_vz_u,\qquad 
p_7(v,w,u,t)=x_vx_t+y_vy_t+z_vz_t, \\
 p_8(v,w,u,t)&=x_wx_u+y_wy_u+z_wz_u,
\qquad p_9(v,w,u,t)=x_wx_t+y_wy_t+z_wz_t,\\
 p_{10}(v,w,u,t)&=x_ux_t+y_uy_t+z_uz_t,\\
p_{11}(v,w,u,t)&=\det(v,w,u)=x_vy_wz_u+x_wy_uz_v+x_uy_vz_w-z_vy_wx_u
-y_vx_wz_u-x_vz_wy_u,\\
p_{12}(v,w,u,t)&=\det(v,w,t)=x_vy_wz_t+x_wy_tz_v+x_ty_vz_w-z_vy_wx_t
-y_vx_wz_t-x_vz_wy_t,\\
p_{13}(v,w,u,t)&=\det(v,u,t)=x_vy_uz_t+x_uy_tz_v+x_ty_vz_u-z_vy_ux_t
-y_vx_uz_t-x_vz_uy_t,
\\
 p_{14}(v,w,u,t)&=\det(w,u,t)=x_wy_uz_t+x_uy_tz_w+x_ty_wz_u-z_wy_ux_t
-y_wx_uz_t-x_wz_uy_t.
\end{align*}
Since the lengths of the four vectors and their relative positions in space are completely determined by these $14$ polynomials,
we find, as expected, $(p_1,\ldots,p_{14})(\R^{12})
\simeq \R^{12}/\operatorname{SO}(3)$.
We compute the vector fields $X_1,\ldots, X_{14}$
associated to these polynomials. Since
\begin{align*}
\widetilde{p}_1(v,w)&=0,\qquad 
\widetilde{p}_2(v,w)=0,\qquad 
\widetilde{p}_3(v,w)=0, \qquad 
\widetilde{p}_4(v,w)=0,\\
\widetilde{p}_5(v,w)&=0,\qquad 
\widetilde{p}_6(v,w)=x_v\dr x_u+y_v\dr y_u+z_v\dr z_u,\qquad 
\widetilde{p}_7(v,w)=x_v\dr x_t+y_v\dr y_t+z_v\dr z_t, \\
\widetilde{p}_8(v,w)&=x_w\dr x_u+y_w\dr y_u+z_w\dr z_u,
\qquad 
\widetilde{p}_9(v,w)=x_w\dr x_t+y_w\dr y_t+z_w\dr z_t,\qquad 
\widetilde{p}_{10}(v,w)=0, \\
\widetilde{p}_{11}(v,w)&=(x_vy_w-y_vx_w)\dr z_u+
(x_wz_v-x_vz_w)\dr y_u+(y_vz_w-z_vy_w)\dr x_u,\\
v_{12}(v,w)&=(x_vy_w-y_vx_w)\dr z_t+(x_wz_v-x_vz_w)\dr y_t
+(y_vz_w-z_vy_w)\dr x_t,\\
\widetilde{p}_{13}(v,w)&=0,\qquad 
\widetilde{p}_{14}(v,w)=0,
\end{align*}
we find 
\begin{align*}
X_1&=X_2=X_3=X_4=X_5=X_{10}=X_{13}=X_{14}=0,\\
X_6(v,w)&=x_v\partial_{x_v}+y_v\partial_{y_v}+z_v\partial_{z_v}, 
\qquad X_7(v,w)=x_v\partial_{x_w}+y_v\partial_{y_w}+z_v\partial_{z_w},\\
X_{8}(v,w)&=x_w\partial_{x_v}+y_w\partial_{y_v}+z_w\partial_{z_v},
\qquad  X_{9}(v,w)=x_w\partial_{x_w}+y_w\partial_{y_w}+z_w\partial_{z_w},\\
X_{11}(v,w)&=(x_vy_w-y_vx_w)\partial_{z_v}+(x_wz_v-x_vz_w)\partial_{y_v}
+(y_vz_w-z_vy_w)\partial_{x_v},\\
 X_{12}(v,w)&=(x_vy_w-y_vx_w)\partial_{z_w}
+(x_wz_v-x_vz_w)\partial_{y_w}+(y_vz_w-z_vy_w)\partial_{x_w}.
\end{align*}
Thus, $\T_G$ is spanned by 
\begin{align*}
\left\{\begin{array}{c}
X_6(v,w)=x_v\partial_{x_v}+y_v\partial_{y_v}
+z_v\partial_{z_v}, 
\quad X_7(v,w)=x_v\partial_{x_w}+y_v\partial_{y_w}
+z_v\partial_{z_w},\\
X_{8}(v,w)=x_w\partial_{x_v}+y_w\partial_{y_v}
+z_w\partial_{z_v},\quad 
 X_{9}(v,w)=x_w\partial_{x_w}+y_w\partial_{y_w}
 +z_w\partial_{z_w},\\
X_{11}(v,w)=(x_vy_w-y_vx_w)\partial_{z_v}
+(x_wz_v-x_vz_w)\partial_{y_v}
+(y_vz_w-z_vy_w)\partial_{x_v},\\
 X_{12}(v,w)=(x_vy_w-y_vx_w)\partial_{z_w}
+(x_wz_v-x_vz_w)\partial_{y_w}+(y_vz_w-z_vy_w)\partial_{x_w}
\end{array}
\right\}.
\end{align*}
The vector field $Y$ defined by
\begin{align*}
Y(v,w)=&((v\times w)\times v)_x\partial_{ x_1}
+((v\times w)\times v)_y\partial_{ y_1}
+((v\times w)\times v)_z\partial_{ z_1}\\
&+((v\times w)\times w)_x\partial_{ x_2}
+((v\times w)\times w)_y\partial_{ y_2}
+((v\times w)\times w)_z\partial_{ z_2}\\
=&\langle v,v\rangle( x_2\partial_{ x_1}+y_2\partial_{y_1}+z_2\partial_{z_1})
-\langle v,w\rangle( x_1\partial_{ x_1}+y_1\partial_{y_1}+z_1\partial_{z_1})\\
&+\langle v,w\rangle( x_2\partial_{ x_2}+y_2\partial_{y_2}+z_2\partial_{z_2})
-\langle w,w\rangle( x_1\partial_{ x_2}+y_1\partial_{y_2}+z_1\partial_{z_2})\\
=&\langle v,v\rangle X_8(v,w)-\langle v,w\rangle X_6(v,w)
+\langle v,w\rangle X_9(v,w)-\langle w,w\rangle X_7(v,w)
\end{align*}
for all $(v,w)\in\R^3\times \R^3$
is then also an element of $\T_G$,
where $((v\times w)\times v)_x$, $((v\times w)\times v)_y$ and $((v\times
w)\times v)_z$
are the $x$-, $y$- and $z$-components of the vector product
$(v\times w)\times v$.
\end{example}

\end{appendix}

\def\cprime{$'$} \def\polhk#1{\setbox0=\hbox{#1}{\ooalign{\hidewidth
  \lower1.5ex\hbox{`}\hidewidth\crcr\unhbox0}}} \def\cprime{$'$}
  \def\cprime{$'$} \def\cprime{$'$} \def\cprime{$'$} \def\cprime{$'$}
  \def\cprime{$'$} \def\cprime{$'$}
  \def\polhk#1{\setbox0=\hbox{#1}{\ooalign{\hidewidth
  \lower1.5ex\hbox{`}\hidewidth\crcr\unhbox0}}}
  \def\polhk#1{\setbox0=\hbox{#1}{\ooalign{\hidewidth
  \lower1.5ex\hbox{`}\hidewidth\crcr\unhbox0}}}
  \def\polhk#1{\setbox0=\hbox{#1}{\ooalign{\hidewidth
  \lower1.5ex\hbox{`}\hidewidth\crcr\unhbox0}}}
  \def\polhk#1{\setbox0=\hbox{#1}{\ooalign{\hidewidth
  \lower1.5ex\hbox{`}\hidewidth\crcr\unhbox0}}} \def\cprime{$'$}
  \def\polhk#1{\setbox0=\hbox{#1}{\ooalign{\hidewidth
  \lower1.5ex\hbox{`}\hidewidth\crcr\unhbox0}}}
  \def\polhk#1{\setbox0=\hbox{#1}{\ooalign{\hidewidth
  \lower1.5ex\hbox{`}\hidewidth\crcr\unhbox0}}}
  \def\polhk#1{\setbox0=\hbox{#1}{\ooalign{\hidewidth
  \lower1.5ex\hbox{`}\hidewidth\crcr\unhbox0}}}
  \def\polhk#1{\setbox0=\hbox{#1}{\ooalign{\hidewidth
  \lower1.5ex\hbox{`}\hidewidth\crcr\unhbox0}}}

\bigskip

\noindent
\textbf{M. Jotz}\\
Section de Math{\'e}matiques\\ 
Ecole Polytechnique
  F{\'e}d{\'e}rale de Lausanne\\ 
CH-1015 Lausanne\\ Switzerland\\
\texttt{madeleine.jotz@epfl.ch}\\
Partially supported by Swiss NSF grant 200021-121512

\medskip

\noindent
\textbf{T.S. Ratiu}\\
Section de Math{\'e}matiques
 and Bernouilli Center\\ 
Ecole Polytechnique
  F{\'e}d{\'e}rale de Lausanne\\ 
CH-1015 Lausanne\\ Switzerland\\
\texttt{tudor.ratiu@epfl.ch}\\
Partially supported by Swiss NSF grant 200021-121512
\end{document}